\documentclass[11pt]{amsart}

\usepackage{amsmath, amsthm, amssymb}

\usepackage[dvipsnames]{xcolor}
\definecolor{darkgreen}{RGB}{0,150,0}
\definecolor{darkred}{RGB}{200,0,0}
\definecolor{darkblue}{RGB}{0,0,220}
\definecolor{gold}{RGB}{255,218,9}
\usepackage[colorlinks, linkcolor=darkblue, citecolor=darkgreen, urlcolor=darkgreen]{hyperref}

\usepackage{palatino}
\usepackage[T1]{fontenc}  

\usepackage{setspace}
\usepackage[margin=1.05in]{geometry}
             
\usepackage{comment}
\usepackage{thmtools, thm-restate}
\usepackage{overpic}
\usepackage{subcaption}
\usepackage{tikz}

\usepackage{aliascnt}
\usepackage[noabbrev,capitalize,nameinlink]{cleveref}

\newtheorem{theorem}{Theorem}[section]
\newtheorem{proposition}[theorem]{Proposition}
\newtheorem{lemma}[theorem]{Lemma}
\newtheorem{corollary}[theorem]{Corollary}
\newtheorem{conjecture}[theorem]{Conjecture}
\theoremstyle{definition}
\newtheorem{definition}[theorem]{Definition}
\newtheorem{example}[theorem]{Example}
\newtheorem{question}[theorem]{Question}
\theoremstyle{remark}
\newtheorem{remark}[theorem]{Remark}
\numberwithin{equation}{section}

\usepackage{mleftright}

\newcommand{\R}{\mathbb{R}}

\makeatletter 
 \def\l@subsection{\@tocline{2}{0pt}{2pc}{6pc}{}} \makeatother

\begin{document}

\title{Hard Legendrian unknots}

\author{Joseph Breen}
\address{University of Alabama, Tuscaloosa, AL 35401}
\email{jjbreen@ua.edu} \urladdr{https://sites.google.com/view/joseph-breen}

\author{Austin Christian}
\address{California Polytechnic State University, San Luis Obispo, CA 93407}
\email{achris66@calpoly.edu} \urladdr{https://sites.google.com/view/austin-christian}

\author{Angela Wu}
\address{Bucknell University, Lewisburg, PA 17837}
\email{ajw031@bucknell.edu} 
\urladdr{https://angelamath.com/}

\thanks{JB was partially supported by NSF Grant DMS-2038103 and an AMS-Simons Travel Grant. AC was partially supported by NSF Grant DMS-2532551 and an AMS-Simons Travel Grant.}

\begin{abstract}
    We initiate the study of Reidemeister hardness of Legendrian unknot front projections. Using normal rulings, we obstruct several infinite families of hard unknot diagrams from being drawn with max-tb unknot fronts, along with 1.7 million of the 2.6 million hard unknot diagrams studied in \cite{applebaum2024unknottingnumberhardunknot}. We construct infinitely many smoothly hard max-tb unknot diagrams, and bound their minimum possible writhe. With respect to these bounds, our constructions are conjecturally sharp.
\end{abstract}

\maketitle

\tableofcontents

\section{Introduction}

The study of hard diagrams of the unknot is a fixture of knot theory, dating back to the origins of the subject \cite{goeritz1934knot}. It has evolved throughout the years via cornerstone theoretical results \cite{hass2001reidemeister,hass2010quadratic,lackenby2015polynomial}, and continues today through modern techniques of computational analysis and reinforcement learning \cite{petronio2016unknots,applebaum2024unknottingnumberhardunknot,burton2024hard,lunel2024harddiagramssplitlinks}. 

More generally, the study of complexity measures for which the unknotting problem admits a monotonic solution has a rich history. In the 2010s, Henrich and Kauffman \cite{henrich2014unknotting} exhibited a quadratic upper bound on crossing number increase for hard unknot diagrams. They leveraged work of Dynnikov \cite{dynnikov2006arc} from the early 2000s, who showed that grid diagrams for the unknot can be simplified monotonically in grid dimension with a finite set of moves. In turn, Dynnikov was inspired by the work of Birman and Menasco \cite{birman1992unlink} in the 1990s and of Bennequin \cite{bennequin1983entrelacements} in the 1980s. The latter used contact topology and its imprint on Seifert surfaces to study braids, while the former proved a Markov-type theorem without stabilizations for braid representations of the unknot, marking the first complexity non-increasing unknotting theorem. 

There are various notions of Reidemeister hardness that have been studied in the literature, and we begin by establishing some clarifying language. Given a knot diagram $D$, let $\mathrm{cr}(D)$ denote its crossing number. If $\{D_i\}_{i=0}^n$ is a sequence of diagrams, we let $\overline{\mathrm{cr}}(\{D_i\}):= \sup_{i} \mathrm{cr}(D_i)$. By a \emph{planar} (resp.\ \emph{spherical}) \emph{unknotting Reidemeister sequence} we mean a sequence of diagrams $\{D_i\}_{i=0}^n$ such that $D_n$ is the trivial diagram of the unknot and, up to planar (resp.\ spherical) isotopy, $D_{i+1}$ is obtained from $D_i$ by a Reidemeister move (see \cref{fig:reid}). 

\begin{figure}[ht]
	\centering
    
	\begin{overpic}[scale=.3]{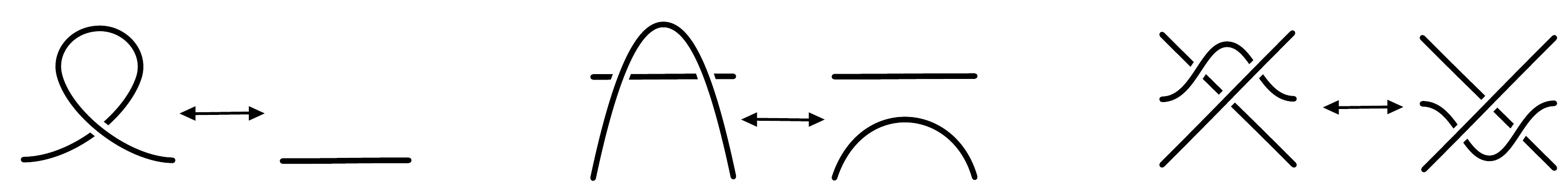}
      \put(12.5,9){\small RI}
      \put(48,9){\small RII}
      \put(85,9){\small RIII}
	\end{overpic}
	\caption{Reidemeister moves, up to rotations and mirroring.}
	\label{fig:reid}
\end{figure}

\begin{definition}\label{def:hardness}
Let $D$ be a diagram of the unknot.  
\begin{enumerate}
    \item We say that $D$ is \emph{weakly hard} if, in any unknotting Reidemeister sequence $\{D_i\}$ with $D_0 = D$, there exists an $i$ such that $\mathrm{cr}(D_i) < \mathrm{cr}(D_{i+1})$. 

    \item We say that $D$ is \emph{strongly hard} if any unknotting Reidemeister sequence $\{D_i\}$ with $D_0 = D$ satisfies $\overline{\mathrm{cr}}(\{D_i\}) > \mathrm{cr}(D)$.
\end{enumerate}
When the distinction is important, we include the adjective \emph{planar} or \emph{spherically} to indicate we are considering unknotting Reidemeister sequences on the plane or sphere. If a distinction between weak and strong hardness is unimportant, we will simply say \emph{hard}.
\end{definition}

In other words, weak hardness means that there is no unknotting Reidemeister sequence which is monotonically non-increasing in crossing number, while strong hardness means that the crossing number must increase before it can decrease. Note that strong hardness implies weak hardness, and that one may produce a weakly but not strongly hard diagram by performing crossing-increasing Reidemeister moves on a strongly hard diagram. 

The earliest example of a hard unknot diagram, drawn in \cref{fig:goeritz}, was discovered by Goeritz in 1934 \cite{goeritz1934knot}. By inspecting each region of the diagram, including the compactified outer region, one can verify that there are no available Reidemeister moves which preserve or decrease the crossing number. Taking for granted that the Goeritz knot is in fact the unknot, it follows that the diagram is strongly spherically hard. Over time, many examples have been discovered and their complexities studied; see \cite{burton2024hard} for a recent survey and \cite{applebaum2024unknottingnumberhardunknot} for a dataset with 2.6 million diagrams. 

\begin{figure}[ht]
	\centering
	\begin{overpic}[scale=.21]{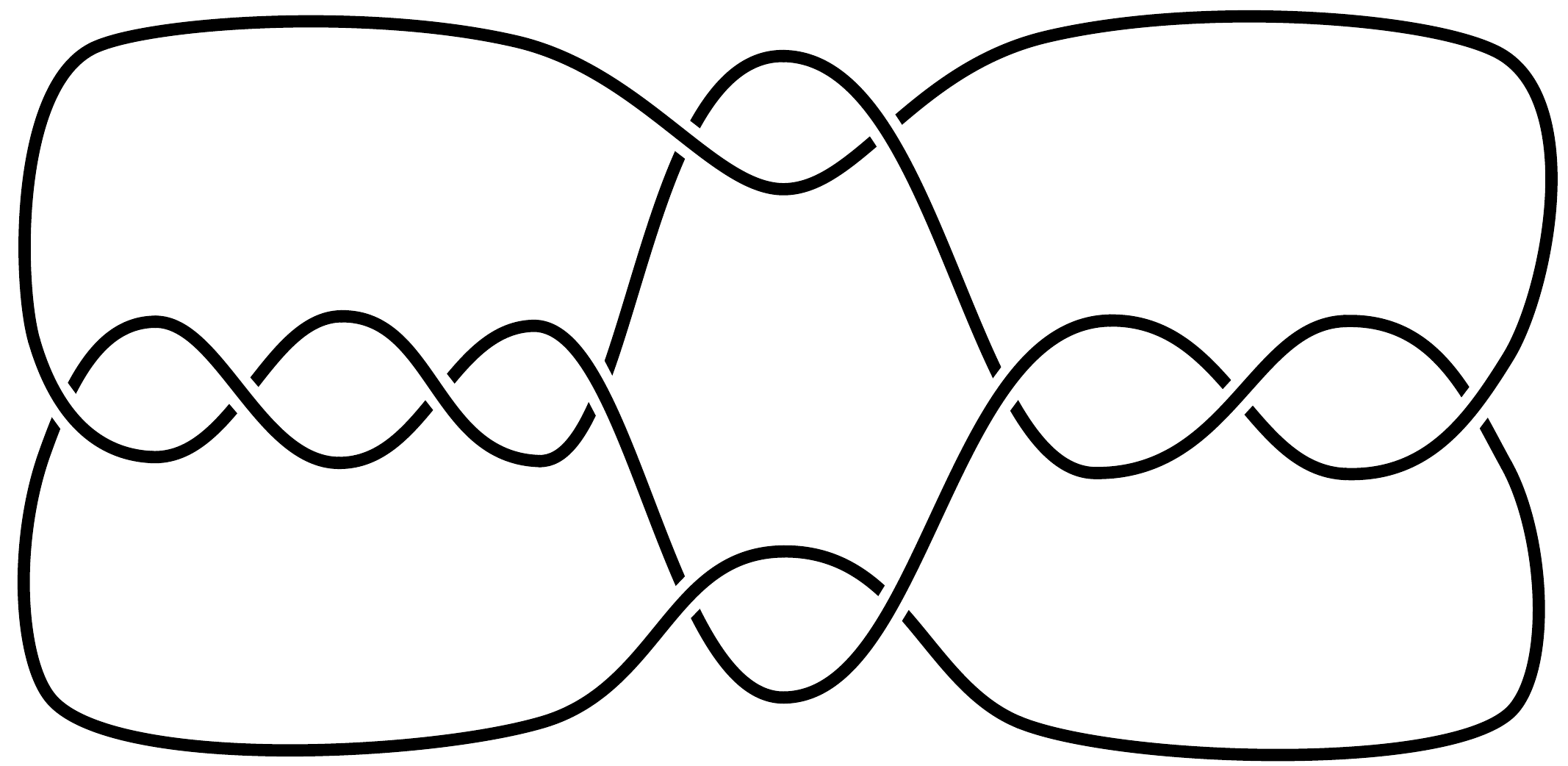}
 
	\end{overpic}
	\caption{The Goeritz unknot.}
	\label{fig:goeritz}
\end{figure}

\subsection{Legendrian knots}\label{subsec:leg_knots}

Consider $\R^3$ with the contact structure $\xi = \ker (dz - y\, dx)$. A knot $\Lambda \subset \R^3$ is \emph{Legendrian} if it is everywhere tangent to $\xi$. Legendrian knot theory is central to modern contact and symplectic topology, leading to rich theories of surgery \cite{gompf1998handlebody,ding2001symplectic,casals2024steintrace}, classification \cite{etnyre2001knots,chekanov2002dga,chongchitmate2013atlas}, cobordism \cite{chantraine2010concordance,ekholm2012exactcobordisms,bourgeois2015cobordisms,casals2022infinitely}, and more \cite{etnyre2005surveyknots}.

A Legendrian knot is completely determined by its \emph{front projection} $\mathcal{F}$ to the $(x,z)$-plane, as the $y$-coordinate can be recovered from $y=\tfrac{dz}{dx}$. Generic front projections have semicubical cusps in place of vertical tangencies, and strands with more negative $(x,z)$-slope always cross over strands with more positive $(x,z)$-slope. Any two front projections of the same Legendrian knot are related by a set of \emph{Legendrian Reidemeister moves}; see \cref{fig:reid_leg}. 

\begin{figure}[ht]
	\centering
	\begin{overpic}[scale=.3]{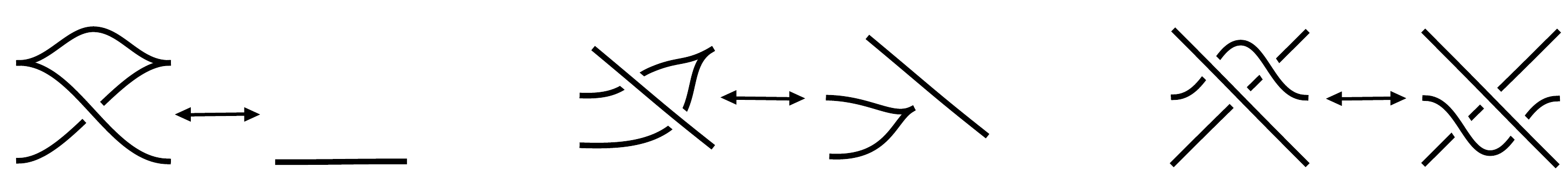}
\put(12.5,9){\small RI}
      \put(47,9){\small RII}
      \put(85,9){\small RIII}
	\end{overpic}
	\caption{Legendrian Reidemeister moves, up to natural symmetries.}
	\label{fig:reid_leg}
\end{figure}

A Legendrian knot $\Lambda$ has three classical invariants: its topological knot type, its \emph{Thurston-Bennequin invariant} $\mathrm{tb}(\Lambda)$, and its \emph{rotation number} $\mathrm{rot}(\Lambda)$. After the knot type, the former is the integral framing induced by the contact structure, and can be computed combinatorially from the front projection $\mathcal{F}$ as follows: 
\begin{align*}
\mathrm{tb}(\Lambda) &= \mathrm{writhe}(\mathcal{F}) - \frac{1}{2}(\#\text{ of cusps})\\ 
&= (\#\text{ of pos. crossings}) - (\#\text{ of neg. crossings}) - \frac{1}{2}(\#\text{ of cusps}).     
\end{align*}
Every knot type $K$ has a maximum Thurston-Bennequin invariant $\overline{\mathrm{tb}}(K)$, a consequence of the Bennequin inequality $\mathrm{tb}(\Lambda) + |\mathrm{rot}(\Lambda)| \leq - \chi(\Sigma)$ where $\Sigma$ is any Seifert surface \cite{bennequin1983entrelacements,eliashberg1992contact}. Consequently, the maximum Thurston-Bennequin invariant of the unknot is $-1$, witnessed by the front projection with two cusps and no crossings. 

\begin{figure}[ht]
	\centering
    \vskip-0.5cm
	\begin{overpic}[scale=.25]{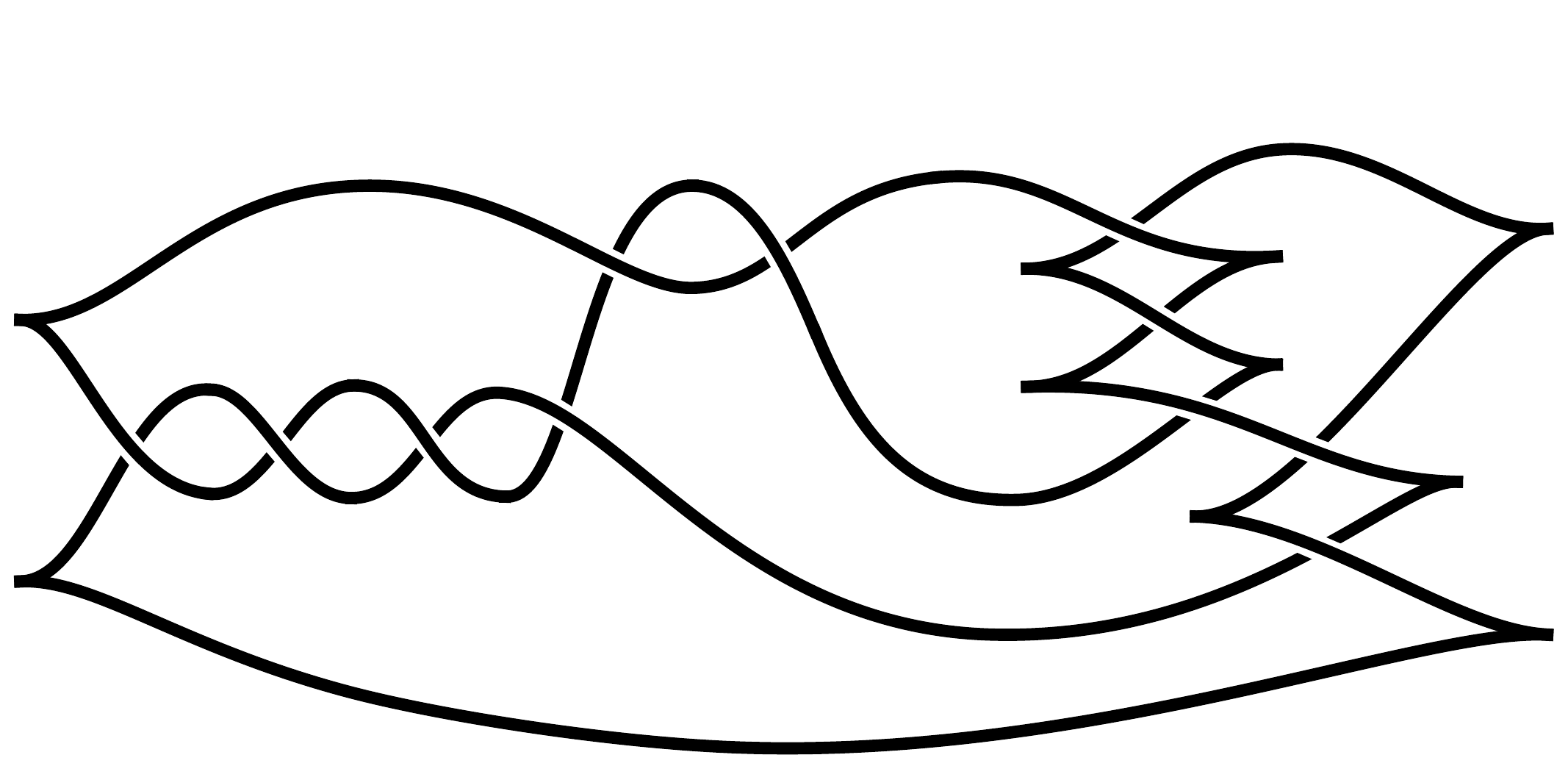}
 
	\end{overpic}
	\caption{A Legendrian realization of the Goeritz unknot with $\mathrm{tb} = -4$.}
	\label{fig:goeritz_leg}
\end{figure}

\subsection{Hard Legendrian unknots}\label{subsec:hard_leg_un}
Given a smooth knot diagram $D$, it is straightforward to find a Legendrian representative of the knot with a front projection planar isotopic to $D$. For example, after a planar isotopy of $D$ we may assume that all overcrossings have negative slope, all undercrossings have positive slope, and all vertical tangencies are nondegenerate. Turning every vertical tangency into a cusp produces a front projection for a Legendrian representative $\Lambda$ of the knot type. \cref{fig:goeritz_leg} gives a Legendrian representative of the Goeritz unknot. Another example is the Legendrian representative of Adams' ``nasty unknot" \cite[Figure 1.29]{adams1994knot} given in \cref{fig:nasty}. 

\begin{figure}[ht]
	\centering
    \vskip-.7cm
	\begin{overpic}[scale=.24]{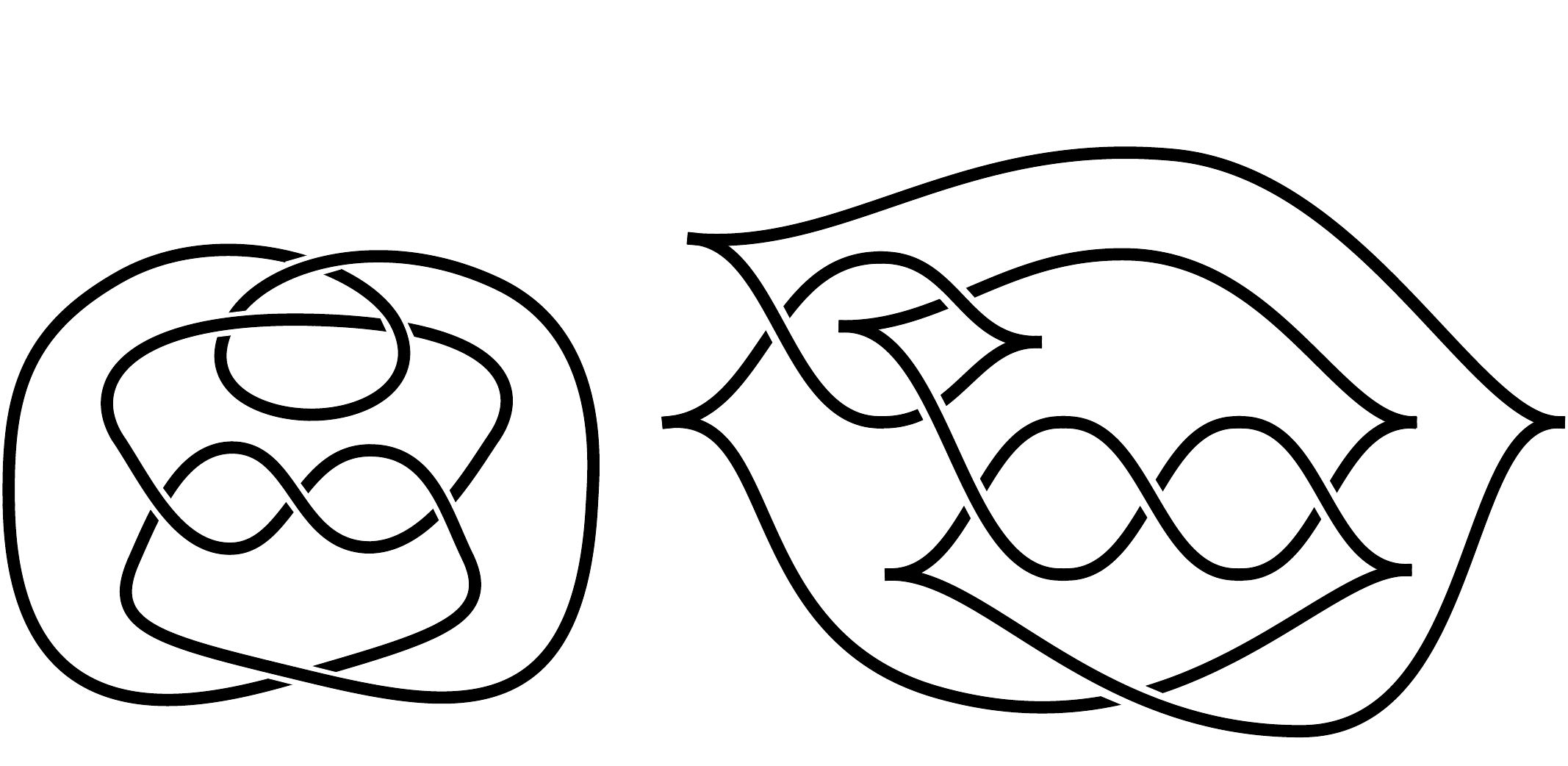}
 
	\end{overpic}
	\caption{Adams' nasty unknot on the left, and a Legendrian realization with $\mathrm{tb} = -3$ on the right. (The smooth diagram is strongly hard on the plane, but not even weakly hard on the sphere.)}
	\label{fig:nasty}
\end{figure}

Neither of these Legendrian unknots achieve the maximum Thurston-Bennequin invariant of $-1$: the Legendrian Goeritz unknot has $\mathrm{tb} = -4$, and the Legendrian nasty unknot has $\mathrm{tb} = -3$. In general, as the reader may freely experiment with and confirm, approximating known hard diagrams (c.f. \cite{burton2024hard}) by fronts with near-maximal Thurston-Bennequin invariant is difficult. The question of hardness as it relates to the rigidity of a Legendrian knot imposed by the contact structure therefore appears interesting. In particular, it is natural to ask whether some notion of unknot hardness can be achieved with $\mathrm{tb} = -1$. 

\begin{definition}[Legendrian hardness]
By a \emph{Legendrian unknotting Reidemeister sequence} we mean an unknotting Reidemeister sequence of front projections $\{\mathcal{F}_i\}_{i=0}^n$ where, up to planar isotopy of fronts, $\mathcal{F}_{i+1}$ is obtained from $\mathcal{F}_{i}$ via a Legendrian Reidemeister move, for $1\leq i\leq n$. We say that a front projection $\mathcal{F}$ of a Legendrian unknot is
\begin{enumerate}
    \item \emph{weakly Legendrian hard} if in any Legendrian unknotting Reidemeister sequence $\{\mathcal{F}_i\}$ with $\mathcal{F}_0 = \mathcal{F}$, there exists an $i$ such that $\mathrm{cr}(\mathcal{F}_i) < \mathrm{cr}(\mathcal{F}_{i+1})$. 

    \item \emph{strongly Legendrian hard} if any Legendrian unknotting Reidemeister sequence $\{\mathcal{F}_i\}$ with $\mathcal{F}_0 = \mathcal{F}$ satisfies $\overline{\mathrm{cr}}(\{\mathcal{F}_i\}) > \mathrm{cr}(\mathcal{F})$.
\end{enumerate}
When the distinction is unimportant, we simply say \emph{Legendrian hard}.
\end{definition}

\begin{definition}[Smooth hardness]
We say that a front projection $\mathcal{F}$ of a Legendrian unknot is \emph{smoothly $\ast$-hard} (or \emph{smoothly hard} when a distinction is unimportant) if its \emph{smooth resolution}, the diagram $D$ obtained by smoothing all cusps, is $\ast$-hard. Here ``$\ast$-hard" stands for any notion of hardness in \cref{def:hardness}.
\end{definition}

Smooth hardness is strictly stronger than Legendrian hardness, as smoothed Legendrian Reidemeister moves comprise a strict subset of smooth Reidemeister moves. A max-tb example witnessing the difference is given in \cref{fig:hard}. 

\begin{figure}[ht]
	\centering
	\begin{overpic}[scale=.25]{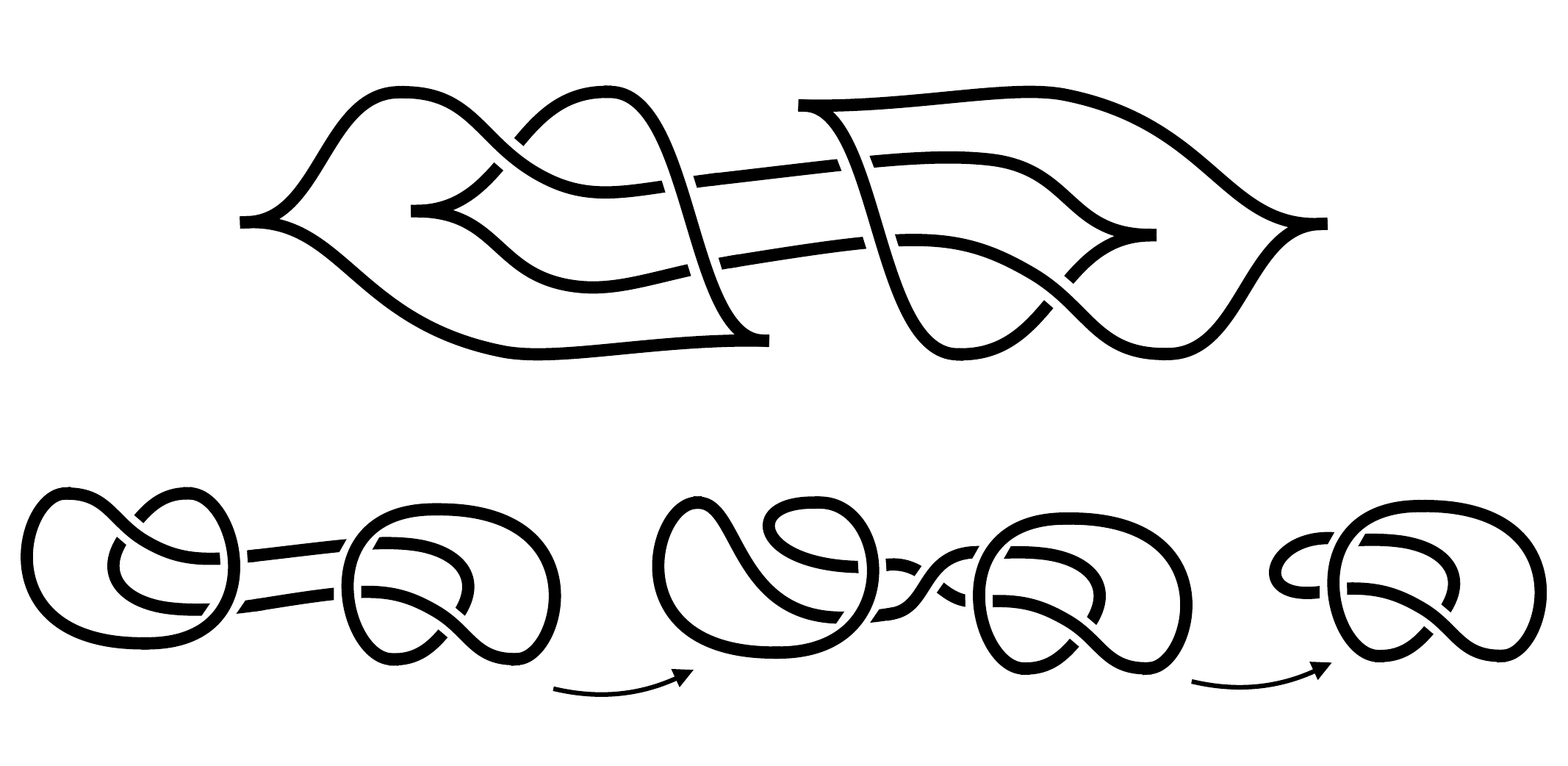}
 
	\end{overpic}
	\caption{A Legendrian hard max-tb unknot which is not smoothly hard. The second row depicts the smooth resolution and a sequence of RIII and RI moves; the resulting diagram clearly monotonically simplifies.}
	\label{fig:hard}
\end{figure}

We arrive at the main question of this article. Despite the historical trajectory from Bennequin's seminal work in contact topology, to Dynnikov's (naturally Legendrian) grid complexity, to the study of hard unknots, it appears to not have been asked before. 

\begin{question}\label{question:main}
Is there a smoothly hard front projection of the max-tb unknot?     
\end{question}

Beyond experimental intrigue, investigating \cref{question:main} is attractive from the point of view of the unknotting problem. Max-tb unknot fronts form a well-defined class of unknot diagrams which are geometrically motivated and have additional rigidity; one might speculate that there are improved polynomial bounds on the number of unknotting Reidemeister moves for such a class of diagrams (c.f.\ \cite{lackenby2015polynomial}).

On the other hand, hardness is also well-motivated from a contact topological point of view. For example, a Legendrian knot is \emph{Lagrangian slice} if it bounds a Lagrangian disk in the standard symplectic $4$-ball \cite{chantraine2010concordance,cornwell2016concordance}. One may further ask for such a Lagrangian disk to be \emph{decomposable}, a combinatorial condition on the front projection \cite{ekholm2012exactcobordisms}, or \emph{regular}, a type of geometric compatibility with a Weinstein structure on the $4$-ball \cite{eliashberg2018flexiblelagrangians}. While decomposable implies regular \cite{ConwayEtnyreTosun2021Disks}, it is an open question as to whether the converse holds in this setting \cite{Breen2024Regularly}. However, by work of Conway, Etnyre, and Tosun \cite{ConwayEtnyreTosun2021Disks}, a Legendrian knot is regularly Lagrangian slice if an only if it can be presented as a max-tb unknot in the boundary of a Weinstein Kirby diagram for the $4$-ball. By understanding hardness of max-tb unknots, one might hope to better understand the relationship between regularity and decomposability.

\subsection{Main results}\label{subsec:main_results}

We first formalize the experimental observation that it is difficult to draw hard unknot diagrams with max-tb fronts. Using the theory of normal rulings, we establish the following obstructions. 

\begin{theorem}\label{thm:obstructions}
Let $\mathcal{F}$ be a front projection of the max-tb unknot. 
\begin{enumerate}
    \item If $\mathrm{writhe}(\mathcal{F})\leq 1$, then $\mathcal{F}$ is not weakly Legendrian hard. 
    \item The smooth resolution of $\mathcal{F}$ cannot contain any of the oriented tangles, up to rotation and reversal of total orientation, depicted in \cref{fig:bad-tangles}.
\end{enumerate}
In particular, if $D$ is a hard diagram of the unknot and either $\mathrm{writhe}(D) > 1$ or $D$ contains one of the tangles in \cref{fig:bad-tangles}, then $D$ cannot be drawn as a max-tb unknot front. 
\end{theorem}

\begin{figure}[ht]
	\centering
	\begin{overpic}[scale=.6]{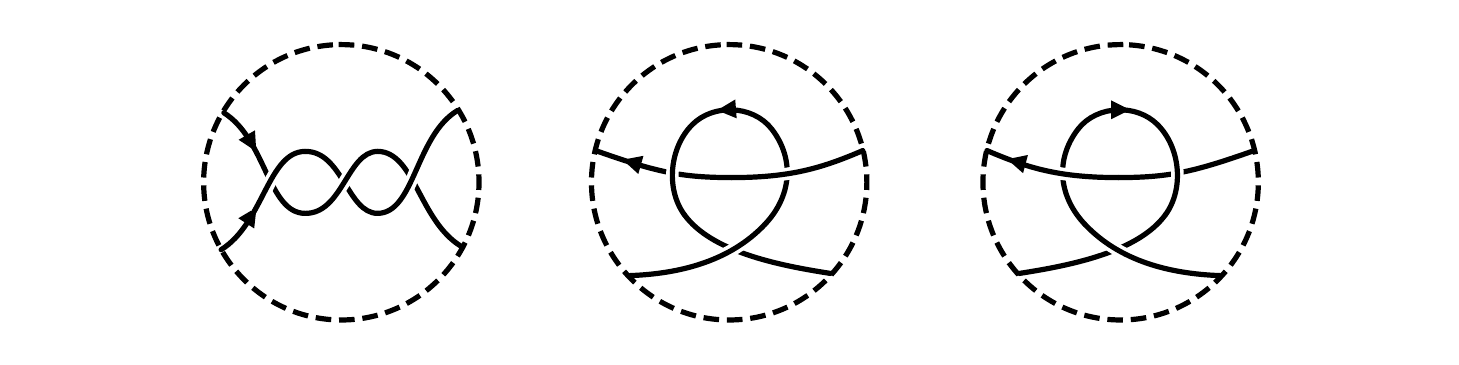}
     \put(21,0){(A)}
     \put(48.5,0){(B)}
     \put(75,0){(C)}
	\end{overpic}
	\caption{Tangles prohibited in a max-tb unknot front.}
	\label{fig:bad-tangles}
\end{figure}

There are several known recipes for constructing complicated diagrams of the unknot. \cref{thm:obstructions} prevents many of these from being max-tb unknots. 

\begin{corollary}\label{cor:families}
The following families of unknots cannot be witnessed by a max-tb front:
\begin{enumerate}
    \item Hard generalized Freedman-He-Wang unknots \cite[\S 4.1]{burton2024hard}.
    \item Generalized Goeritz unknots \cite[\S 4.2]{burton2024hard}.
    \item Hass-Nowik unknots \cite[Figure 2]{hass2010quadratic}.
\end{enumerate}    
\end{corollary}

\noindent Using additional obstructions and more sophisticated arguments, we can obstruct a family of hard unknots due to Kauffman and Lambropoulou. 

\begin{theorem}\label{thm:KL}
The hard unknots of Kauffman-Lambropoulou \cite{kauffman2011collapsing} constructed by rational tangles cannot be witnessed by a max-tb front.
\end{theorem}

Beyond the families above, there are $21$ diagrams of the unknot highlighted in the survey \cite[Appendix A]{burton2024hard}. Not all of them are hard in the sense of \cref{def:hardness}, but they are a representative sampling of well-known complicated diagrams. With \cref{thm:obstructions} and additional ad hoc arguments using normal rulings, we can obstruct almost all of them from being max-tb unknots. 

\begin{theorem}\label{thm:burton}
Of the 21 unknot diagrams (up to mirroring) surveyed in \cite[Appendix A]{burton2024hard}, at least 20 of them cannot be witnessed by a max-tb front.    
\end{theorem}

\noindent The one diagram we have not obstructed is Haken's infamous Gordian unknot (see \cite[Figure 25]{burton2024hard}, \cite{petronio2016unknots}, \cite{stewart2009professor}, or the discussion in \cite{gowers2011hard}) with $141$ crossings and writhe $37$. See \Cref{table1} for an accounting of the diagrams considered by \Cref{thm:burton}.

The recent article \cite{applebaum2024unknottingnumberhardunknot} applied reinforcement learning to generate and study 2.6 million hard unknot diagrams. Using  \cref{thm:obstructions} and a Python script which computes writhes and searches for prohibited tangles via PD codes, we can obstruct about $67\%$ of them from being max-tb fronts. 

\begin{corollary}\label{cor:data}
Of the 2,623,588 hard diagrams (up to mirroring) of the unknot studied in \cite{applebaum2024unknottingnumberhardunknot}, at least 1,770,489 cannot be witnessed by a max-tb front. 
\end{corollary}

\begin{remark}
Below (\cref{conj:main}) we state a conjectural sharp writhe obstruction that a smoothly spherically hard max-tb unknot front must have writhe $\geq 4$. This, together with the prohibited tangles, would obstruct 2,227,501 (about $85\%$) of the 2,623,588 hard diagrams (up to mirroring) from being witnessed by a max-tb front.
\end{remark}

These obstructions confirm that it is difficult, and in many cases impossible using known examples, to draw a complicated unknot diagram with a max-tb front. Nevertheless, we construct infinitely many smoothly hard max-tb unknots. To state our main result in an even stronger form, note that if $\mathcal{F}$ is a front projection of a max-tb unknot, then $\mathrm{writhe}(\mathcal{F}) = -1 + \frac{1}{2}(\# \text{ of cusps}) \geq 0$. Let $\mathrm{minwr}_{\R^2}$ and $\mathrm{minwr}_{S^2}$ denote the minimum writhe of a max-tb unknot front projection which is smoothly hard on the plane and sphere, respectively. \cref{question:main} may then be rephrased as asking whether either of these quantities is finite, and then further refined by asking for the minimum value. 

\begin{theorem}\label{thm:main}
Smoothly hard Legendrian unknots exist on both the plane and sphere, and 
\begin{align*}
    2 \leq \mathrm{minwr}_{\R^2} \leq 3, \\
    2\leq \mathrm{minwr}_{S^2} \leq 4.
\end{align*}
\end{theorem}

An example of a smoothly (strongly, spherically) hard max-tb unknot of writhe $4$ is given in \cref{fig:writhe4hard}. In \cref{sec:construct} we verify its smooth hardness and give other constructions of infinitely many smoothly hard max-tb unknots, establishing the upper bounds of the theorem. 

\begin{figure}[ht]
	\centering
	\begin{overpic}[scale=.3]{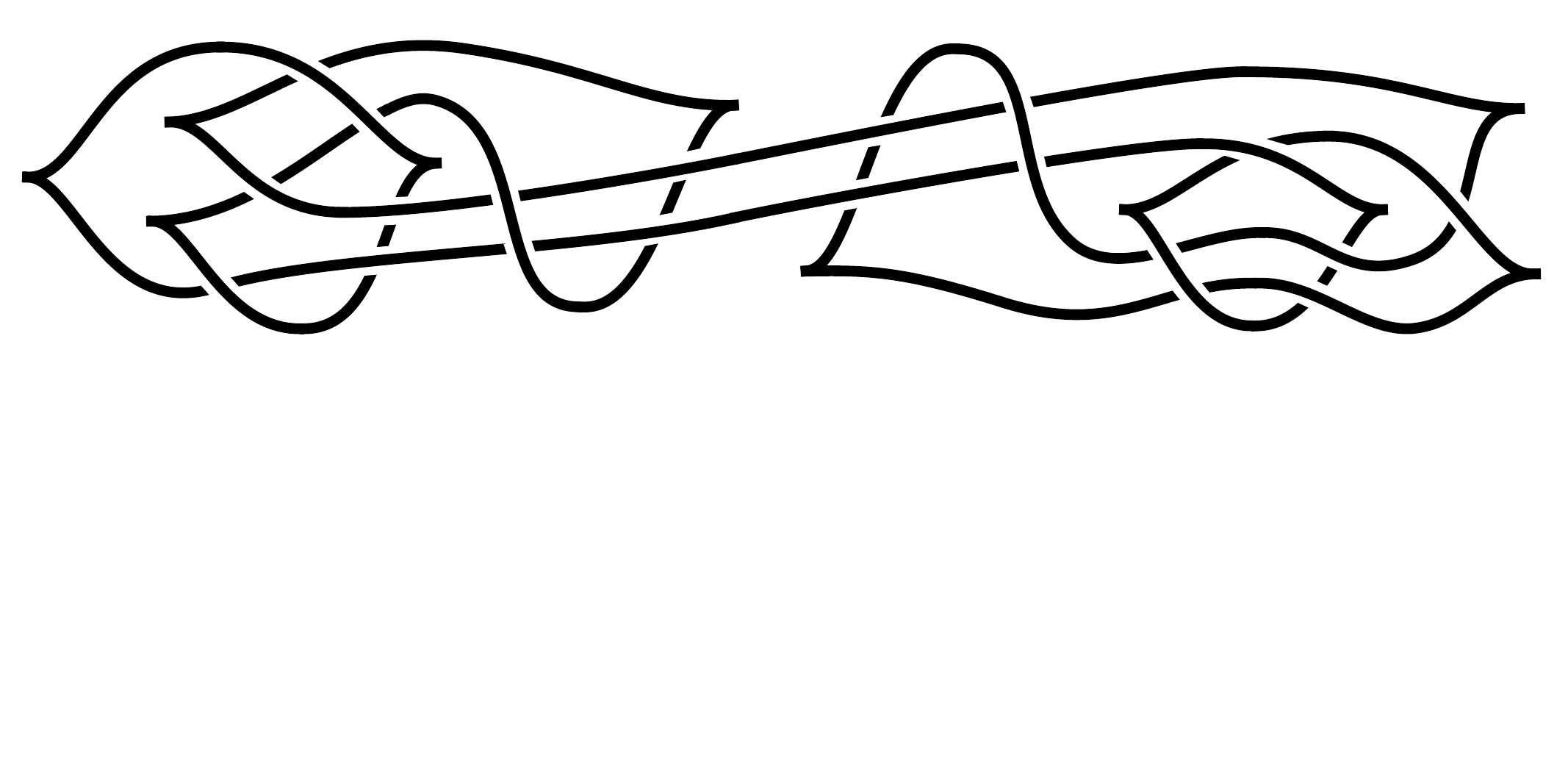}
 
	\end{overpic}
    \vskip-3cm
	\caption{A smoothly strongly hard max-tb unknot of writhe $4$ on the sphere.}
	\label{fig:writhe4hard}
\end{figure}

Our work suggests the following sharp version of \cref{thm:main}. 

\begin{conjecture}\label{conj:main}
For smoothed max-tb Legendrian unknots,
\begin{align*}
    \mathrm{minwr}_{\R^2} = 3, \\
    \mathrm{minwr}_{S^2} = 4.
\end{align*}
\end{conjecture}

\noindent To prove \cref{conj:main}, one must increase the lower bounds in \cref{thm:main}. As alluded to above, our lower bounds are obtained by appealing to the existence and uniqueness of \emph{normal rulings} of a max-tb unknot front, a decomposition of the front into certain planar disks first considered by Eliashberg \cite{eliashberg1987wave} and formalized by Fuchs \cite{fuchs2003rulings} and Chekanov and Pushkar \cite{chekanov2007pushkar}. The \emph{ruling polynomial} is a Legendrian isotopy invariant which collects graded counts of normal rulings, and is closely related to the Chekanov-Eliashberg DGA \cite{fuchs2003rulings,fuchs2004invariants,sabloff2005augmentations,leverson2014augmentations}.

Attempting to use this strategy to prove that $\mathrm{minwr}_{\R^2}>2$ leads one to consider writhe-$2$ diagrams which admit unique normal rulings. Unlike the writhe-$1$ case, there are many such fronts representing nontrivial knot types.

\begin{definition}\label{def:impostor}
An \emph{impostor} is a front projection $\mathcal{F}$ which satisfies the following properties:  
\begin{enumerate}
    \item The knot represented by $\mathcal{F}$ is not the unknot. 
    \item The front $\mathcal{F}$ admits a unique normal ruling. 
    \item Up to smooth RIII moves, there are no available crossing-decreasing smooth Reidemeister moves.
\end{enumerate}
\end{definition}

\begin{theorem}\label{thm:impostor}
There are infinitely many writhe-$2$ impostors.   
\end{theorem}

In other words, a writhe-$2$ impostor is a front which is immune to the normal ruling proof strategy used for obtaining the lower bounds in \cref{thm:main}. A proof of \cref{conj:main} will therefore need to appeal to other topological features of the diagram, which is a subtle problem to overcome. 

\subsection{Further questions}\label{subsec:questions}

Beyond \cref{conj:main}, we include some other questions of interest. First, as referenced above, normal ruling invariants may be extracted from the Chekanov-Eliash\-berg DGA, a deep and widely used tool in Legendrian knot theory. It is known that there are nontrivial Legendrian knots whose DGA is stable-tame isomorphic to that of the max-tb unknot (see \cite[p. 821]{cornwell2016concordance} or \cite[Appendix A]{etnyre2020legendrian}), which in particular implies that the ruling invariants are identical to that of the unknot. Referencing \cref{def:impostor}, a \emph{strong impostor} is a front satisfying (1) and (3), and in place of (2), the requirement that the DGA is stable-tame isomorphic to that of the max-tb unknot. Our impostors constructed in \cref{thm:impostor} are smoothly isotopic to a subset of the pretzel knots referenced by \cite{cornwell2016concordance} and \cite{etnyre2020legendrian}. We do not know if they are Legendrian isotopic, nor have we directly computed the stable-tame isomorphism class of the DGA. Thus: 

\begin{question}
Are there infinitely many writhe-$2$ strong impostors?     
\end{question}

Finally, recall that a knot diagram $D$ is \emph{composite} if there is an embedded planar circle $C$ such that $|D \cap C| = 2$, and there are crossings of $D$ both inside and outside of $C$. A diagram is \emph{prime} otherwise. All of our constructions of smoothly hard Legendrian unknots are composite, and crucially so. 

\begin{question}
Is there a smoothly hard prime front projection of the max-tb unknot?     
\end{question}

\subsection{Organization} 

In \cref{sec:obstruct} we review background on normal rulings of Legendrian knots and use them to prove the various obstruction theorems and corollaries stated above. In \cref{sec:construct}, we construct hard unknots, obtaining the upper bounds of \cref{thm:main}. Finally, in \cref{sec:impostors} we construct writhe-$2$ impostors to prove \cref{thm:impostor} and discuss their implications and complications toward a complete proof of \cref{conj:main}.

\subsection{Acknowledgments} The authors thank Josh Sabloff for interest in our work, and James Hughes for feedback, corrections, and discussion on a preliminary draft. 

\section{Obstructing hardness}\label{sec:obstruct}

We recall the definition of a normal ruling of a Legendrian knot. Given a front projection $\mathcal{F}$, assume that the $x$-coordinates of all crossings and cusps are distinct. A \emph{normal ruling} of $\mathcal{F}$ is a subset of the crossings, called \emph{switches}, or \emph{switched crossings}, such that: 
\begin{enumerate}
    \item Performing $0$-resolutions of all switches yields a link of $2$-cusp max-tb unknots. We refer to the planar disks bounded by the unknots as \emph{ruling disks}. 

    \item Each strand of the $0$-resolution near a switch belongs to a different ruling disk.

    \item If $x_0$ is the $x$-coordinate of a switch, then the ruling disks adjacent to the switch are either disjoint or nested in an arbitrarily small interval $(x_0-\epsilon, x_0+\epsilon)$; see \cref{fig:rulingdef}. This condition is referred to as \emph{normality}.
\end{enumerate}

\begin{figure}[ht]
	\begin{overpic}[scale=.4]{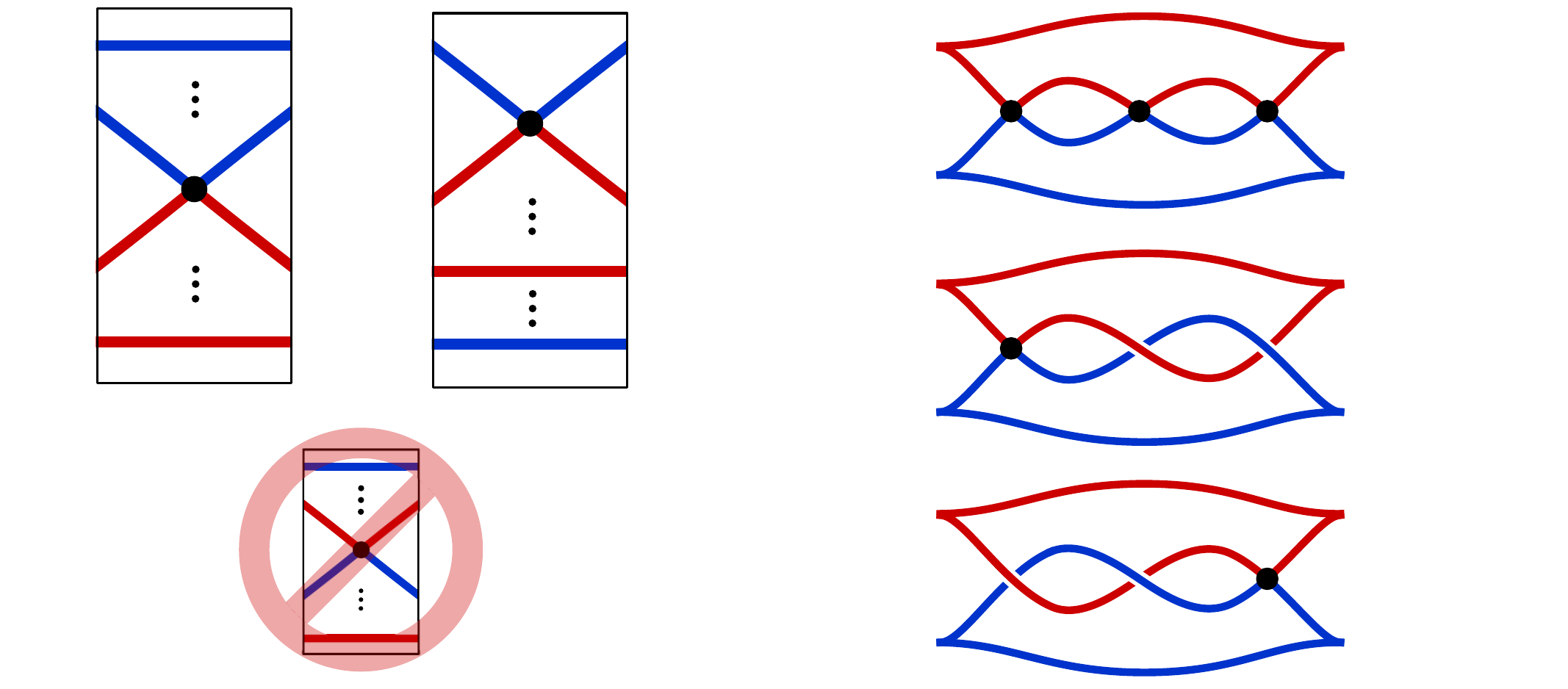}   
    
	\end{overpic}
	\caption{The normality condition near switched crossings on the left, and the three normal rulings of the trefoil on the right.}
	\label{fig:rulingdef}
\end{figure}

Not every Legendrian knot admits a normal ruling. By \cite{fuchs2003rulings,fuchs2004invariants,sabloff2005augmentations}, existence of a normal ruling of a knot is equivalent to existence of an augmentation of the Chekanov-Eliashberg DGA. In particular, a Legendrian knot admitting a ruling has maximal Thurston-Bennequin invariant in its topological type. 

Specifying to normal rulings of the max-tb unknot, observe that the standard $2$-cusp projection admits a unique (trivial) normal ruling. By \cite{chekanov2007pushkar}, counts of normal rulings are a Legendrian isotopy invariant, hence any front projection $\mathcal{F}$ of the max-tb unknot admits a unique normal ruling. One can moreover show that this unique ruling is \emph{orientable} in that every switch is a positive crossing. 

The combinatorial formula for the Thurston-Bennequin number gives 
\begin{align*}
    -1 &= \mathrm{writhe}(\mathcal{F}) - \tfrac{1}{2}(\# \text{ of cusps}) \\
    &= \mathrm{writhe}(\mathcal{F}) - n
\end{align*}
where $n$ is the number of ruling disks. Hence, $\mathrm{writhe}(\mathcal{F}) = n - 1$. To $\mathcal{F}$ we may associate a graph $G$ with vertices for each ruling disk and edges for each switch adjacent to two eyes. By \cite[\S 6]{chekanov2007pushkar}, the Euler characteristic of $G$ is invariant under Legendrian isotopy. From the standard $2$-cusp unknot front we have $\chi(G) = 1$, i.e.\ $n - (\# \text{ of switches}) = 1$. Combining this observation with the above calculation, we see
\[
\mathrm{writhe}(\mathcal{F}) = (\# \text{ of switches}).
\]
Summarizing the important facts: 

\begin{lemma}
Any front projection of the max-tb unknot admits a unique normal ruling and $\mathrm{writhe}(\mathcal{F})$ is given by the number of switches. Moreover, each switch is a positive crossing.
\end{lemma}

This is used in the proof of the following proposition, which establishes the writhe obstruction in \cref{thm:obstructions} and the lower bounds in \cref{thm:main}.

\begin{proposition}\label{prop:writhe1obstruct}
Let $\mathcal{F}$ be a front projection of a max-tb unknot with $\mathrm{writhe}(\mathcal{F}) \leq 1$. Then $\mathcal{F}$ admits either a simplifying Legendrian RI or a simplifying Legendrian RII move. In particular, $\mathcal{F}$ is not smoothly hard on the plane or the sphere.     
\end{proposition}

\begin{proof}
The only max-tb unknot front with $\mathrm{writhe}(\mathcal{F}) = 0$ is the standard two-cusp projection, so assume that $\mathrm{writhe}(\mathcal{F}) = 1$. After a planar isotopy we may further assume that all crossings and cusps have distinct $x$-coordinates, hence $\mathcal{F}$ admits a unique normal ruling. As $\mathrm{writhe}(\mathcal{F}) = 1$, this unique ruling has one switched crossing at an $x$-coordinate we denote $x_{\mathrm{sw}}$. Decompose the front as 
\[
\mathcal{F} = \mathcal{F}_{\mathrm{LC}} \cup \mathcal{F}_{\mathrm{sw}} \cup \mathcal{F}_{\mathrm{RC}}
\]
where $\mathcal{F}_{\mathrm{sw}} := \mathcal{F}\cap \{|x-x_{\mathrm{sw}}|< \epsilon\}$, $\mathcal{F}_{\mathrm{LC}} = \mathcal{F} \cap \{x\leq x_{\mathrm{sw}} - \epsilon\}$, and $\mathcal{F}_{\mathrm{RC}} = \mathcal{F} \cap \{x\geq x_{\mathrm{sw}} + \epsilon\}$. Here $\epsilon > 0$ is chosen sufficiently small enough so that the switched crossing is the only crossing in $\mathcal{F}_{\mathrm{sw}}$. In this case, $\mathcal{F}_{\mathrm{sw}}$ is either a \emph{nested} or \emph{non-nested} switch region (see \cref{fig:Obs1}) and the left and right cusp regions each have two cusps. If neither cusp region contains crossings, then $\mathcal{F}$ admits a simplifying Legendrian RI move, so we may assume by reflective symmetry that the left cusp region $\mathcal{F}_{\mathrm{LC}}$ contains crossings. 

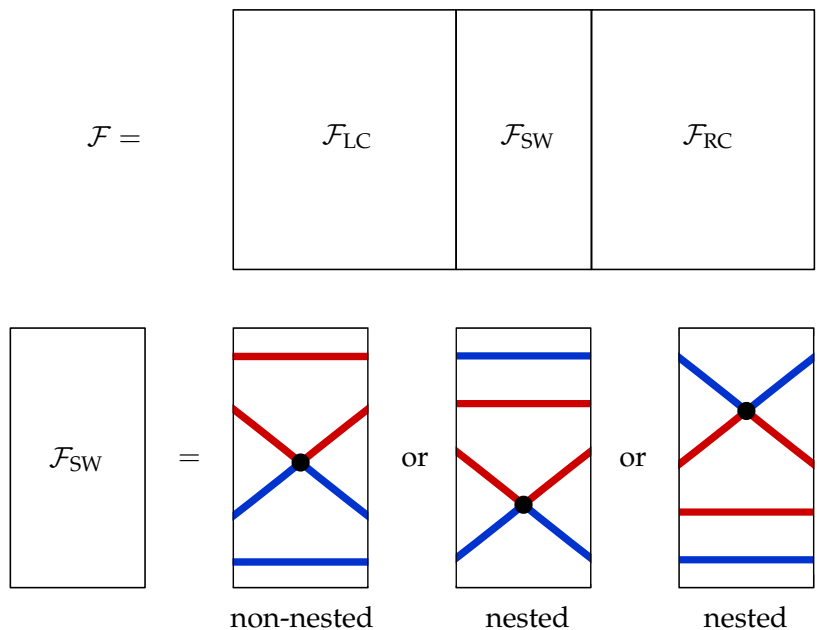
\begin{figure}
	\centering
	\begin{tikzpicture}
		\node at (0,0) {\includegraphics[scale=0.4]{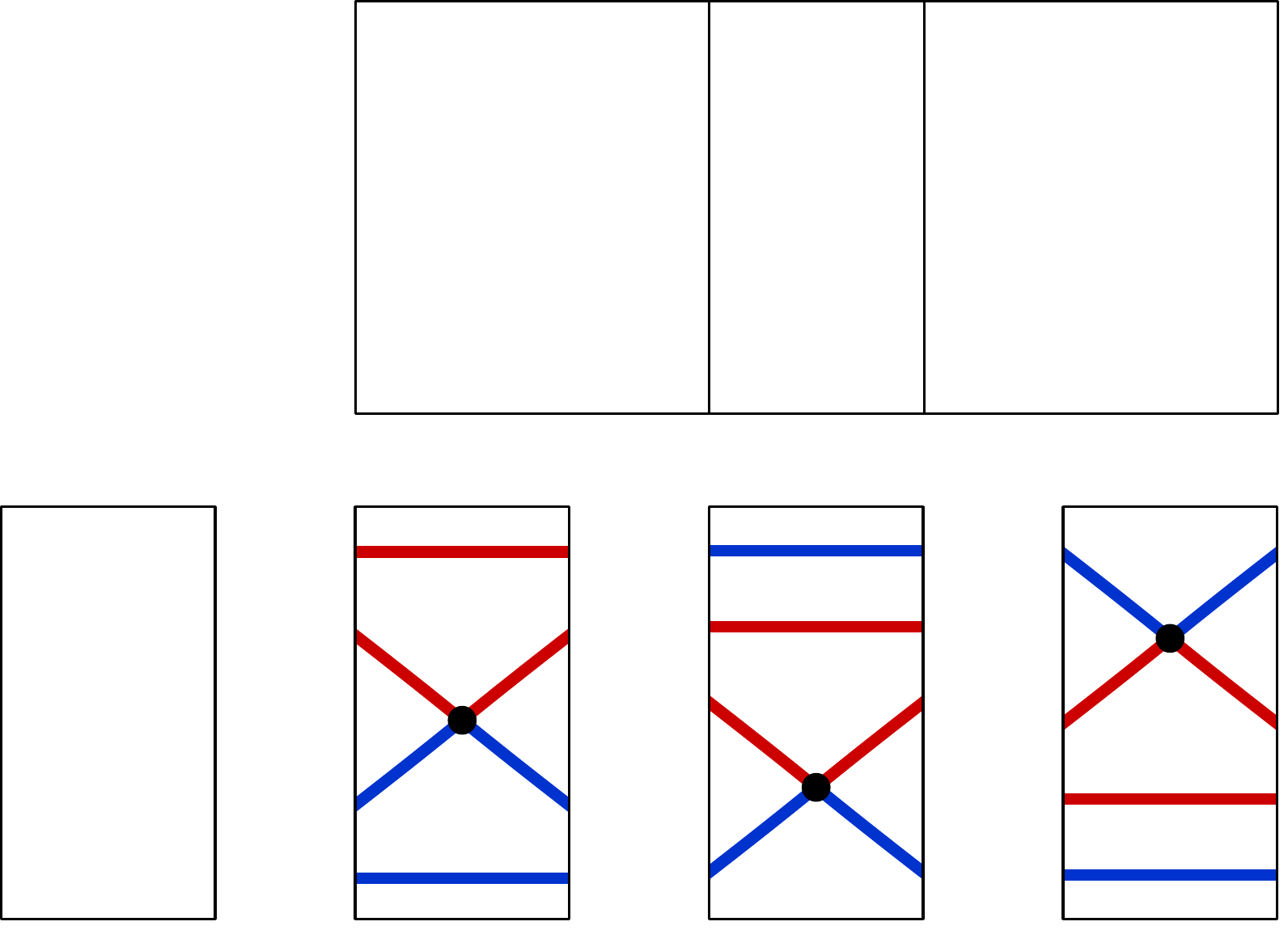}};
		\node at (-4,2.2) {$\mathcal{F}=$};
		\node at (-0.9,2.2) {$\mathcal{F}_{\text{LC}}$};
		\node at (1.5,2.2) {$\mathcal{F}_{\text{SW}}$};
		\node at (3.9,2.2) {$\mathcal{F}_{\text{RC}}$};
		
		\node at (-4.5,-2.1) {$\mathcal{F}_{\text{SW}}$};
		\node at (-3,-2.1) {$=$};
		\node at (0,-2.1) {or};
		\node at (2.9, -2.1) {or};
		\node at (-1.5,-4.2) {non-nested};
		\node at (1.5,-4.2) {nested};
		\node at (4.4,-4.2) {nested};
	\end{tikzpicture}
	\caption{Decomposition of a writhe-$1$ unknot front into three regions.}
	\label{fig:Obs1}
\end{figure}

Observe that the cusp regions cannot contain any of the three forbidden ``clasp" subregions indicated in the top row of \cref{fig:Obs2}. If they did, the modifications drawn in the bottom row of the same figure would produce distinct normal rulings of $\mathcal{F}$, contradicting uniqueness.

\begin{figure}
	\centering
	\includegraphics[scale=0.4]{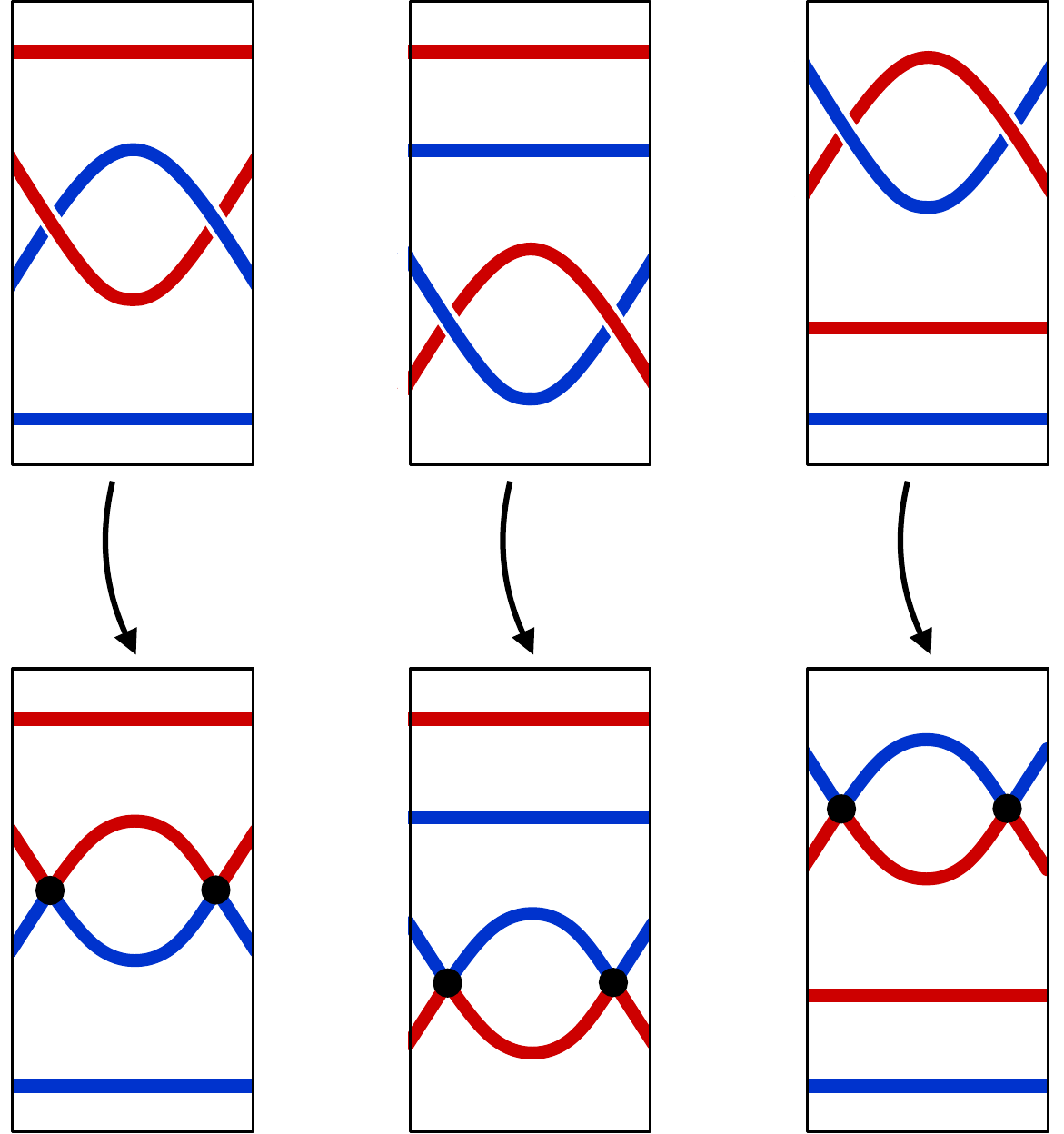}
	\caption{Configurations that violate uniqueness of normal rulings.}
	\label{fig:Obs2}
\end{figure}

We now proceed with casework based on the the type of switch region. 

\vspace{2mm}
\textsc{Case 1}: Non-nested switch region.
\vspace{2mm}

In the switch region we color the boundary of uppermost ruling disk red. Let $C$ denote the rightmost left cusp in $\mathcal{F}_{LC}$, which, up to reflective symmetry, we may assume is red. By assumption, there must be a first (from the right) crossing of the upper blue and lower red strand. Since $C$ is red, the upper blue strand must either intersect the lower red strand again, or intersect the upper red strand.

In the former case, the outlawed configurations in \cref{fig:Obs2} force the configuration pictured on the left of \cref{fig:Obs3}, where the blue ruling disk intersects the red ruling disk in a ribbon singularity (any other behavior of the lower blue strand would result in a clasp as in \cref{fig:Obs2}). We call such a region an $R^*$-region. Note that passing through an $R^*$-region from the right to the left is the identity map on the ordering of the strands (i.e.\ an $R^*$-region is a pure braid). Therefore, in general, we pass through some number of $R^*$-regions; if there are no crossings to the left of an $R^*$-region, there is a Legendrian RII as pictured on the right side of \cref{fig:Obs3}. 

\begin{figure}
	\centering
	\begin{tikzpicture}
		\node at (0,0) {\includegraphics[scale=0.3]{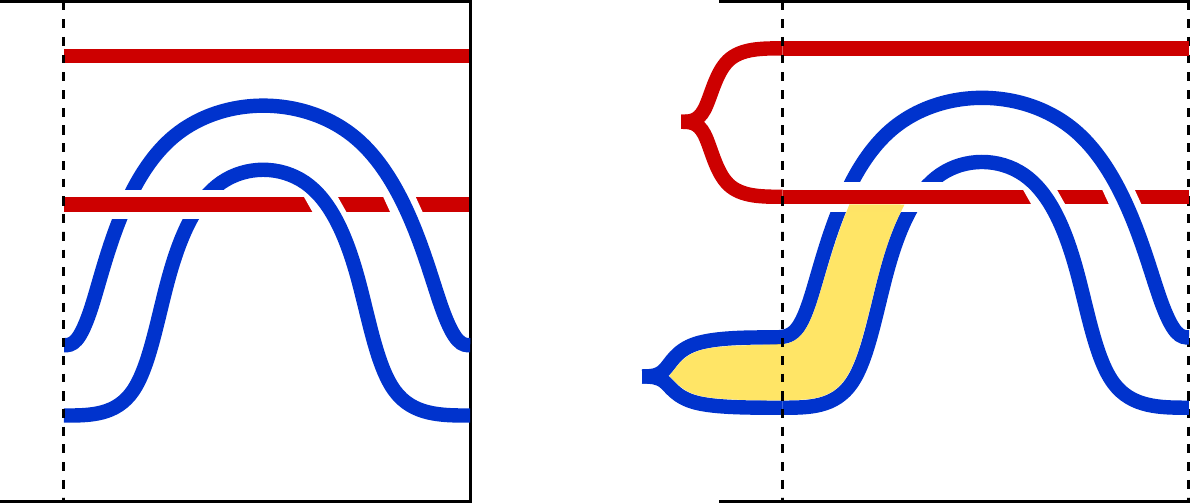}};
		\node at (-1.8,-1.6) {$R^*$};
		\node at (2,-1.6) {$R^*$};
	\end{tikzpicture}
	\caption{Identifying the simplifying Legendrian RII move after passing through some number of $R^*$-regions.}
	\label{fig:Obs3}
\end{figure}

Therefore, returning to the context of the first paragraph in \textsc{Case 1}, we may assume that after the initial red-blue crossing, the upper blue strand has a next intersection with the upper red strand. Up to a planar isotopy that doesn't affect the switched crossing, we further assume that the two crossings in question occur before any crossings involving the lower blue strand; see the left side of \cref{fig:Obs4}. If there are no additional crossings to the left of this region, there is a Legendrian RII as pictured on the right of the same figure. Consequently, we have reduced the proof to \textsc{Case 2}.

\begin{figure}
	\centering
	\includegraphics[scale=0.3]{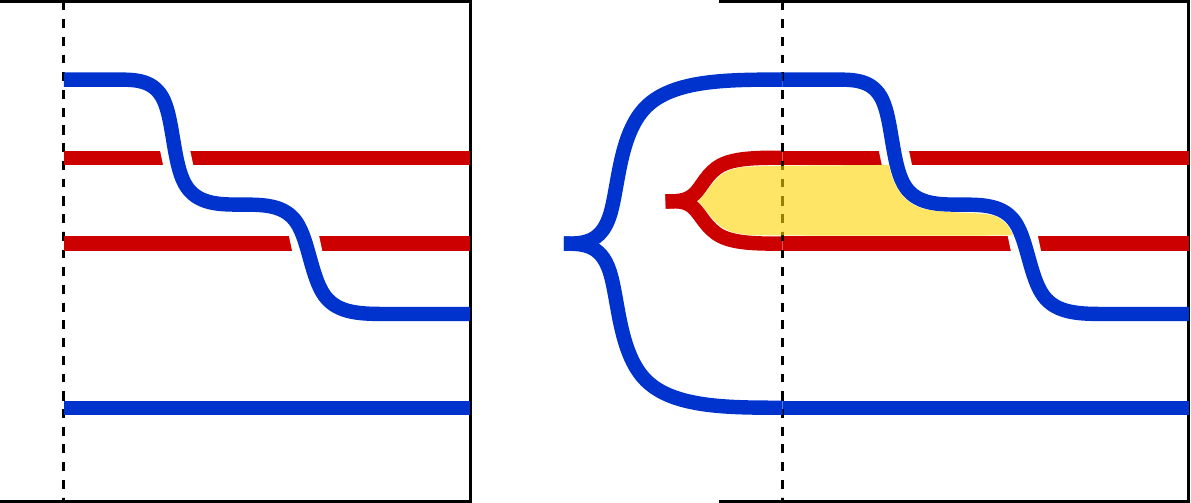}
	\caption{The other type of Legendrian RII move in \textsc{Case 1}.}
	\label{fig:Obs4}
\end{figure}

\vspace{2mm}
\textsc{Case 2}: Nested switch region.
\vspace{2mm}

There must be a first crossing, from the right, of a red and blue strand. By reflective symmetry, assume without loss of generality that the intersection is with the upper red and blue strands. In order to eventually form a red cusp, there must be a next red-blue crossing. It cannot be another crossing of the upper strands lest there be a clasp as in \cref{fig:Obs2}, so the next crossing must produce either one of the two regions indicated in the top row of \cref{fig:Obs5}.

\begin{figure}
	\centering
	\includegraphics[scale=0.3]{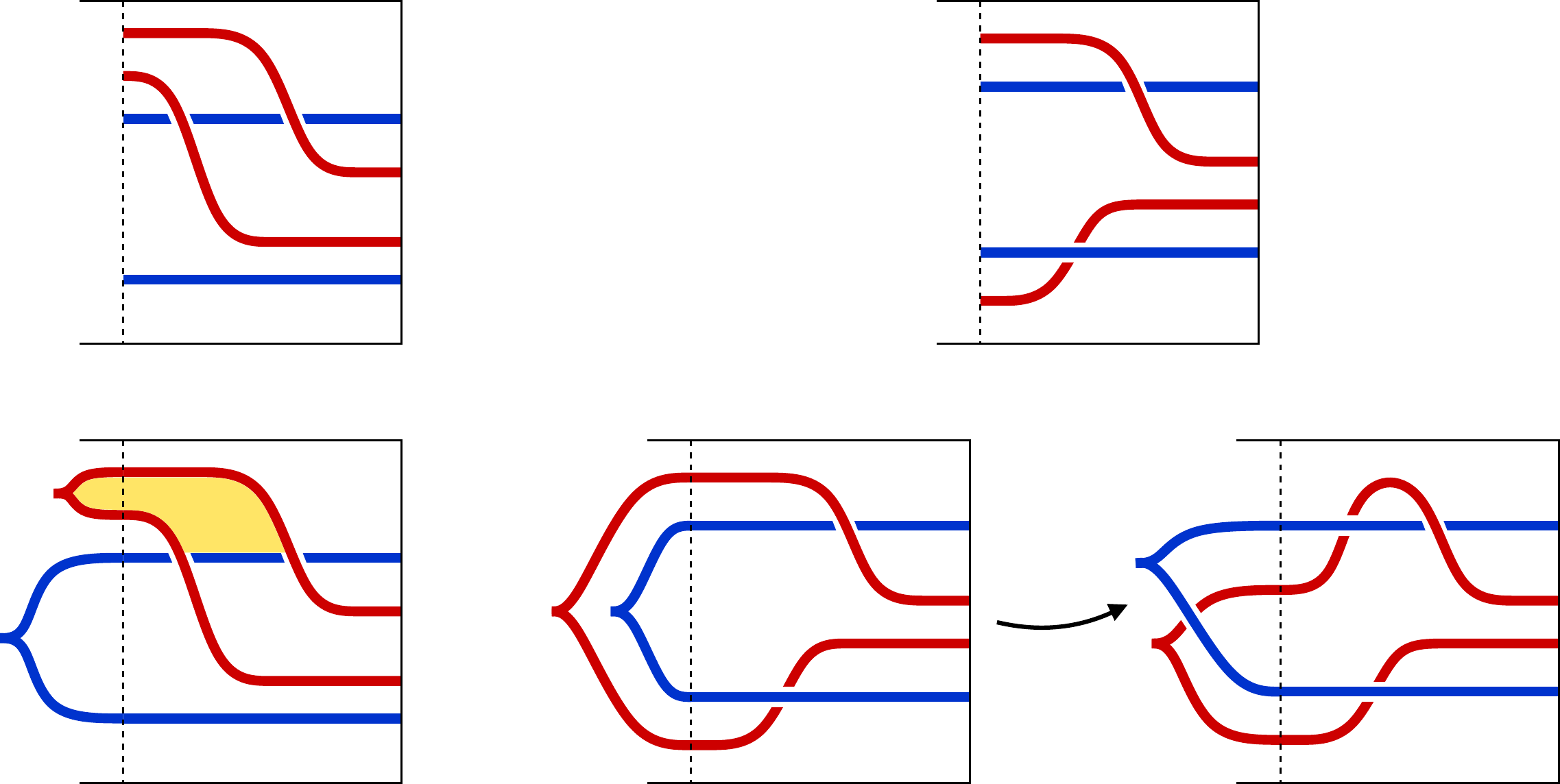}
	\caption{The possibilities in \textsc{Case 2}.}
	\label{fig:Obs5}
\end{figure}

In the left possibility, if there are no further crossings then there is a Legendrian RII as indicated below. Otherwise, we return to \textsc{Case 1}. In the right possibility, suppose first that there are no additional crossings to the left of the region. Then the left cusp region is determined as indicated, and after a Legendrian isotopy yields a forbidden clasp region. Therefore, there must be additional crossings, which takes us back to the start of \textsc{Case 2}. 

To conclude, the structure of the proof so far is to cycle between \textsc{Case 1} and \textsc{Case 2} repeatedly with the only terminating condition being a Legendrian RII move. By compactness of $\mathcal{F}$, the proof must terminate. Therefore, there exists a simplifying Legendrian RII.
\end{proof}

Next we establish the prohibited tangles of \cref{thm:obstructions}, stated separately as:

\begin{proposition}\label{prop:prohibited_tangles}
The oriented tangles depicted in \cref{fig:bad-tangles2}, up to rotation and total reversal of orientation (but not mirroring) cannot appear in a max-tb unknot front.
\end{proposition}

\begin{figure}[ht]
	\centering
	\begin{overpic}[scale=.6]{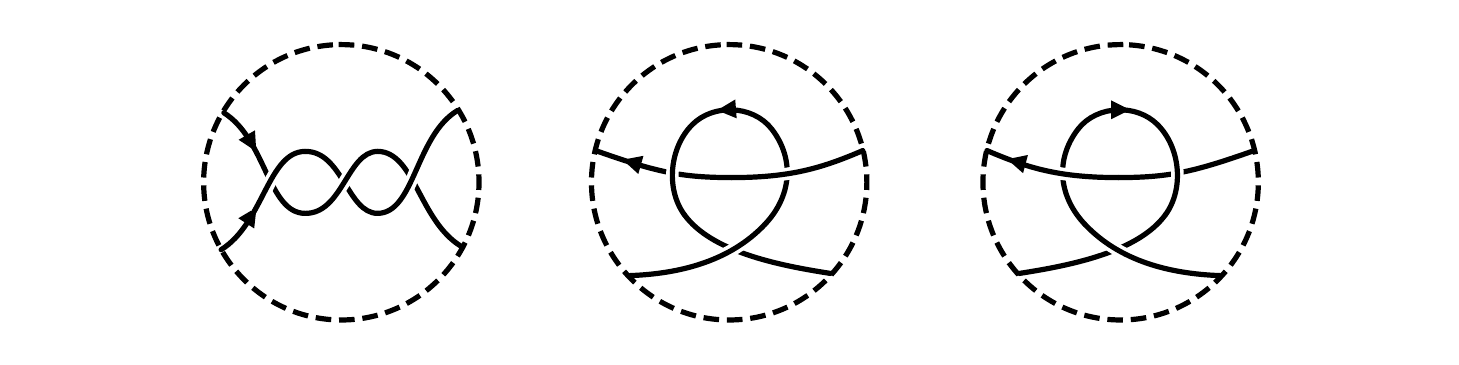}
     \put(21,0){(A)}
     \put(48.5,0){(B)}
     \put(75,0){(C)}
	\end{overpic}
	\caption{Tangles prohibited in a max-tb unknot front.}
	\label{fig:bad-tangles2}
\end{figure}

\begin{proof}
Recall that the unique normal ruling of a max-tb front has the property that every switched crossing is a positive crossing. 

Beginning with the tangle (A), note that each of the three crossings is negative. If the tangle were found in a max-tb unknot front, then none of the three crossings could be switched. The lack of switches means that each strand of the tangle can have $0$, $1$, or $2$ cusps. We will systematically refer to the strand beginning on the lower-left of the tangle as the red strand; the strand beginning on the upper left will be called the blue strand. If one of the strands has $0$ cusps, we can normalize it to have slope $0$. Assume that the slope-$0$ strand is the red strand; see the left of \cref{fig:bad-tangle-a}. Necessarily, the slopes of the blue strand at the crossings, from left to right, are positive, negative, and positive. With the assumed orientation, this is only possible with at least $4$ cusps, which is a contradiction.  

\begin{figure}[ht]
	\centering
	\begin{overpic}[scale=.6]{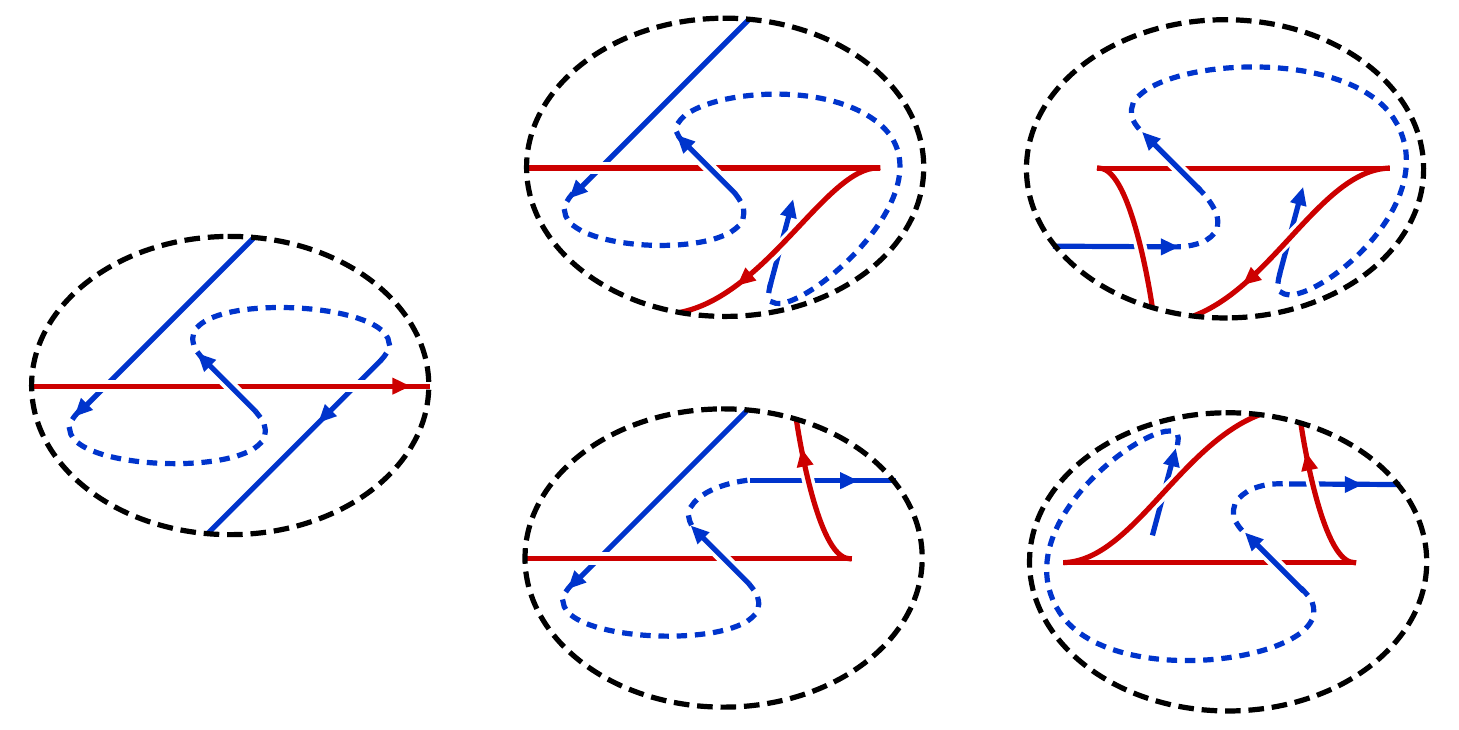}
     
	\end{overpic}
	\caption{The minimal cusp cases of tangle (A). Solid lines are Legendrian realized, dashed curves are topologically necessary paths.}
	\label{fig:bad-tangle-a}
\end{figure}

If the red strand has one cusp, then at least two of the crossings must occur on one of the two smooth loci of the red strand. If all three of the crossings occur on the same smooth locus, then the argument from the previous case applies, so assume that two of the crossings occur on one smooth locus. The middle of \cref{fig:bad-tangle-a} then depicts the two possible cases, up to natural symmetries. In either case the slopes and orientations of the blue strand force more than two cusps. The argument for the third case, where the red strand has two cusps, is similar. Note that a hypothetical normal ruling forces one red cusp to be an up-cusp and the other to be a down-cusp. We may also assume that each of the three smooth loci of the red projection supports one intersection, lest we reduce to previous cases. The two cases are depicted on the right side of \cref{fig:bad-tangle-a} and again require more than two blue cusps.

Moving on to tangle (B), note that again all three crossings are negative, hence none of them can be switched in a max-tb unknot ruling. This is impossible, as the strand with the self-crossing would be unicolored, violating a defining property of a ruled front. 

In the tangle (C), the self-crossing of the looping strand is now positive, while the other two crossings are negative. By the same argument in the previous paragraph, we see that the positive crossing must be switched. As the other crossings cannot be switched, the switched loop must be an eye. Moreover, by orientation considerations the color of the strand passing through the loop must be distinct from the color adjacent to the switch. As this strand has no switched crossings, it has either $0$, $1$, or $2$ cusps, as in the analysis of tangle (A). Thus, there are  three possibilities for the front, depicted in the top row of \cref{fig:bad-tangle-c}. Each of the configurations leads to a contradictory extraneous (non-orientable) ruling, possibly after some Reidemeister moves which do not affect the initial ruling, as depicted in the lower row of the figure. 
\end{proof}

\begin{figure}[ht]
	\centering
	\begin{overpic}[scale=.6]{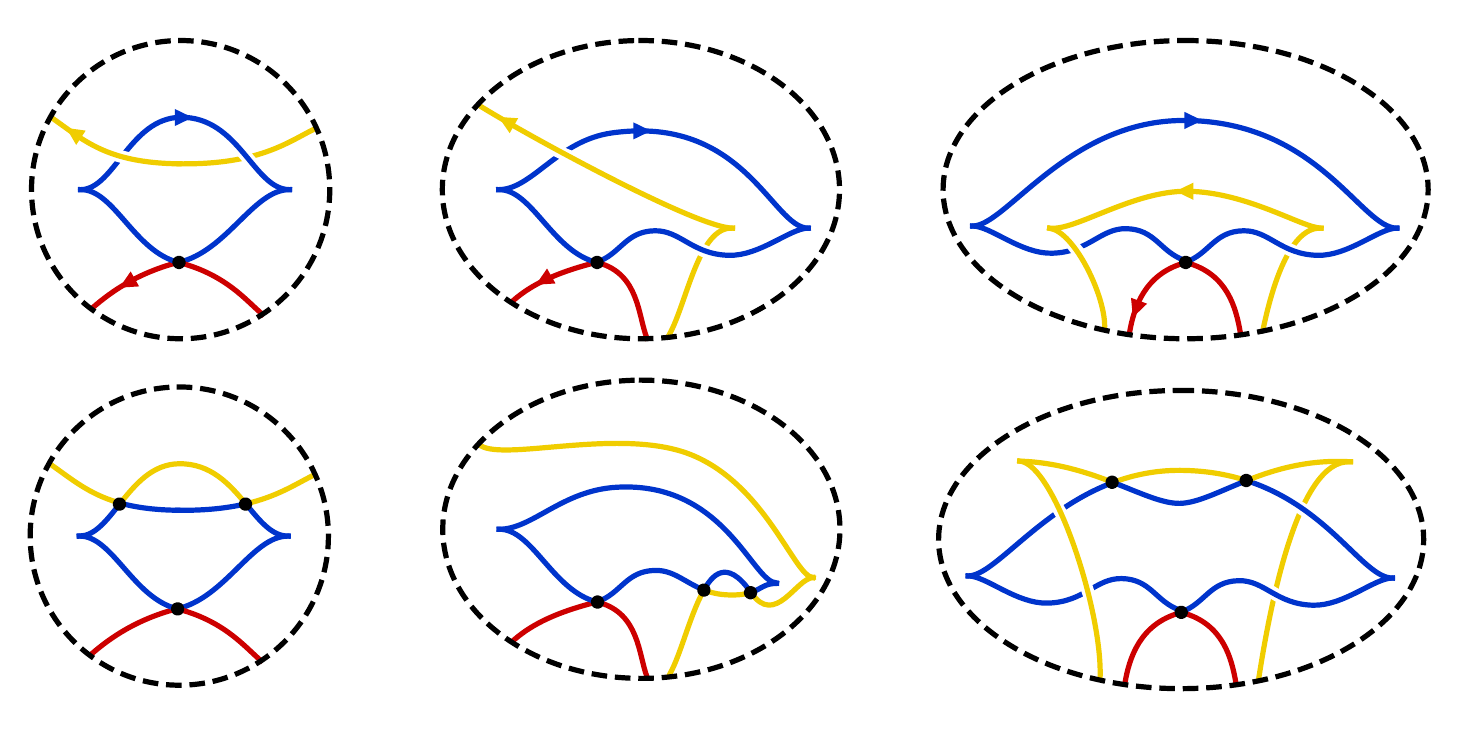}
     
	\end{overpic}
	\caption{The minimal cusp cases of tangle (C) in the top row, and the resulting extraneous normal rulings in the bottom row.}
	\label{fig:bad-tangle-c}
\end{figure}

\begin{remark}
The mirroring and orientation of the prohibited tangles in \cref{fig:bad-tangles} is important. For example, see \cref{fig:tangle2} for a max-tb unknot which contains the mirror of a tangle prohibited by \cref{prop:prohibited_tangles}.
\end{remark}

\begin{figure}[ht]
	\centering
	\begin{overpic}[scale=.6]{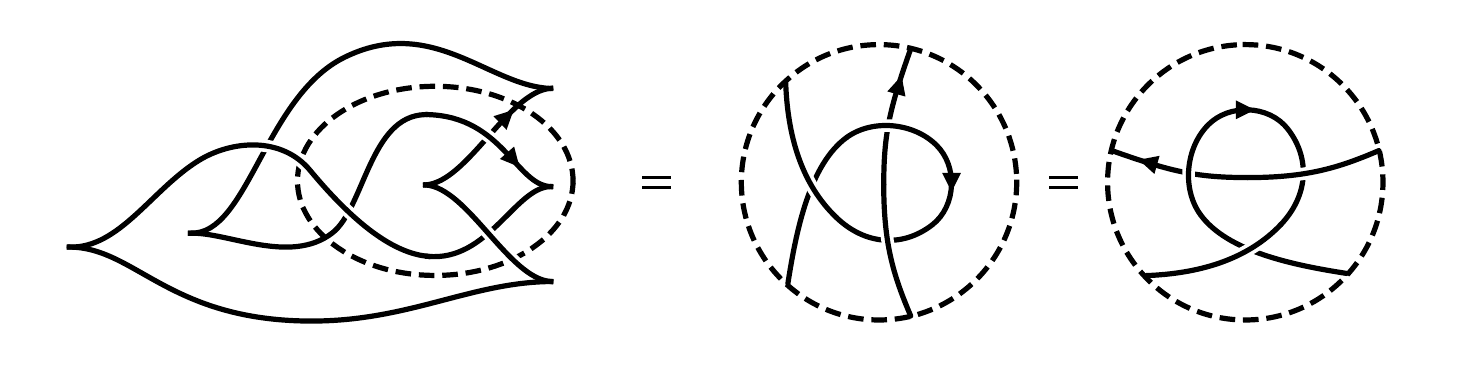}
     
	\end{overpic}
	\caption{A max-tb unknot front with the mirror of a prohibited tangle.}
	\label{fig:tangle2}
\end{figure}

With the proof of \cref{thm:obstructions} completed by the previous two propositions, we prove \cref{cor:families}.

\begin{proof}[Proof of \cref{cor:families}.]
In \cref{fig:generalized} we recall the construction of generalized Goeritz unknots, generalized Freedman-He-Wang unknots, and Hass-Nowik unknots. The integral boxes represent half-twists (so that the classical Goeritz unknot corresponds to $n=3$); a box labeled $K$ is any knot, $\bar{K}$ is its mirror, and the two strands passing through the box form a cut open copy of the knot and its blackboard pushoff. From the diagrams it is immediate that the generalized Goeritz and Hass-Nowik unknots have prohibted tangles of type (A), so these can never be max-tb unknots. The generalized Freedman-He-Wang unknot has writhe $0$, so any hard diagram arising from this construction cannot be a max-tb unknot. 
\end{proof}

\begin{figure}[ht]
	\centering
	\begin{overpic}[scale=.2]{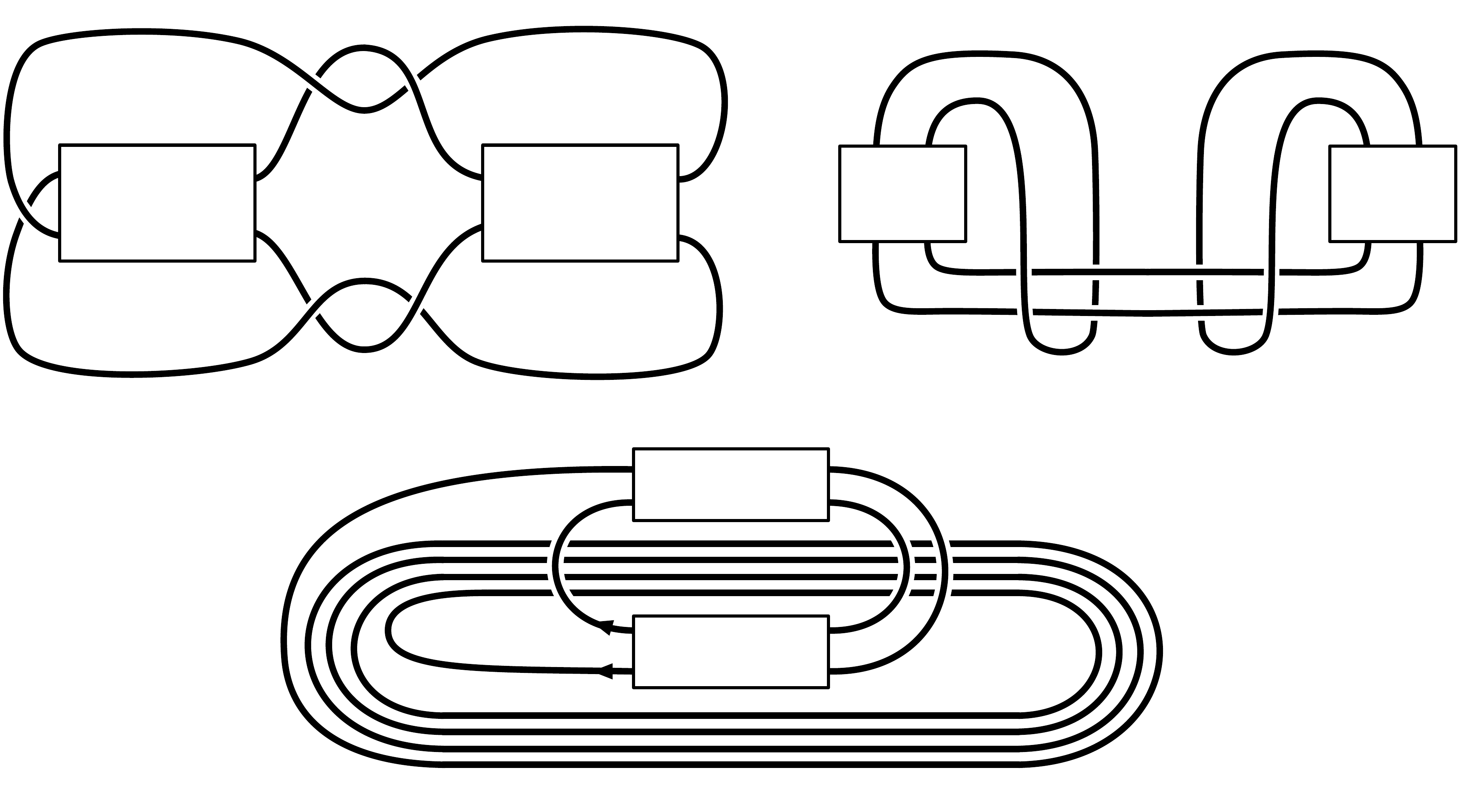}
        \put(10,41){$n$}
        \put(37,41){$-n$}

        \put(60,41.25){$K$}
        \put(93.75,41.25){$\bar{K}$}

        \put(46,21){$2n-1$}
        \put(47,10){$-2n$}
	\end{overpic}
	\caption{Generalized Goeritz (top left, $n\geq 3$), generalized Freedman-He-Wang (top right), and Hass-Nowik (bottom, $n\geq 4$) unknots.}
	\label{fig:generalized}
\end{figure}
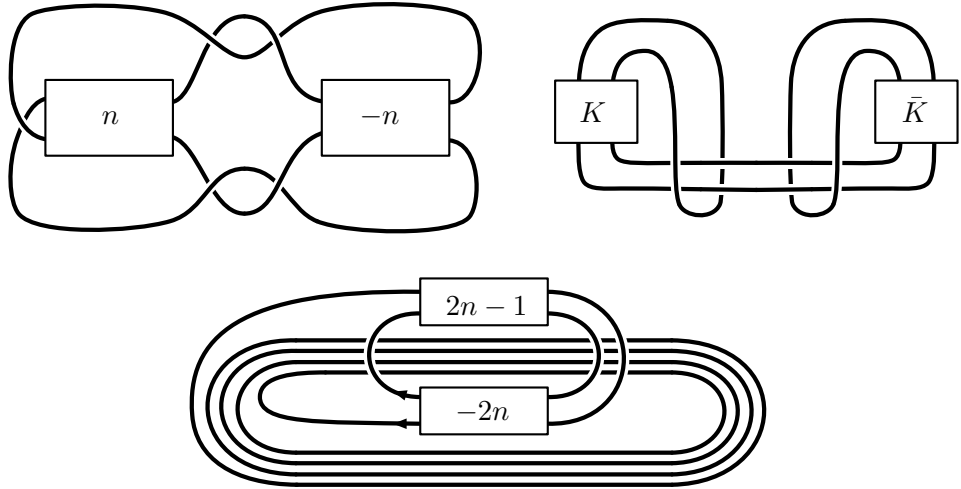

In preparation for the proof of \cref{thm:KL}, we record another useful observation using the same idea underlying \cref{prop:prohibited_tangles}. 

\begin{proposition}\label{prop:two-switch}
If a tangle of the form in \cref{fig:two-switch} appears in a max-tb unknot front, then both crossings must be switched.     
\end{proposition}

\begin{figure}[ht]
	\centering
	\begin{overpic}[scale=.5]{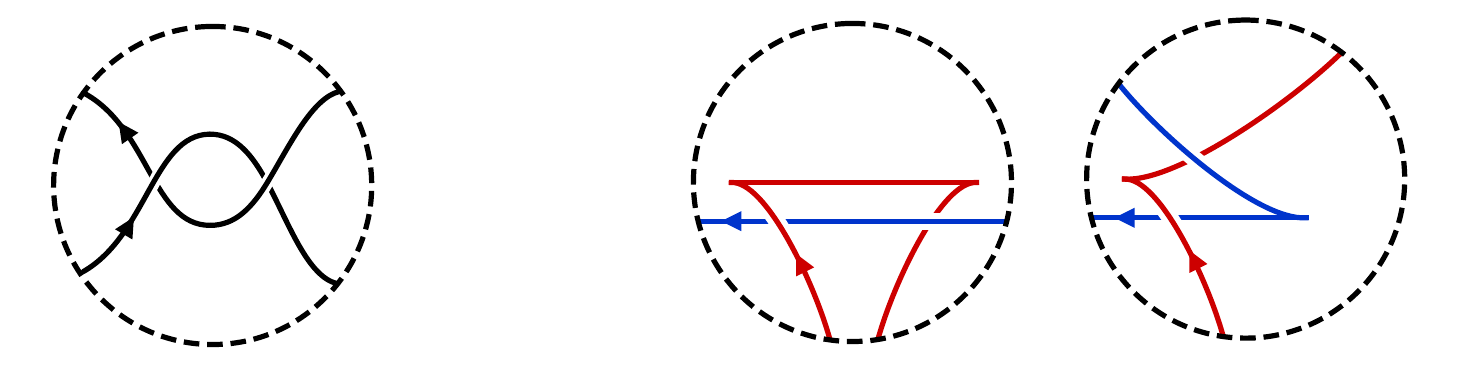}
     
	\end{overpic}
	\caption{A tangle (left) requiring all crossings to be switched.}
	\label{fig:two-switch}
\end{figure}

\begin{proof}
It is not possible for exactly one of the crossings to orientably switch, for otherwise there would be an unswitched intersection of the same color. Thus, assume for the sake of contradiction that neither of the crossings are switched.

Normalize the leftmost crossing so that the overcrossing strand has more negative slope. We additionally may assume without loss of generality that there are no cusps on the segments from this crossing to the boundary of the tangle, i.e.\ the segments with the arrows in \cref{fig:two-switch}, as these have no other intersections with the rest of the tangle. Now we consider all possibilities of minimal cusp examples according to whether the rest of the undercrossing strand has $0$, $1$, or $2$ cusps. By symmetry it suffices to consider the case of $0$ or $1$ cusp. The two minimal cusp examples are given on the right side of  \cref{fig:two-switch}, and both lead to extraneous rulings after switching both crossings (c.f.\ \cref{fig:bad-tangle-c}). This contradicts uniqueness of rulings, so both crossings must be switched.  
\end{proof}

Next, we consider the hard unknot diagrams of Kauffman and Lambropoulou. There are a few variations (see \cite[p.\ 33]{kauffman2011collapsing}) that can each be obstructed by a similar argument, so we restrict to their main construction, which we recall now. Given a tuple $(a_1, a_2, \dots, a_{2n})$ of positive natural numbers with $n\geq 2$, Kauffman and Lambropoulou prove that the diagram in \cref{fig:KLrational} is hard. (The reader may quickly verify that the diagram is not hard on the sphere when $n=1$.)  

\begin{figure}[ht]
	\centering
	\begin{overpic}[scale=.36]{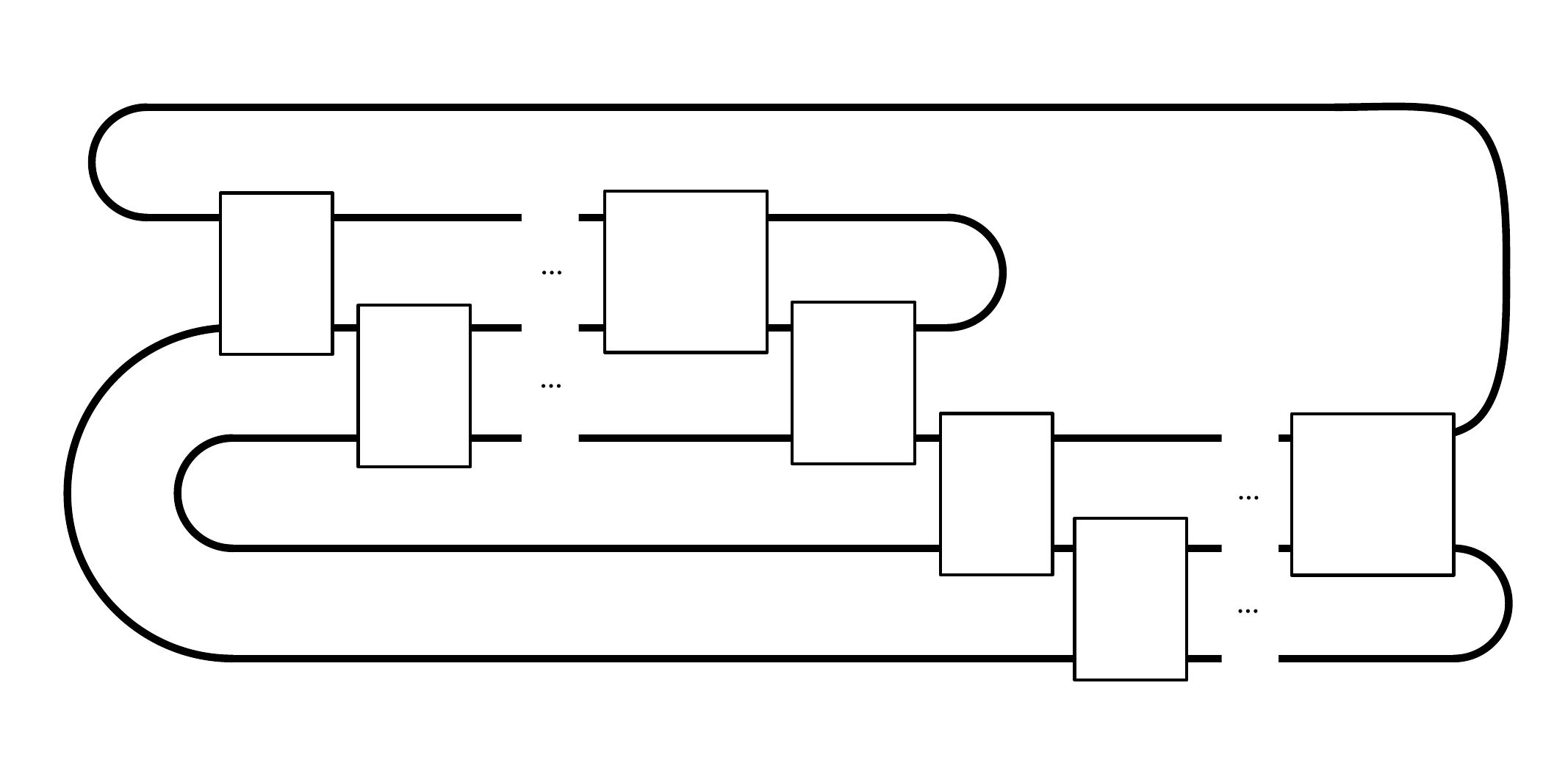}
      \put(15,30.5){\small $\pm a_1$}
      \put(24,23.5){\small $\mp a_2$}
      \put(39.25,30.5){\small $\pm a_{2n-1}$}
      \put(51.25,23.5){\small $\mp a_{2n}$}

      \put(61,16.25){\small $\mp a_1$}
      \put(69.5,9.75){\small $\pm a_2$}
      \put(83,16.25){\small $\mp a_{2n-1}$}
	\end{overpic}
    \vskip-0.3cm
	\caption{Kauffman-Lambropoulou hard unknots.}
	\label{fig:KLrational}
\end{figure}
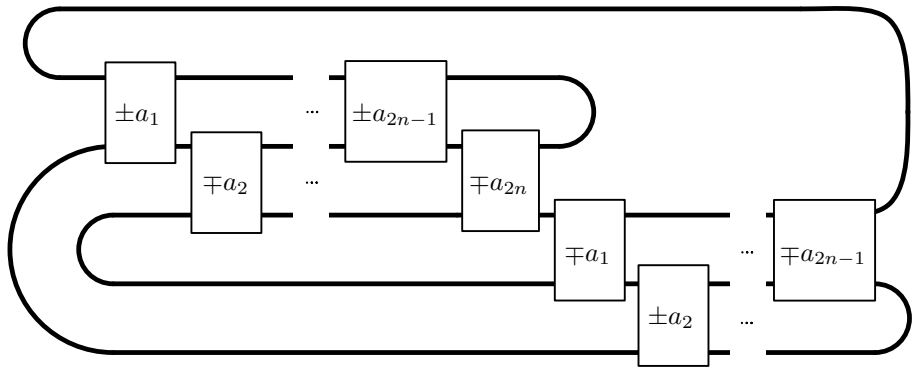

\begin{proof}[Proof of \cref{thm:KL}.]
If both strands entering the left of the central twist box labeled $\mp a_{2n}$ are oriented in the same direction, then the writhe of the diagram is $\mp a_{2n}$; otherwise, if the strands have opposite orientation the writhe is $\pm a_{2n}$. If the diagram is a max-tb unknot, then \cref{thm:obstructions} requires the writhe to be $a_{2n}\geq 2$. Therefore, under this assumption, there are two cases to consider: the choice of signs leading to a central box labeled $a_{2n}$ and the strands entering the box with the same orientation, or the choice of signs leading to a central box labeled $-a_{2n}$ and the strands entering the box with opposite orientation. 

\vspace{2mm}
\textsc{Case 1}: Central box $a_{2n}$, strands in the same direction.
\vspace{2mm}

We will show that this case cannot occur. Orient the strands passing through the $a_{2n}$ box from left to right. Let $s\in S_3$ denote the permutation induced by the the $3$-braid $\sigma_2^{-a_1}\sigma_1^{a_2}\cdots \sigma_2^{-a_{2n-1}}$, where we index (horizontal) strands from the bottom. Note that $s^{-1}(3) \neq 3$, for otherwise the strand exiting the upper right of the $a_{2n}$ box would turn around, pass backward through the above $3$-braid, and exit the top left of the $-a_1$ box heading west. It would then circle around the top of the diagram and enter the upper right of the $a_{2n-1}$ box, once again emerging from the top left of the $a_1$ box in a westward direction. But this contradicts the fact that the strands passing through the $a_{2n}$ box are oriented eastward. 

Assume for the sake of contradiction that $s^{-1}(3) = 2$. Then the upper left $3$-braid region determines a westward orientation for the strand emanating from the lower left of the $-a_1$ box, and thus an eastward orientation for the bottom horizontal strand. Note that the upper strand on the left side of the $-a_1$ box must be oriented eastward, since this strand enters the $a_{2n}$ box from its left after passing through the $3$-braid region. This, then, means that the strand emanating from the upper right of the $a_{2n-1}$ box is oriented eastward, and because $s^{-1}(3)=2$, that all three lower strands are eastwardly-oriented. See \cref{fig:KLrational2}. The case $s^{-1}(3) = 1$ is nearly identical, and is left to the reader. These arguments show that \textsc{Case 1} cannot occur.

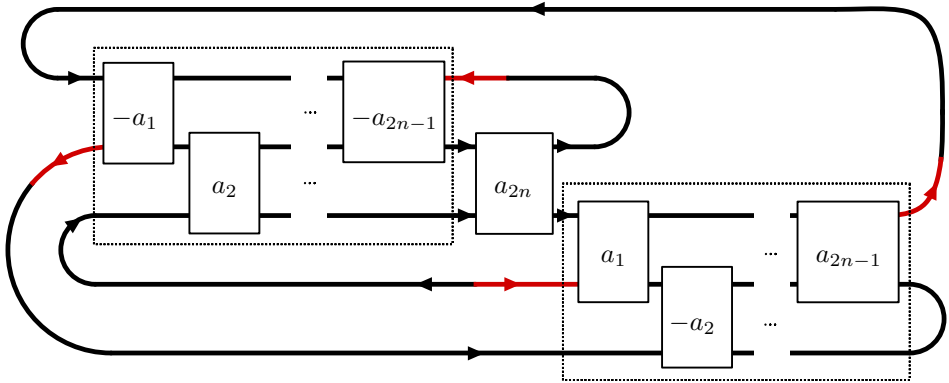
\begin{figure}[ht]
	\centering
	\begin{overpic}[scale=.36]{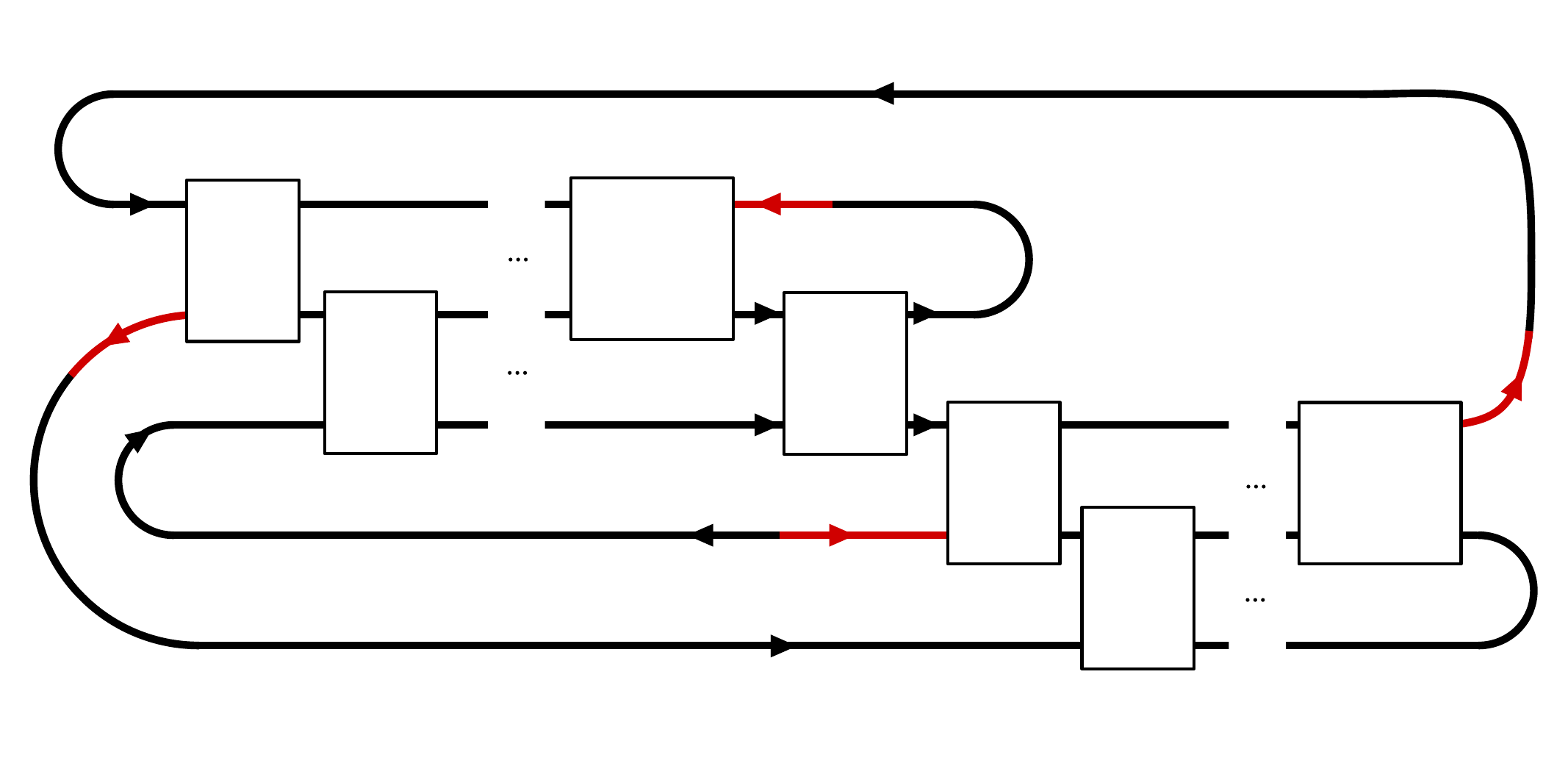}
      \put(12.75,31){\small $- a_1$}
      \put(23,24){\small $a_2$}
      \put(37.25,31){\small $-a_{2n-1}$}
      \put(51.85,24){\small $a_{2n}$}

      \put(62.75,17){\small $a_1$}
      \put(69.75,10){\small $-a_2$}
      \put(84.75,17){\small $a_{2n-1}$}
	\end{overpic}
    \vskip-0.3cm
	\caption{Orientation contradiction when $s^{-1}(3) = 2$.}
	\label{fig:KLrational2}
\end{figure}

\vspace{2mm}
\textsc{Case 2}: Central box $-a_{2n}$, strands in opposite direction, $n\geq 3$.
\vspace{2mm}

Now consider the diagram where the central box is $-a_{2n}$ and the strands passing through the box have opposite orientation. Suppose for the sake of contradiction that the diagram is a ruled max-tb unknot. First assume $n\geq 3$. Since $a_{2n}\geq 2$, \cref{prop:two-switch} implies that all crossings in the tangle formed by the $-a_{2n}$ box must be switched. Therefore, the two strands emanating from the right of the $-a_{2n}$ box have the same color (blue), and the two strands emanating from the left have the same color (red). Since the writhe of the diagram is $a_{2n}$, there are no other switches. As no other crossings in the diagram are switched, two strands of the same color cannot pass through any of the other boxes, for otherwise this would violate the ruling property. 

This latter observation forces the lower strand emanating from the left of the $-a_1$ box to the right of the $-a_{2n}$ box to be red. It turns around and enters the $-a_2$ box from the left as the lower strand, forcing the other strand passing through the box to be blue, and the third strand in the $3$-braid region to be red. Analyzing the $a_3$ box, the same reasoning implies that the strands exiting to the right of the $-a_2$ are red on the bottom and blue on top. In particular, $a_2$ must be even. Since $n\geq 3$, the same argument applies by studying the region between the $a_3$ and $-a_4$ box and we conclude that $a_3$ is even. Continuing, we conclude that $a_2, \dots, a_{2n-2}$ are even.  

\begin{figure}[ht]
	\centering
	\begin{overpic}[scale=.36]{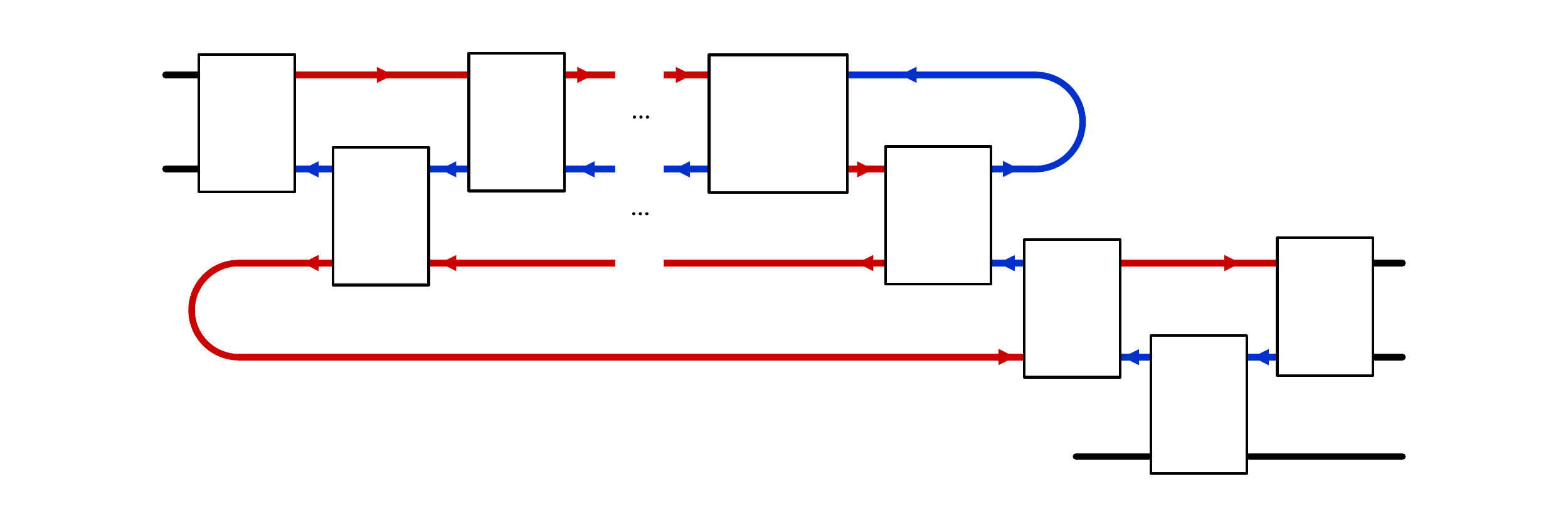}
      \put(14.25,25){\small $a_1$}
      \put(22.25,19){\small $-a_2$}
      \put(32,25){\small $a_3$}
      \put(47,25){\small $a_{2n-1}$}
      \put(56.85,19){\small $-a_{2n}$}

      \put(66.25,13){\small $-a_1$}
      \put(75.25,6.75){\small $a_2$}
      \put(82.25,13){\small $-a_{3}$}
	\end{overpic}
    
	\caption{The portion of the diagram relevant for the contradiction when $n\geq 3$ and the central box is $-a_{2n}$.}
	\label{fig:KLrational3}
\end{figure}
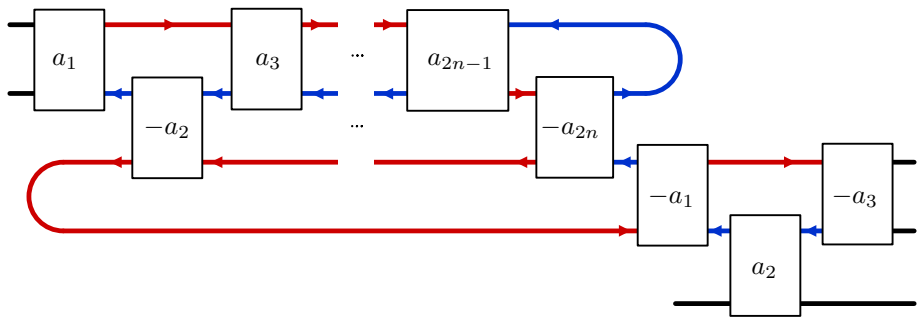

On the other hand, consider the two red strands exiting to the left of the $-a_{2n}$ box. It follows that $a_{2n-1}$ must be odd, lest there be two red strands passing through the $a_{2n-2}$ box.

Now we investigate the orientations of the strands. Orient the blue strands incident to the $-a_{2n}$ box as indicated in \cref{fig:KLrational3}; this determines the orientation of the rest of the blue strand in the upper left $3$-braid region (but it does not determine the orientation of the red strand, as this depends on the parity of $a_{2n}$). Notice, however, that the red strand must pass through the $-a_2$ box from right to left, for otherwise \cref{prop:two-switch} would imply that all $a_n \geq 2$ crossings in that box are switched, which is a contradiction. Following along this strand, we see that the strands pass through the $-a_1$ box with opposite orientation. Appealing once again to \cref{prop:two-switch}, we are forced to conclude $a_1 = 1$. Evenness of $a_2$ then forces the red strands to enter in opposite orientation through the $-a_3$ box. Earlier, though, we concluded that $a_3$ is even, and \cref{prop:two-switch} then forces both of these crossings to switch, a contradiction.

\vspace{2mm}
\textsc{Case 3}: Central box $-a_{2n}$, strands in opposite direction, $n = 2$.
\vspace{2mm}

For the special case $n=2$, the same argument implies $a_2$ is even and $a_1 = 1$. The difference is that $a_3 = a_{2n-1}$ is definitively odd, so we do not yet have a contradiction. However, the same reasoning we applied to the $-a_1$ box applies to the $-a_3$ box and we conclude that, in fact, $a_3 = 1$. Moreover, to prevent the $-a_2$ box from becoming a prohibited (A) tangle, we in fact have $a_2 = 2$. Letting $a_{4} =: 2k$, these simplifications then lead us to the diagram on the left side of \cref{fig:KL3}.

\begin{figure}[ht]
	\centering
	\begin{overpic}[scale=.32]{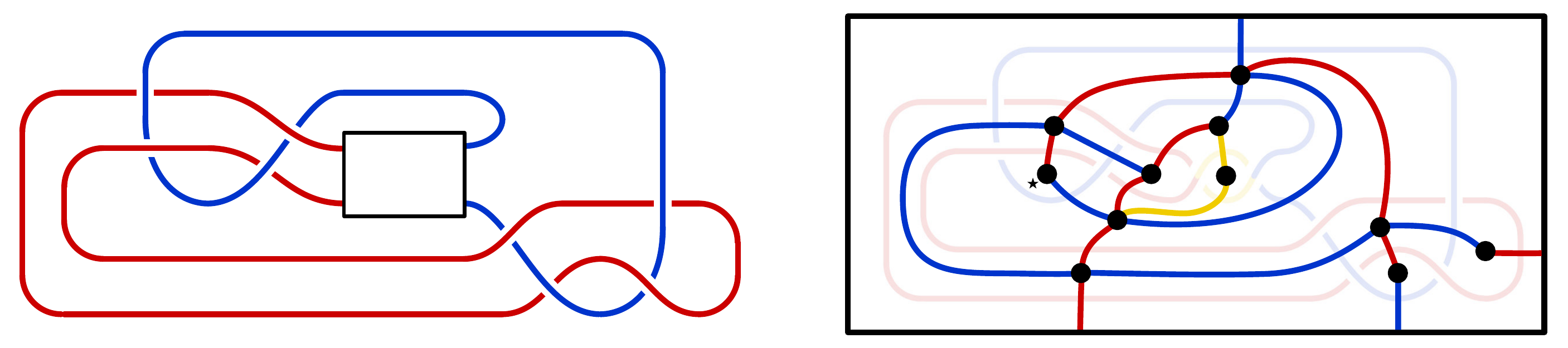}
      \put(23.5,10.25){\small $-2k$}
     
	\end{overpic}
    
	\caption{The case $n=2$, and the dual graph for $-2k=-2$.}
	\label{fig:KL3}
\end{figure}

Next, observe that any front projection $\mathcal{F}$ of the max-tb unknot with its unique normal ruling has the following property. Viewing the smooth diagram $D \subset S^2$, and given any point $p\notin D$, there is a circle $C \subset S^2$ passing through $p$, intersecting $D$ 
transversely and away from double points, such that $C$ intersects each color no more than twice. (These circles correspond to vertical lines in the front projection.) We claim that the diagram in \cref{fig:KL3} does not have this property. 

This is easiest to see from the perspective of the dual graph of the diagram, with edges colored according to the ruling. We have drawn the dual graph for $-2k=-2$ on the right, with the vertex corresponding to the region at infinity drawn as a black rectangle; it will be clear that the argument holds for any $k$. The above circle property implies the following property of the dual graph: given any interior vertex, there is a walk from the rectangle at infinity to the prescribed vertex and back to the rectangle at infinity such that each edge color is traversed no more than twice. 

Our diagram fails to have this property because of the starred vertex. Note that any walk from the starred vertex to the rectangle at infinity traverses at least two red edges. Therefore, any ``circle" from the rectangle to the starred vertex and back must traverse at least four red edges. This is a contradiction, which completes the proof. 
\end{proof}

Now we prove \cref{thm:burton}. The summary of the obstructions is in \cref{table1}.

\begin{table}[ht]
\centering
\begin{tabular}{|c|c|c|c|c|c|c|}
\hline
\textbf{\cite[Figure $\ast$]{burton2024hard}} & \textbf{Mirror} & \textbf{W.P.H?} & \textbf{S.S.H?} &  \textbf{Writhe} & \textbf{Max-tb front?} & \textbf{Obstruction} \\
\hline
8 & & $\checkmark$ & $\checkmark$ & 0 & X & writhe \\
\hline
9 & $\checkmark$ & $\checkmark$ & $\checkmark$ & 3 & X & ad hoc \\
\hline
10 & & $\checkmark$ & $\checkmark$ & 0 & X & writhe \\
\hline
11 & $\checkmark$ & $\checkmark$ & $\checkmark$ & 1 & X & writhe \\
\hline
12 &  & $\checkmark$ & $\checkmark$ & 1 & X & writhe \\
\hline
13 & & $\checkmark$ & $\checkmark$ & 2 & X & tangle (B) \\
\hline
14 & $\checkmark$ & X & X & 2 & X & tangle (A) \\
\hline
15 & & $\checkmark$ & $\checkmark$ & 1 & X & writhe \\
\hline
16 &  & ? & X & 1 & X & ad hoc \\
\hline
17 & & ? & X & 2 & X & ad hoc \\
\hline
18 & & ? & X & 5 & X & tangle (A) \\
\hline
19 & & $\checkmark$ & X & 0 & X & writhe \\
\hline
20 & & $\checkmark$ & X & 0 & X & writhe \\
\hline
21 & $\checkmark$ & ? & X & 3 & X & tangle (C) \\
\hline
22 & & $\checkmark$ & $\checkmark$ & 3 & X & tangle (C) \\
\hline
23 & $\checkmark$ & ? & X & 5 & X & tangle (A) \\
\hline
24 & & $\checkmark$ & $\checkmark$ & 5 & X & tangle (A) \\
\hline
25 & & ? & X & 37 & ? & \\
\hline
26 & $\checkmark$ & ? & X & 1 & X & tangle (A) \\
\hline
27 & $\checkmark$ & $\checkmark$ & $\checkmark$ & 4 & X & ad hoc \\
\hline
28 & & $\checkmark$ & $\checkmark$ & 0 & X & writhe \\
\hline
\end{tabular}
\vspace{5mm}
\caption{Analyzing the unknot diagrams in \cite[Appendix A]{burton2024hard} using the figure enumerations in that article. For diagrams with nonzero writhe, we have indicated whether we use the mirror of their diagram to obtain a non-negative writhe. W.P.H stands for weakly planar hard (the weakest form of hardness) and S.S.H stands for strongly spherically hard (the strongest form of hardness).}
\label{table1}
\end{table}

\begin{proof}[Proof of \cref{thm:burton}.]
Any of the diagrams with writhe $0$ or $1$ that are weakly planar hard are ruled out from being max-tb unknot fronts by the writhe obstruction in \cref{thm:obstructions}. This accounts for \cite[Figures 8, 10, 11, 12, 15, 19, 20, 28]{burton2024hard}.  By inspection, the (appropriately mirrored, as indicated in \cref{table1}) \cite[Figures 14, 18, 23, 24, 28]{burton2024hard} contain tangle (A), \cite[Figure 13]{burton2024hard} contains tangle (B), and the \cite[Figures 21, 22]{burton2024hard} contain tangle (C), obstructing each of these. 

This leaves the appropriately mirrored \cite[Figures 9, 16, 17, 27]{burton2024hard}. The rest of the proof is dedicated to ad hoc arguments using the same reasoning underlying \cref{prop:prohibited_tangles} for these figures. We can quickly dispense of \cite[Figure 16]{burton2024hard}, because it has writhe $1$, hence the hypothetical unique oriented normal ruling would have one switch, which is a positive crossing. One can quickly verify that none of the $8$ positive crossings, when switched, yield a decomposition of the diagram into two planar disks. The remaining three diagrams are considered in three separate arguments. 

\vspace{2mm}
\textsc{Case 1}: \cite[Figure 27]{burton2024hard}.
\vspace{2mm}

We begin with the simplest argument for \cite[Figure 27]{burton2024hard} (which is the same as \cite[Figure 27]{petronio2016unknots}). First, note that in a diagram of the max-tb unknot, any tangle of the form given in \cref{fig:switch-tangle} (or its mirror) must have at least one switched crossing, for if it didn't, the self-intersecting strand would violate the ruling condition. The diagram in question has $16$ pairwise disjoint tangle regions of this form, hence a hypothetical normal ruling would have at least $16$ switched crossings. This contradicts the fact that the writhe is $4$. 

\begin{figure}[ht]
	\centering
	\begin{overpic}[scale=.5]{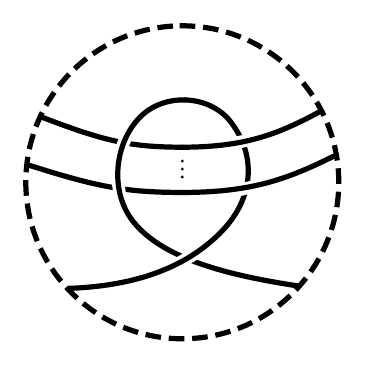}
     
	\end{overpic}
	\caption{A tangle requiring at least one switched crossing.}
	\label{fig:switch-tangle}
\end{figure}

\vspace{2mm}
\textsc{Case 2}: \cite[Figure 17]{burton2024hard}, the Ochiai I unknot.
\vspace{2mm}

The relevant unknot diagram is drawn on the left of \cref{fig:fig17}, with positive (negative) crossings highlighted in green (resp.\ pink). Assume for the sake of contradiction that there is a max-tb front witnessing the diagram, so that it admits a unique normal ruling which is moreover orientable. Then none of the negative crossings may be switched. Consequently, we must color the strand emanating north from the crossing labeled $a$ to crossing $b$ with a single color. Likewise, the strand emanating east from crossing $c$ to crossing $d$ must be uni-colored. Since it crosses the red strand at an unswitched crossing, we must use a different color (blue). Similarly, the strand emanating east from $e$ to $f$ must be uni-colored with a third color (yellow) as it intersects both red and blue at unswitched crossings. This produces the partial coloring in the middle of \cref{fig:fig17}.  

\begin{figure}[ht]
	\centering
	\begin{overpic}[scale=.34]{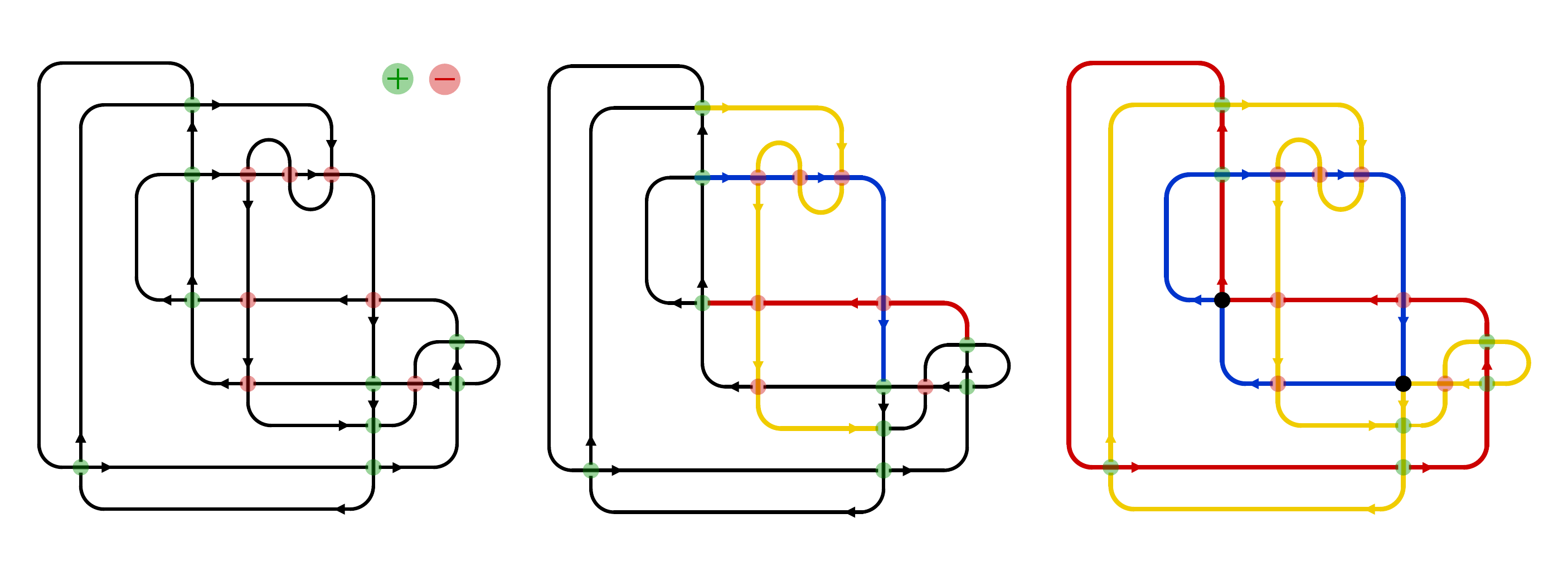}
     \put(62.5,15.75){\tiny $a$}
     \put(43.5,16){\tiny $b$}
     \put(43.5,26.25){\tiny $c$}
     \put(54.5,13.25){\tiny $d$}
     \put(43.25,30.5){\tiny $e$}
     \put(54.35,10.35){\tiny $f$}
	\end{overpic}
	\caption{The Ochiai I unknot, \cite[Figure 17]{burton2024hard}.}
	\label{fig:fig17}
\end{figure}

Now we consider the positive crossings $b$ and $c$. Note that it is not possible for both crossings to be switched, for otherwise the strand emanating north from $b$ to $c$ would be simultaneously red and blue. The same reasoning applied to the strand emanating west from $b$ to $c$ implies that at least one of these crossings must be switched. Thus, precisely one of $b$ or $c$ must be a switched crossing. 

On the right of \cref{fig:fig17} we consider the case where $b$ is switched and $c$ is unswitched. As a consequence, the strand from the overcrossing at $d$ to $b$ must be blue. But this forces a switch at crossing $d$, where both strands are blue. At this point, we have used two switches, and since the writhe of the diagram is $2$ there are no other switched crossings. We are then forced to color the rest of the diagram as indicated in the third panel of the figure, which yields a contradiction: there are intersections of the yellow color at unswitched crossings. The case where $c$ is switched and $b$ is unswitched is similar.

\vspace{2mm}
\textsc{Case 3}: The mirror of \cite[Figure 9]{burton2024hard}, the $D_{43}$ unknot.
\vspace{2mm}

The mirror of \cite[Figure 9]{burton2024hard} requires a similar but more involved ad hoc argument. The diagram is drawn on the left side of \cref{fig:fig9a}; for simplicity, we have suppressed explicit crossing information and have only indicated the orientation and whether the crossing is positive or negative, which is all that is needed for the argument. 

\begin{figure}[ht]
	\centering
	\begin{overpic}[scale=.45]{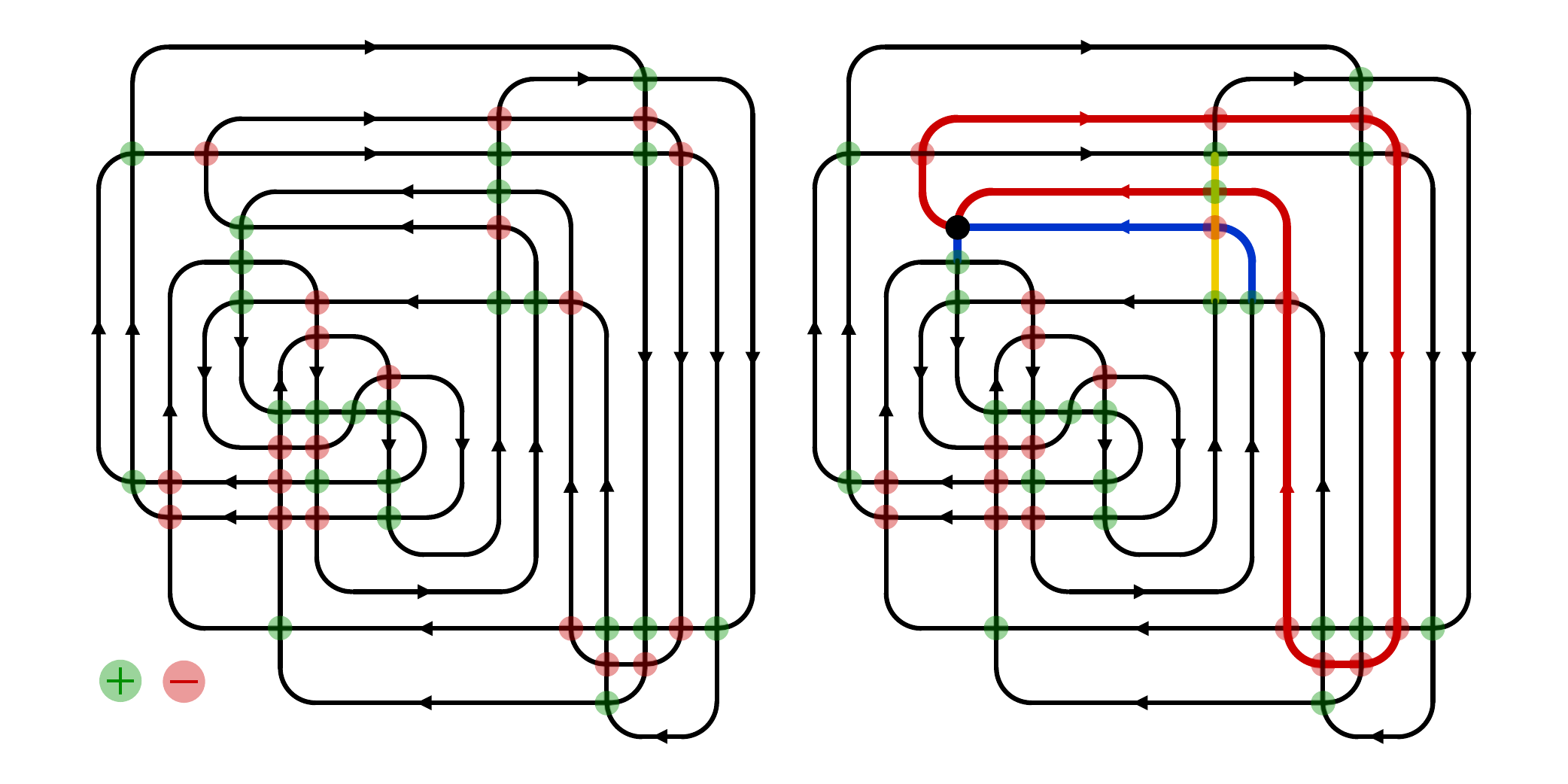}

     \put(59.5,34.2){\tiny $a$}
     \put(76,38.25){\tiny $b$}
     \put(80.5,31.5){\tiny $c$}
     \put(76,31.5){\tiny $d$}
     \put(76,41){\tiny $e$}
     
	\end{overpic}
	\caption{The mirror of the $D_{43}$ unknot, \cite[Figure 9]{burton2024hard}.}
	\label{fig:fig9a}
\end{figure}

Assume for the sake of contradiction that the diagram admits a normal ruling. Since the writhe is $3$, there are three switched crossings, all positive. Beginning at the positive crossing labeled $a$ in the right side of \cref{fig:fig9a}, the strand emanating to the west passes through a sequence of negative crossings, all necessarily unswitched, until it reaches the positive crossing labeled $b$. We color this stretch red. Note that $b$ cannot be a switch, for otherwise the orientation would force the red strand to turn north. To avoid a red-red unswitched intersection, the subsequent crossing needs to switch, sending the red strand east; the same argument holds for the next three crossings, resulting in $5$ switches. Thus, $b$ is unswitched. The strand from $b$ to $a$ is thus red, and we see that $a$ is necessarily a switched crossing. We color the other strands adjacent to $a$ blue, and note that the eastward segment extends through a negative crossing to the positive crossing labeled $c$. 

Now consider the segment beginning at the crossing labeled $d$ and heading north. This passes through both blue and red via unswitched crossings en route to the positive crossing labeled $e$, hence is necessarily a third color, yellow.

\begin{figure}[ht]
	\centering
	\begin{overpic}[scale=.45]{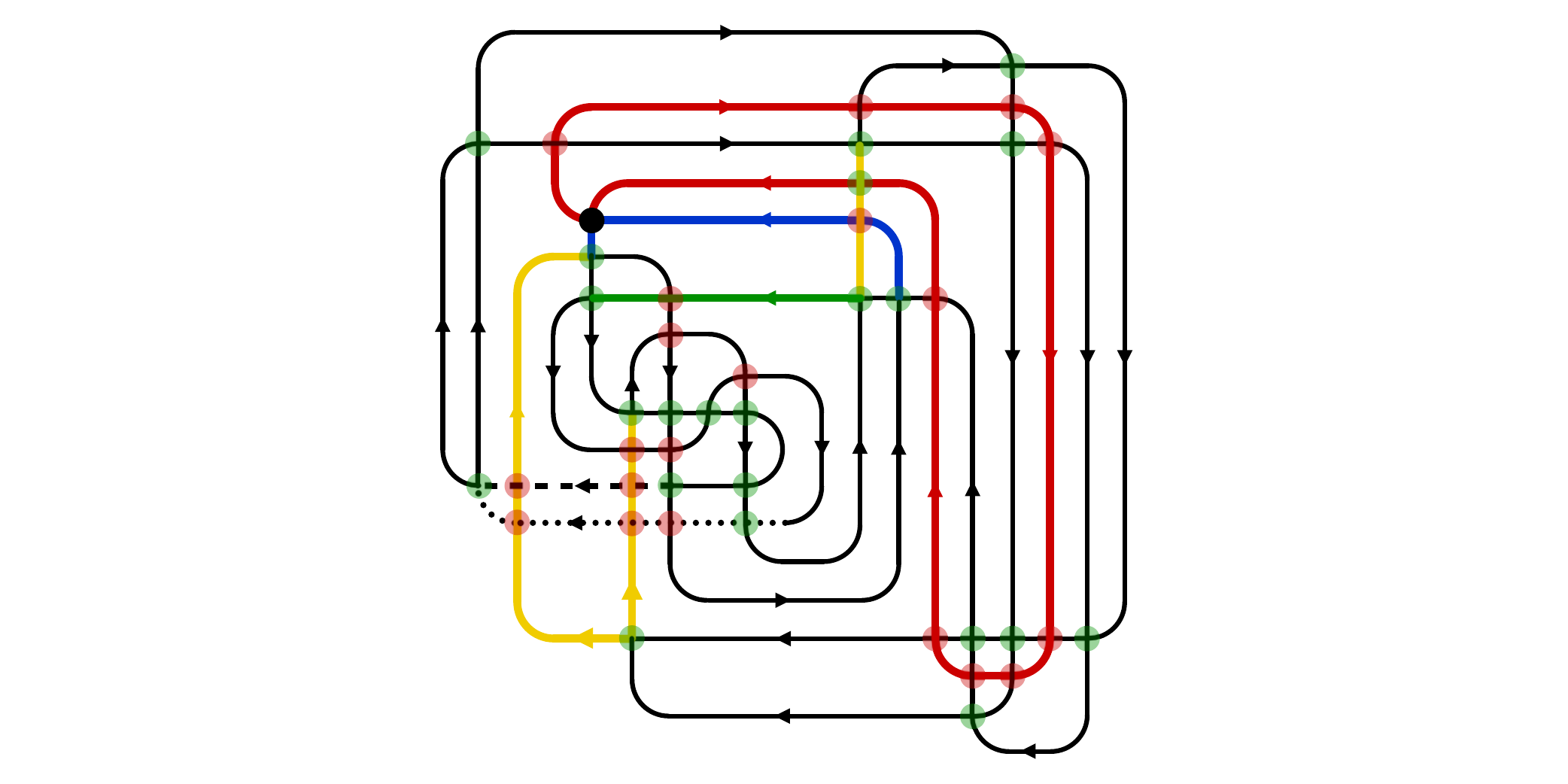}

     \put(36.5,37){\tiny $a$}
     \put(53.25,38.75){\tiny $b$}
     \put(56,31.75){\tiny $c$}
     \put(53.25,31.75){\tiny $d$}
     \put(53.25,41.25){\tiny $e$}
     \put(35.75,31.25){\tiny $f$}
     \put(38.55,10.5){\tiny $g$}
     \put(35.85,34.4){\tiny $h$}
     \put(29.25,19.75){\tiny $i$}
     \put(38.95,24.6){\tiny $j$}
     
	\end{overpic}
	\caption{The final contradiction ruling out \cite[Figure 9]{burton2024hard}.}
	\label{fig:fig9b}
\end{figure}

Next, in \cref{fig:fig9b}, consider the segment emanating to the west from crossing $d$ to the positive crossing labeled $f$. We claim that this segment must be the fourth and final color, green, of the assumed normal ruling. It cannot be red, as the entire red ruling disk is already drawn. It cannot be yellow, for the crossing at $d$ would result in a yellow-yellow intersection that cannot be orientably switched. If it were blue and $d$ were unswitched, then $c$ would be a blue-blue intersection which cannot be orientably switched. If it were blue and $d$ were switched, then the trigon involving $c,d$, and the negative crossing north of $d$ would violate the normality condition at a switch under any possible Legendrian realization of the diagram. Thus, the $d$-to-$f$ segment must be green. 

Now we turn our attention to the unicolored strand emanating west from the crossing labeled $g$, passing through two unswitched crossings, and terminating at crossing $h$. This strand cannot be blue, because $h$ would be a blue-blue intersection which cannot be orientably switched. It cannot be green, for if $h$ were unswitched there would be an unswitched green-green intersection to the right of $f$, and if $h$ were switched then $f$ would be a green-green intersection which cannot be orientably switched. Thus, the $g$-to-$h$ segment must be yellow. 

Finally, observe that the two horizontal segments (dashed and dotted in \cref{fig:fig9b}) which intersect this yellow segment must be green and blue, in some combination. Indeed, they cannot be the same color, for otherwise the positive crossing at $i$ would involve a same-color intersection which cannot be orientably switched. As the segment emanating north from $g$ and terminating at $j$ also intersects these green and and blue segments at unswitched crossings, it must be colored yellow. However, this produces our contradiction: at crossing $g$, there is a yellow-yellow intersection which cannot be orientably switched. 
\end{proof}

\section{Constructing hard Legendrian unknots}\label{sec:construct}

We identify a class of minimally-twisting Legendrian arcs which are smoothly hard relative to their endpoints. Precisely:

\begin{definition}\label{def:LBB}
A \emph{(left) building block} is a front projection $\mathcal{F}_\ell\subset (-\infty,0]_x \times \R_z$ such that
\begin{enumerate}
    \item $\mathcal{F}_\ell \cap \{x = 0\}$ consists of two points,\label{part:LBB1}
    \item the front projection $\hat{\mathcal{F}}_\ell \subset \R^2$ obtained by closing $\mathcal{F}_\ell \cap \{x = 0\}$ with a right cusp represents a max-tb unknot,\label{part:LBB2} 
    \item the only good spherical Reidemeister move available in $\hat{D}_\ell$, the smoothing of $\hat{\mathcal{F}}_\ell$, is a RII$^-$ move corresponding to a region bounded by said right cusp, and\label{part:LBB3} 
    \item the exhaustion of RIII moves in $\hat{D}_\ell$ reveals no other good Reidemeister moves and additionally leaves the region corresponding to the RII$^-$ move unaffected.\label{part:LBB4} 
\end{enumerate}
A \emph{(right) building block} is defined analogously for a front in $[0,\infty)\times \R$. We call a building block \emph{strong} if the following additional property holds: 
\begin{enumerate}
    \item[(5)] After introducing an additional positive crossing near the endpoints at $x=0$ (see \cref{fig:blocksproof2}, for example), the created smooth RIII move and its exhaustion leads to no new good Reidemeister moves.                     \end{enumerate}
\end{definition}

\begin{lemma}\label{lemma:blocks}
The front projections in \cref{fig:blocks} are strong building blocks of writhe $2$.      
\end{lemma}

\begin{figure}[ht]
	\centering
	\begin{overpic}[scale=.32]{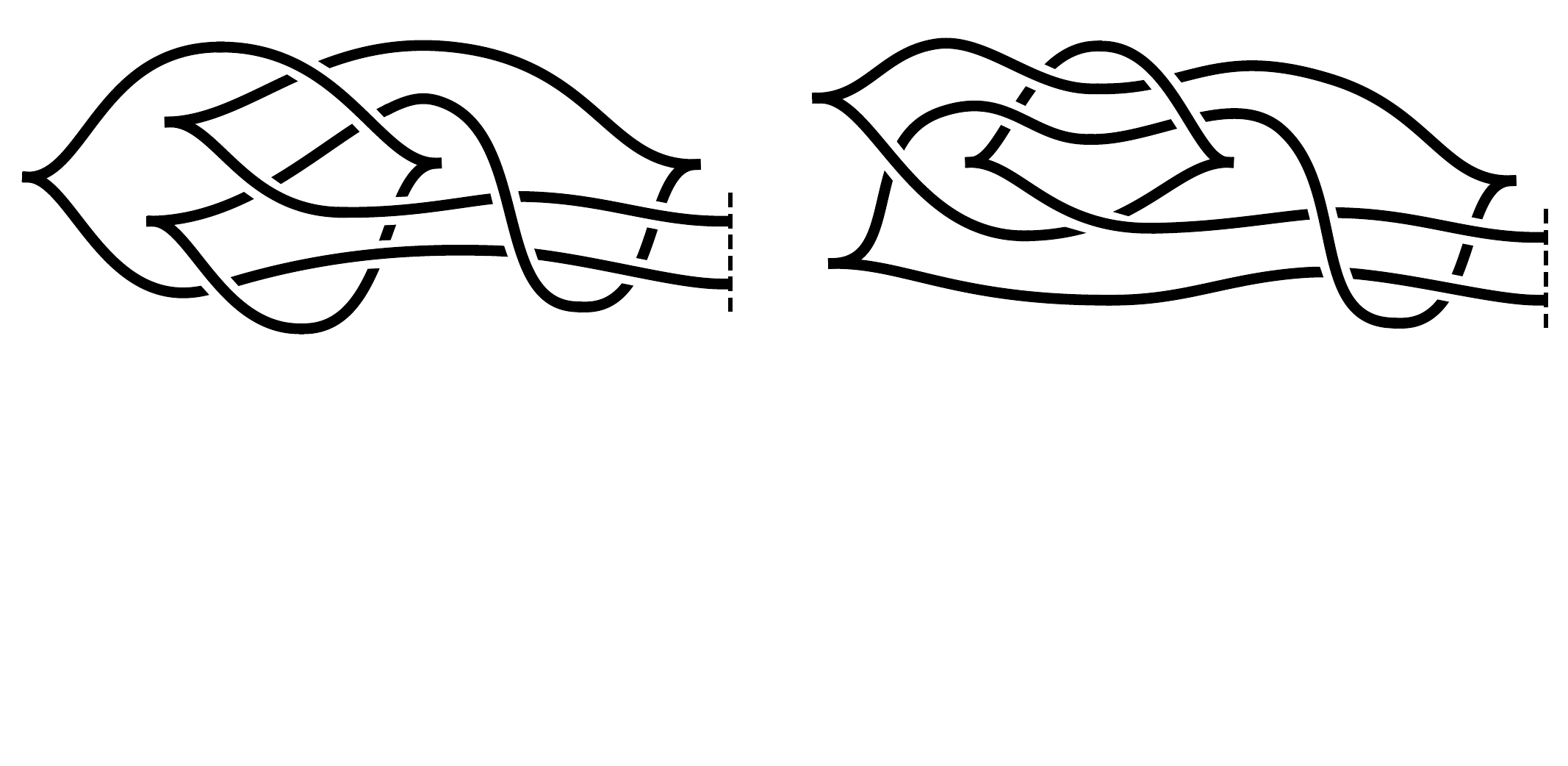}
    
	\end{overpic}
    \vskip-3cm
	\caption{Two (left, strong) building blocks of writhe $2$.}
	\label{fig:blocks}
\end{figure}

\begin{proof}
That each front admits a right-cusp closure to the max-tb unknot is immediate (see the top of \cref{fig:blocksproof}), as is the observation that each smoothing has only one available good Reidemeister (RII$^-$) move corresponding to the right-cusp region highlighted in yellow in the middle row of \cref{fig:blocksproof}. These observations give \eqref{part:LBB1}, \eqref{part:LBB2}, and \eqref{part:LBB3} of \cref{def:LBB}. It remains to verify \eqref{part:LBB4}. To that end, we have also highlighted (in blue and green, and blue, respectively) the available RIII moves for each smoothed building block. On the left, performing each move (blue or green) only introduces the corresponding inverse RIII move (red); moreover, neither move interacts with the yellow region. On the right, the blue RIII introduces its inverse (red) and a new RIII (green); performing the green move leads to a dead end, and again the yellow region is untouched. Thus, for both blocks, \eqref{part:LBB4} holds.

\begin{figure}[ht]
	\centering
	\begin{overpic}[scale=.44]{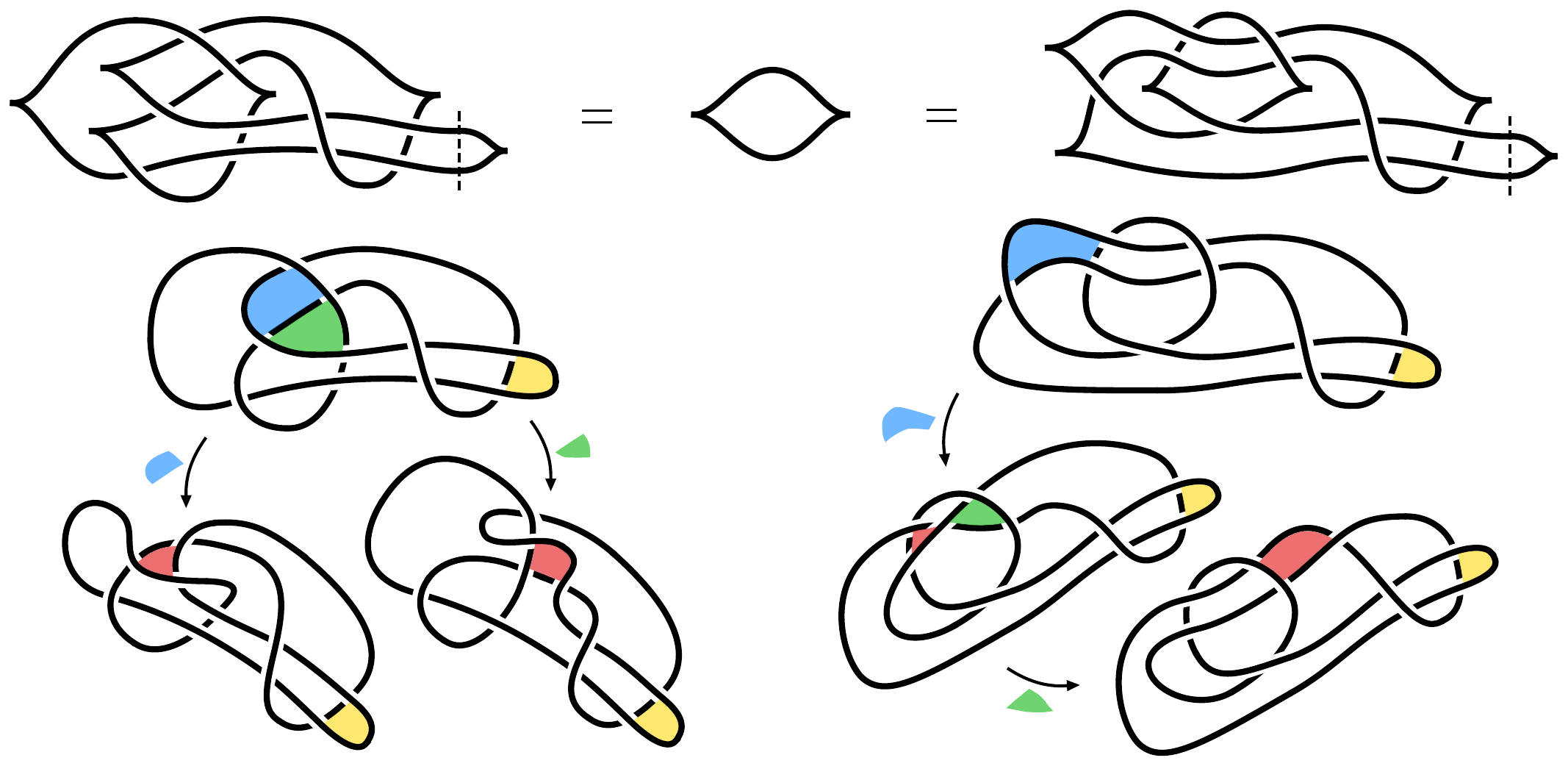}
    
	\end{overpic}
	\caption{The building block proof in \cref{lemma:blocks}.}
	\label{fig:blocksproof}
\end{figure}

The verification of strongness boils down to an exhaustion of the Reidemeister graph of the diagram with the additional positive crossing; we verify this explicitly for the second front, and leave the slightly more tedious check of the first front to the reader.\footnote{No generality of our results is lost by only considering the second front, for which we explicitly verify strongness; we have simply provided two examples for the sake of variety.}

Fixing the already present blue RIII move, the newly created RIII move leads to a linear chain of RIII moves which terminates; this is depicted in \cref{fig:blocksproof2}. One may check that at no point the linear chain of RIII moves beginning with the blue region (as shown on the right side of \cref{fig:blocksproof}) admits a nontrivial interaction with the new linear chain of RIII moves. Therefore, the indicated building block is strong.  
\end{proof}

\begin{figure}[ht]
	\centering
	\begin{overpic}[scale=.44]{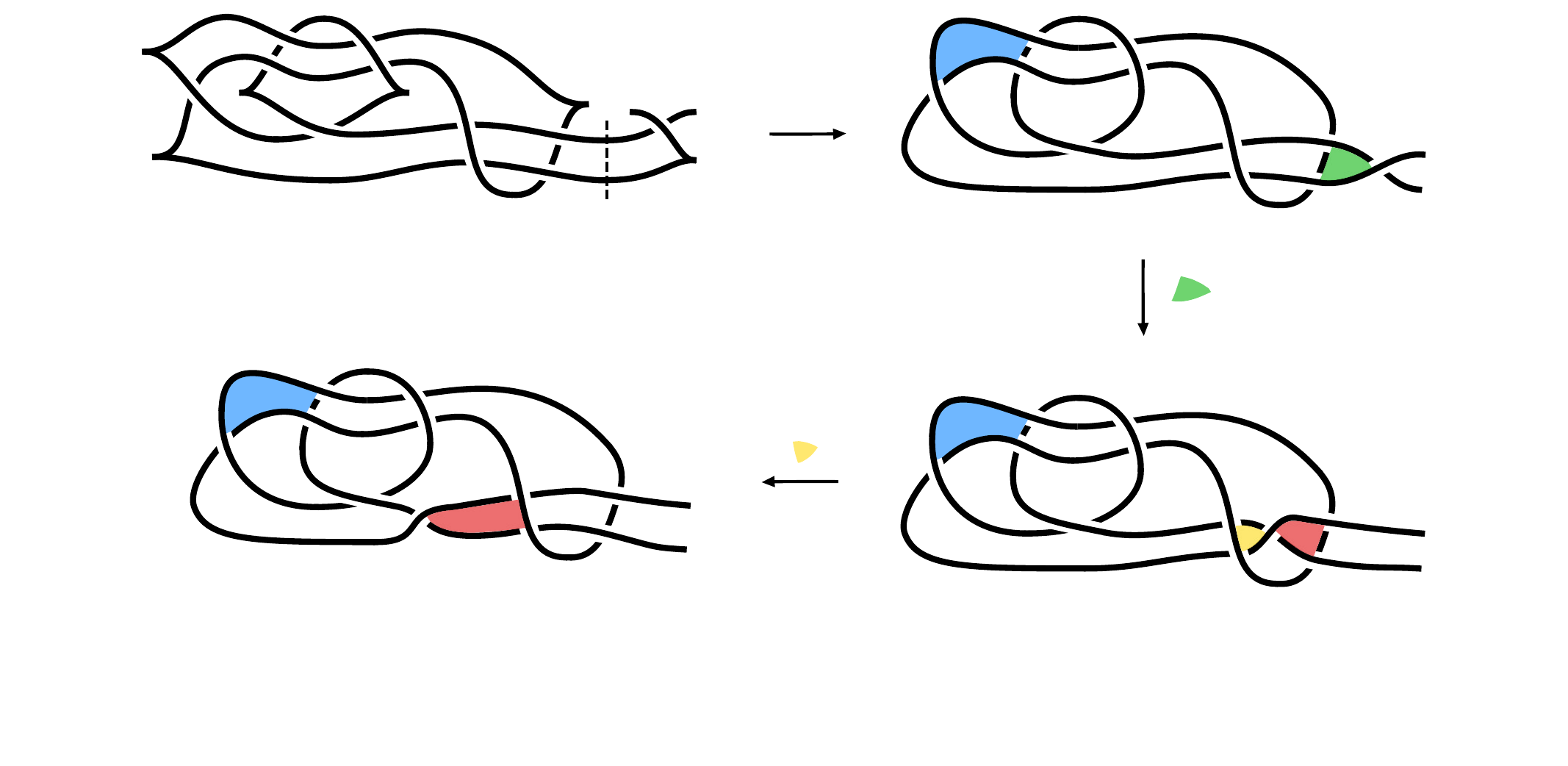}
    
	\end{overpic}
    \vskip-1.5cm
	\caption{Proving strongness in \cref{lemma:blocks}.}
	\label{fig:blocksproof2}
\end{figure}

\begin{remark}
Given the examples in \cref{fig:blocks}, it is straightforward to concoct building blocks with larger complexity, either in crossing number and/or writhe, by weaving eyes in and out of each other to create additional clasps. See \cref{fig:blocks-infinite} for two example families, one which remains minimal in writhe and another which increases in writhe. The blocks in \cref{fig:blocks} are variations on the simplest examples (in both crossing number and writhe) the authors were able to construct. In fact, by \cref{prop:writhe1obstruct} and \cref{prop:recipe} below, it follows that the writhe of a building block must be at least $2$, hence these examples are writhe-minimizing.
\end{remark}

\begin{figure}[ht]
	\centering
	\begin{overpic}[scale=.32]{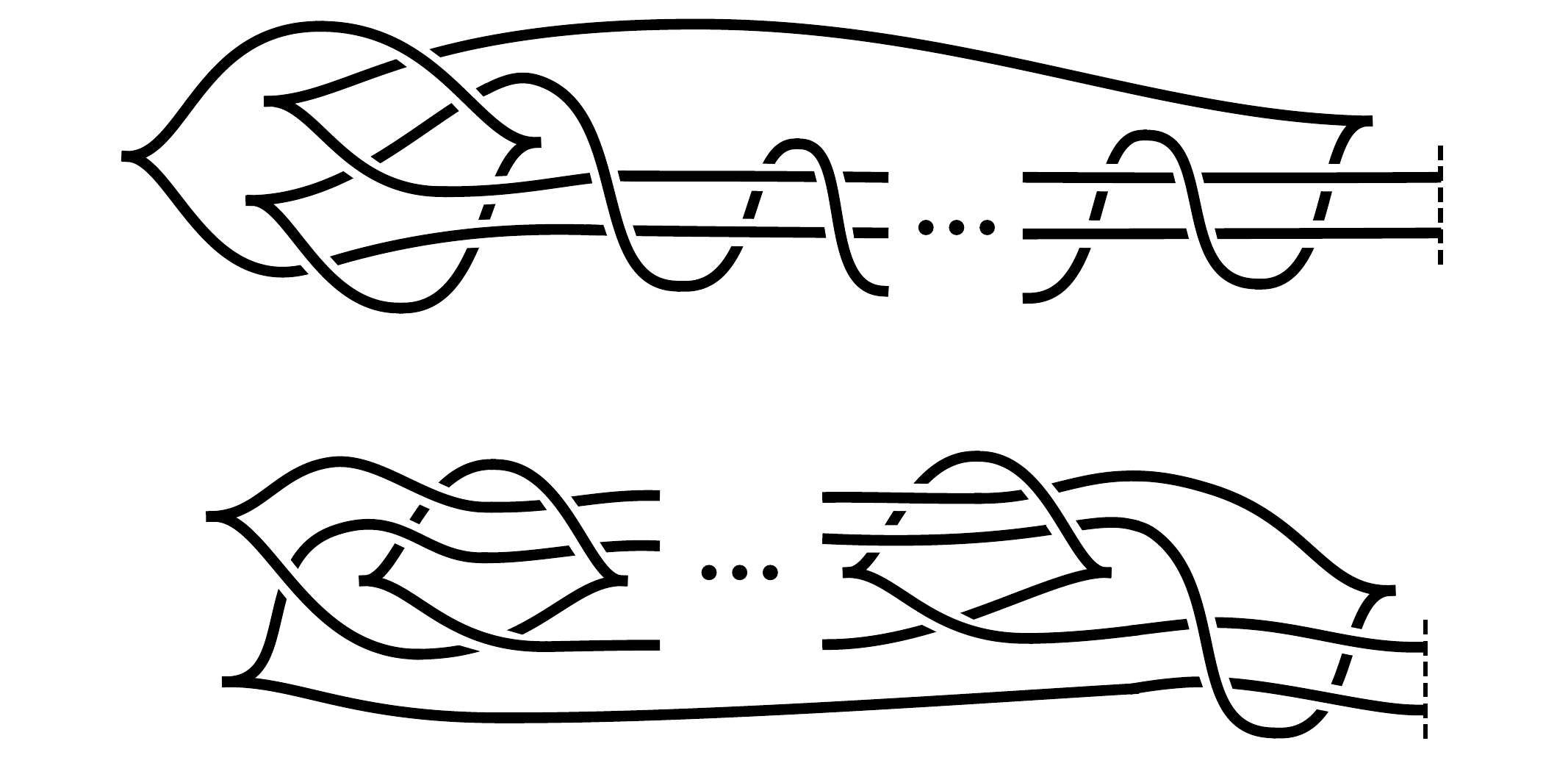}
    
	\end{overpic}
	\caption{On the top, an infinite family of building blocks of writhe $2$. On the bottom an infinite family of building blocks with increasing writhe.}
	\label{fig:blocks-infinite}
\end{figure}

\begin{proposition}\label{prop:recipe}
Given left and right building blocks $\mathcal{F}_\ell$ and $\mathcal{F}_r$, the front projection obtained by their sum as on the right side of \cref{fig:recipe} is a smoothly spherically hard max-tb unknot of writhe $\mathrm{wr}(\mathcal{F}_\ell) + \mathrm{wr}(\mathcal{F}_r)$. Given a strong building block $\mathcal{F}_\ell$, the front projection obtained as on the left side of \cref{fig:recipe} is a smoothly planar hard max-tb unknot of writhe $\mathrm{wr}(\mathcal{F}_\ell) + 1$.
\end{proposition}

\begin{figure}[ht]
	\centering
	\begin{overpic}[scale=.35]{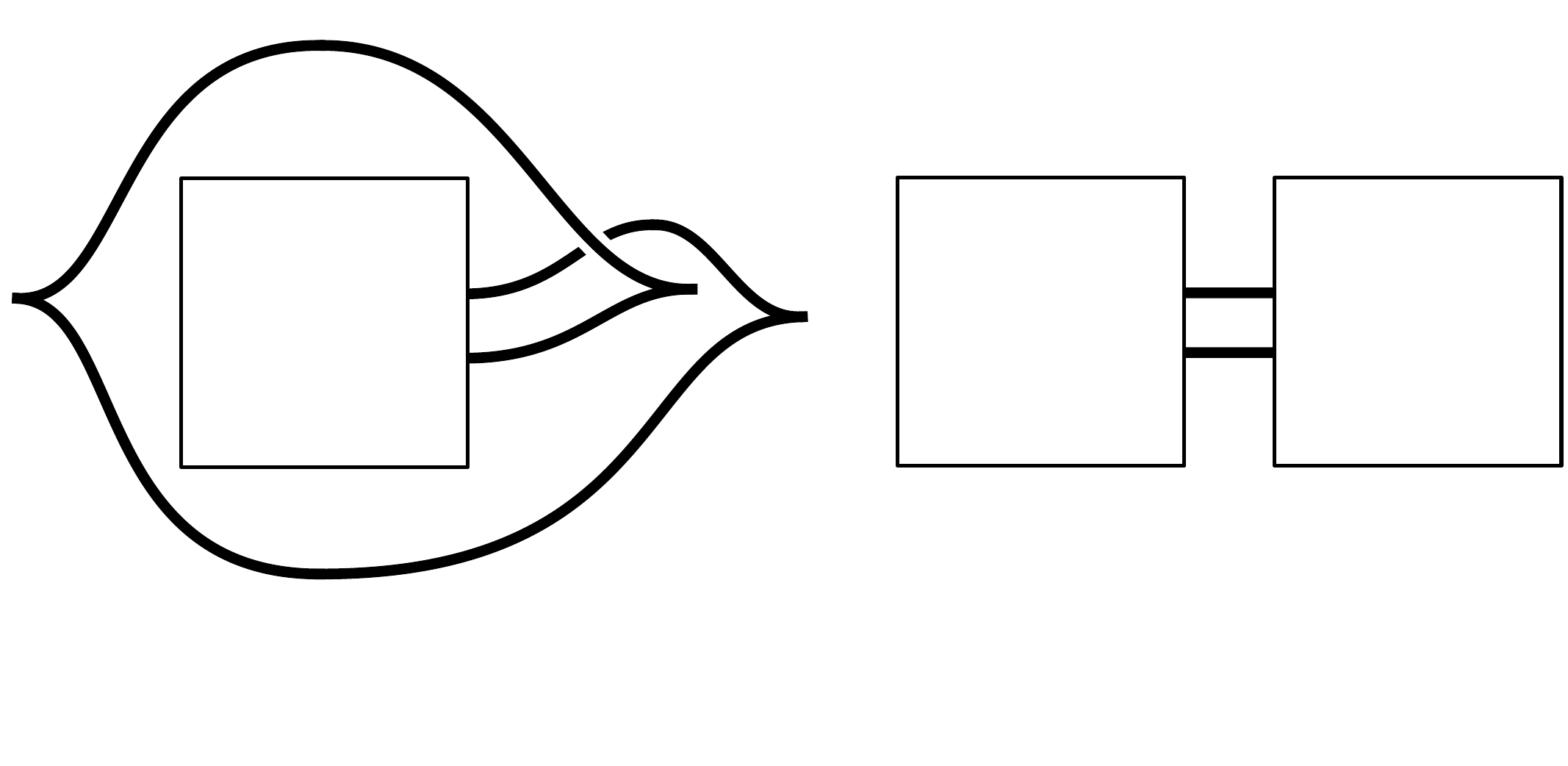}
    \put(19,27){$\mathcal{F}_\ell$}
    \put(64.5,27){$\mathcal{F}_\ell$}
    \put(88.5,27){$\mathcal{F}_r$}
	\end{overpic}
    \vskip-1cm
	\caption{On the left, a planar hard max-tb unknot of writhe $\mathrm{wr}(\mathcal{F}_\ell) + 1$. On the right, a spherically hard max-tb unknot of writhe $\mathrm{wr}(\mathcal{F}_\ell) + \mathrm{wr}(\mathcal{F}_r)$.}
	\label{fig:recipe}
\end{figure}

\begin{proof}
First, we consider the spherical recipe on the right. By definition, $\hat{\mathcal{F}}_\ell$ and $\hat{\mathcal{F}}_r$ are max-tb unknots; the composite front is obtained by performing a Legendrian surgery (smoothly, a connect sum) on split copies $\hat{\mathcal{F}}_\ell \sqcup \hat{\mathcal{F}}_r$ of the fronts. Being a split sum of unknots, the result is an unknot. The Thurston-Bennequin invariant is maximized because the Legendrian surgery on fillable knots induces a Lagrangian filling of the composite \cite{chantraine2010concordance,ekholm2012exactcobordisms}. The Legendrian surgery also destroys both available good RII$^-$ present in the smoothing $\hat{D}_\ell \sqcup \hat{D}_r$ and thus the composite diagram has no good Reidemeister moves. Moreover, by \eqref{part:LBB4} of \cref{def:LBB}, the available RIII moves on either side of the diagram sum do not interact and thus lead to no emergent good moves; it follows that the diagram is smoothly spherically hard. The writhe computation follows from the fact that writhe is additive under connect sum.

Next, we consider the planar recipe on the left. The resulting front is smoothly an unknot because of the spherical RIII move off to infinity, and the Thurston-Bennequin calculation is immediate. It only remains to verify that the diagram is smoothly hard on the plane. This is where the assumption of strongness enters: the additional positive crossing in the diagram creates a smooth RIII move. However, by assumption, the exhaustion of this RIII move leads to no good Reidemeister moves on the plane. Thus, the diagram is smoothly hard on the plane. 
\end{proof}

\begin{example}
The front in \cref{fig:writhe4hard} is given by applying \cref{prop:recipe} to the leftmost building block in \cref{fig:blocks} and (the rotation of) the rightmost building block in the same figure. The front is therefore smoothly spherically hard. 
\end{example}

\section{Impostors}\label{sec:impostors}
Given the relatively small number of cases which arise in proving \Cref{prop:writhe1obstruct}, it is tempting to apply the approach of that proof to a writhe-2 front projection $\mathcal{F}$ of the max-tb unknot.  In the present section we disabuse ourselves of this temptation by constructing infinitely many \emph{impostors} --- writhe-2 front projections of Legendrian knots which are not the max-tb unknot, but whose classical invariants and ruling polynomial agree with those of the max-tb unknot.  While the only structural feature of $\mathcal{F}$ needed to prove \Cref{prop:writhe1obstruct} was the uniqueness of its ruling, these examples demonstrate that additional structure will be necessary in analyzing the writhe-2 case.

\begin{figure}
\centering
\includegraphics[scale=0.25]{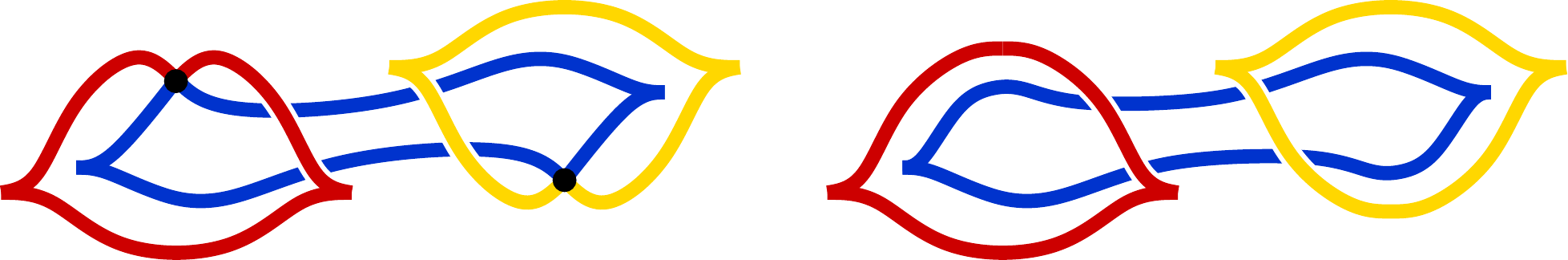}
\caption{The 0-resolution of a front projection of the standard Legendrian unknot with respect to its unique ruling.}
\label{fig:0-resolution}
\end{figure}

\subsection{Writhe 2 diagrams of the standard unknot}
We explain the origin of our impostors by first making an observation about 0-resolutions of writhe-2 front projections of the max-tb unknot.  Given any ruling $R$ of a front projection $\mathcal{\mathcal{F}}$ (of any Legendrian knot in $(\mathbb{R}^3,\xi_{\mathrm{std}})$), we may denote by $\Lambda_0(\mathcal{F},R)$ the front projection obtained by performing the 0-resolution at each switch of $R$.  See \Cref{fig:0-resolution}.  When the ruling $R$ is understood (e.g., when $\mathcal{F}$ admits a unique ruling), we will write $\Lambda_0(\mathcal{F})$.

\begin{proposition}\label{prop:writhe2pairwiseunlinked}
Let $\mathcal{F}$ be a front projection of the max-tb unknot with $\mathrm{writhe}(\mathcal{F}) = 2$. Then there is a sequence of Legendrian RIII moves which, when applied to $\mathcal{F}$, produces a front projection $\mathcal{F}'$ with the property that $\Lambda_0(\mathcal{F}')$ consists of three pairwise unlinked Legendrian eyes.
\end{proposition}

\begin{proof}
Because $\mathcal{F}$ is a front projection of the max-tb unknot, it admits a unique ruling; because $\mathrm{writhe}(\mathcal{F}) = 2$, the resolution $\Lambda_0(\mathcal{F})$ of this ruling consists of three Legendrian eyes, which we denote $B$, $R$, and $Y$.  We will show that the components $B$ and $R$ must be unlinked in the absence of $Y$, possibly after first applying some Legendrian RIII moves to $\mathcal{F}$.  We will make no assumption about the existence of a switch between these components, and thus the argument will apply to any pair of the three components.

Even without a switch between $B$ and $R$, we may find a portion $\{x_0-\epsilon\leq x\leq x_0+\epsilon\}$ of the front projection in which the regions bounded by $B$ and $R$ are either nested or disjoint.  For instance, such a region may be found near the right cusps of $\mathcal{F}$.  We can now proceed to apply the casework of \Cref{prop:writhe1obstruct}, except that ruling out the clasp subregions depicted in the top row of \cref{fig:Obs2} may first require the application of Legendrian RIII moves.

\begin{figure}
\centering
\includegraphics[scale=0.25]{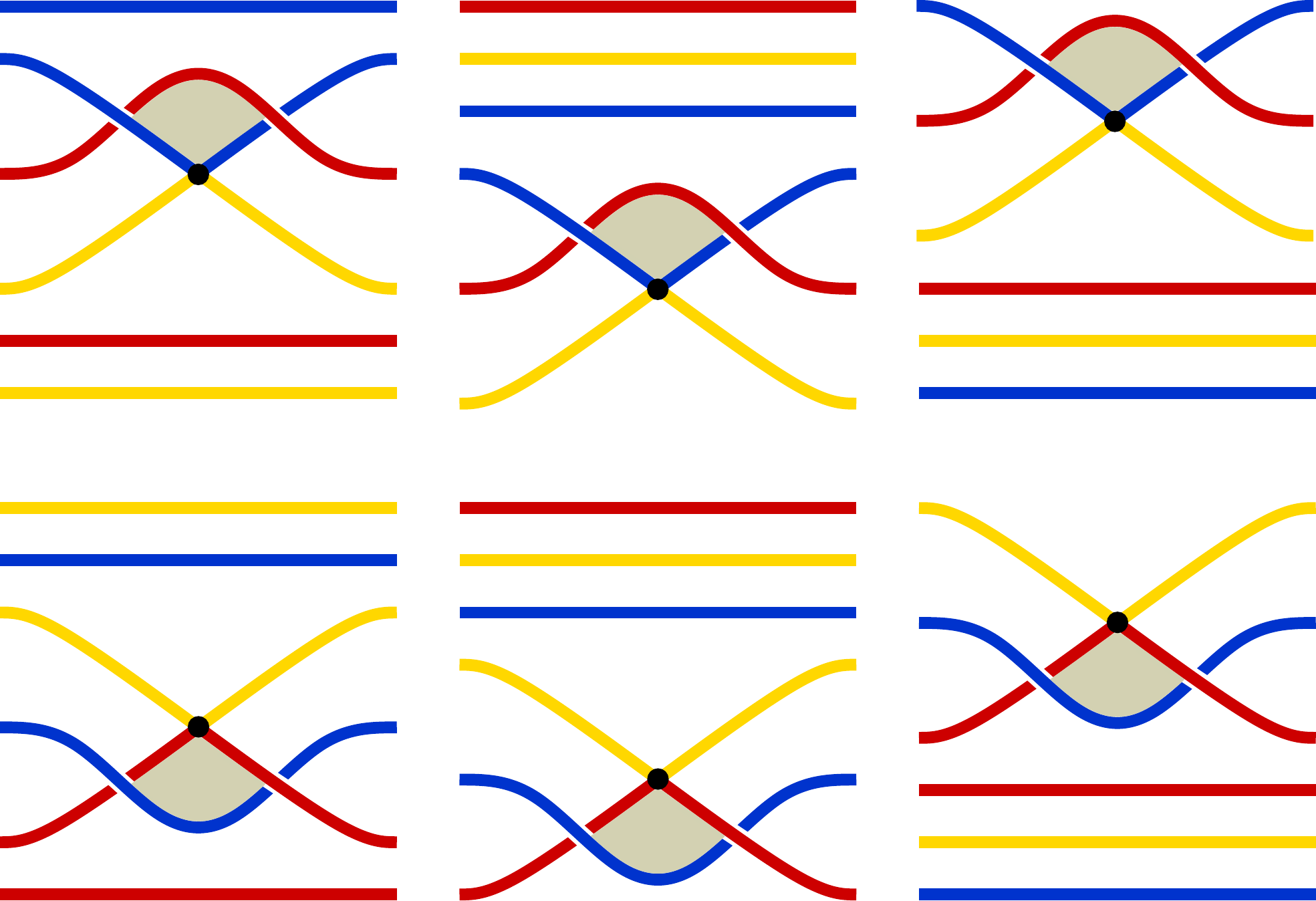}
\caption{The three configurations of $B$ and $R$ prohibited by \cref{fig:Obs1} fail to lead to an extraneous ruling if $Y$ meets the clasp region at a switch.}
\label{fig:switchedClaspRegions}
\end{figure}

Specifically, a clasp subregion $r_c$ may fail to lead to an extraneous ruling if the third component $Y$ shares a switch with $B$ or $R$ along the boundary of $r_c$.  In this case, we say that $r_c$ is a \emph{switched} clasp subregion.  See \Cref{fig:switchedClaspRegions}.  For instance, suppose $Y$ and $B$ share a switch along the boundary of $r_c$, as depicted in \Cref{fig:switchedClaspRegionReidemeister}.  In this case, we may destroy the subregion $r_c$ by applying a Legendrian RIII move which slides $R$ across the switch.  See the top row of \Cref{fig:switchedClaspRegionReidemeister}.  Notice that this modification creates additional crossings between $R$ and $Y$, but does not create a prohibited clasp subregion between these components.  If a prohibited clasp subregion between $R$ and $Y$ were created by this move, then the two crossings of $R$ and $B$ which occur along the boundary of $r_c$ could have been made into switches, providing an extraneous ruling of $\mathcal{F}$.  See the bottom row of \Cref{fig:switchedClaspRegionReidemeister}.

\begin{figure}
\centering
\includegraphics[scale=0.25]{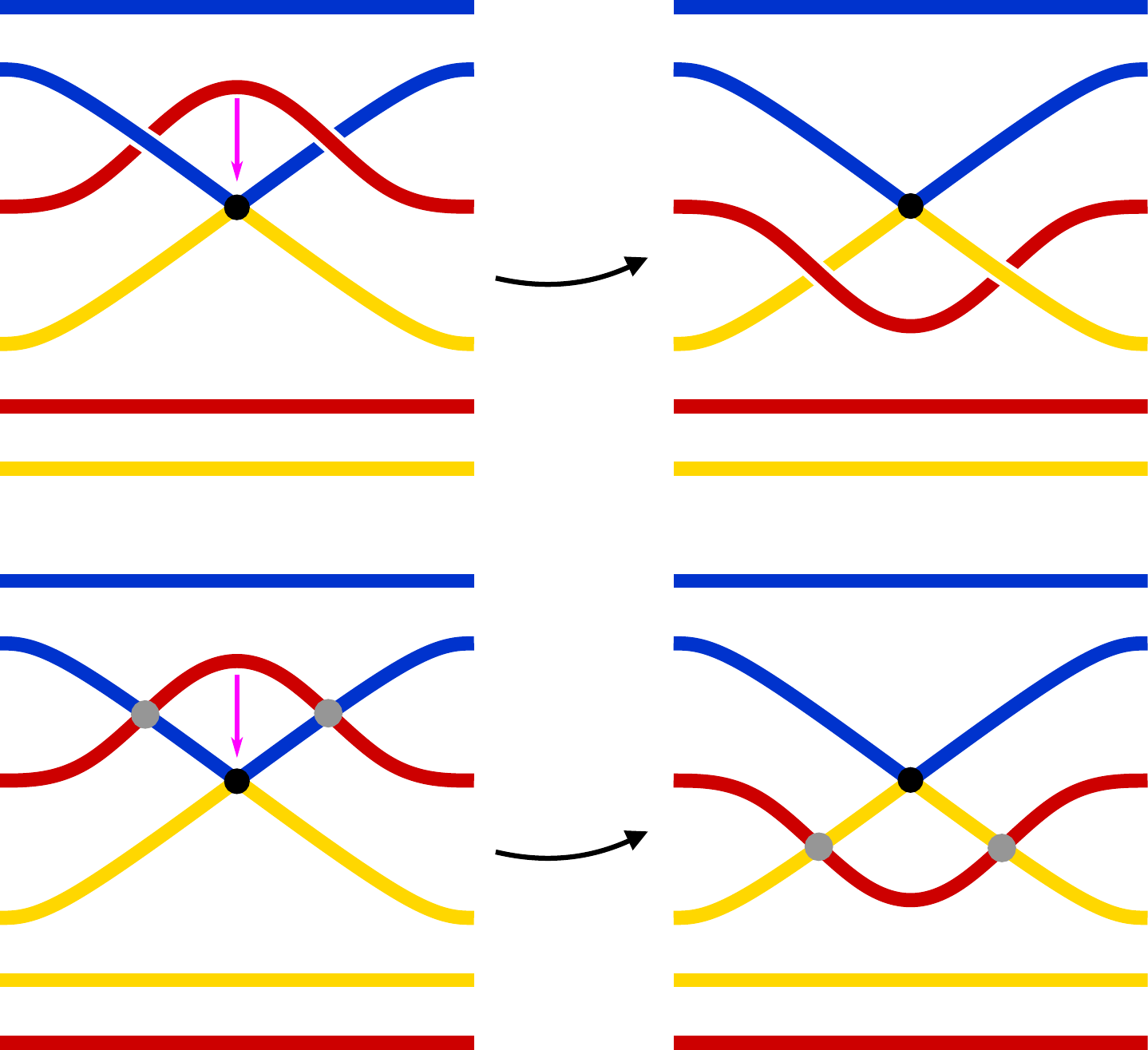}
\caption{By applying a Legendrian RIII we may eliminate the clasp regions seen in \Cref{fig:switchedClaspRegions}.  This will create additional crossings between $R$ and $Y$, but does not create prohibited clasps, lest $\mathcal{F}$ admit an extraneous ruling.}
\label{fig:switchedClaspRegionReidemeister}
\end{figure}

Following the elimination of all clasp subregions, $B$ and $R$ admit an initial region analogous to the switch region $\mathcal{F}_{\mathrm{sw}}$ used in \Cref{prop:writhe1obstruct}.  We may now apply the argument of \Cref{prop:writhe1obstruct} to the link $B\cup R$ in the absence of $Y$ to see that $B$ and $R$ are unlinked.  Because the Legendrian Reidemeister moves used to eliminate switched clasp regions between $B$ and $R$ do not create additional clasp regions, the same argument applies to the pairs $B,Y$ and $R,Y$.
\end{proof}

\Cref{prop:writhe2pairwiseunlinked} tells us that the Legendrian eyes $B$, $R$, and $Y$ of $\Lambda_0(\mathcal{F})$ form either an unlink or a Brunnian link.  Mohnke has shown in \cite{mohnke2001legendrian} that it is not possible to form the Borromean rings out of three max-tb Legendrian unknots, regardless of their front projections.  On the other hand, there exist non-trivial, $n$-component Brunnian links whose components are all max-tb Legendrian unknots, for any $n\geq 2$. See \cref{fig:brunnian} for an example when $n=3$.

\begin{figure}[ht]
	\begin{overpic}[scale=.4]{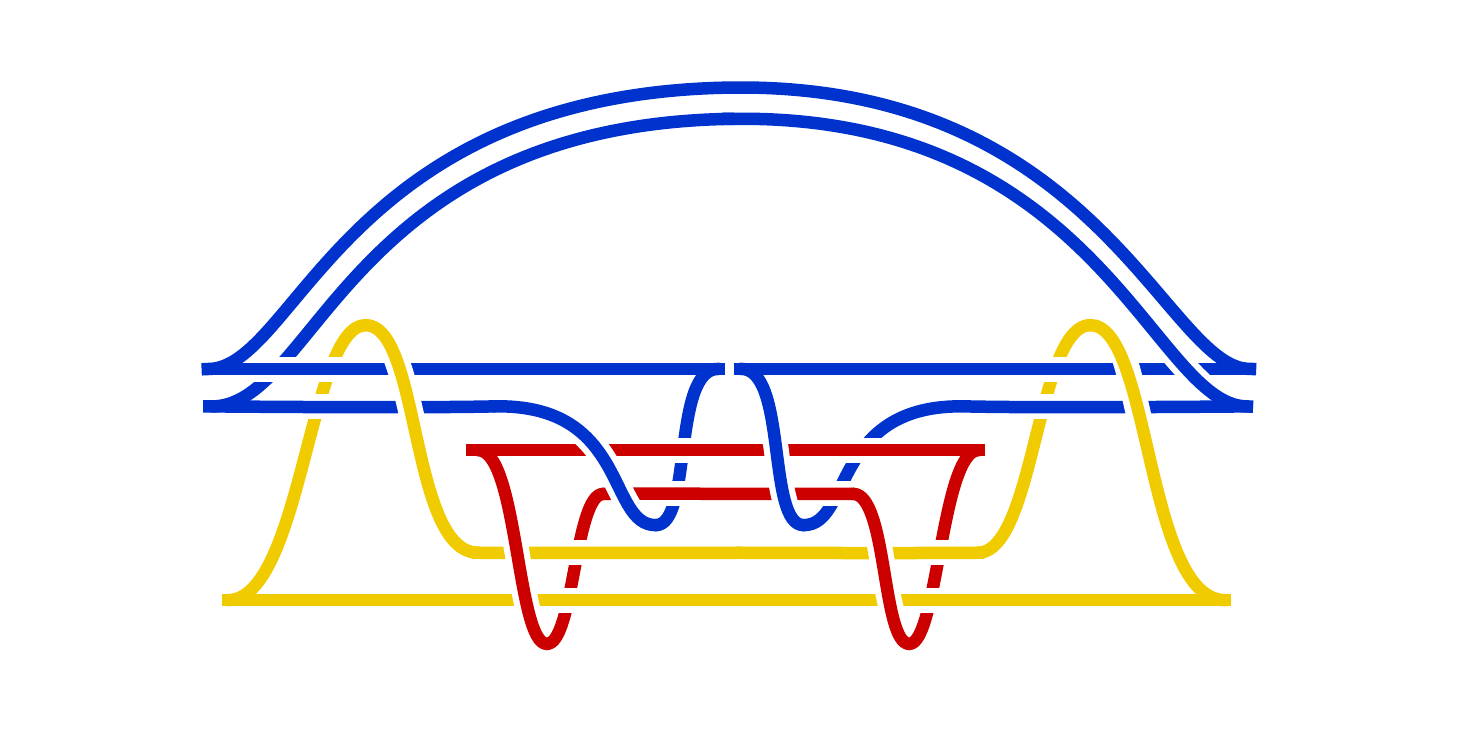}   
    
	\end{overpic}
	\caption{A $3$-component Brunnian link.}
	\label{fig:brunnian}
\end{figure}

Given our requirement that $B$, $R$, and $Y$ be eyes, we have the following question:
\begin{question}\label{question:brunnian-eyes}
Does there exist a Legendrian realization of a Brunnian link whose front projection consists of Legendrian eyes?
\end{question}
The authors suspect that the answer to this question is no, and thus that $B\sqcup R\sqcup Y$ is an unlink.  The 3-component case could plausibly be answered with some tedious casework in the front projection, but the impostors we produce in the next subsection undermine the usefulness of knowing $B\sqcup R\sqcup Y$ to be an unlink, so we do not pursue this casework here.

\subsection{Constructing impostors}

\Cref{prop:writhe2pairwiseunlinked} provides a recipe for constructing all front projections of the max-tb Legendrian unknot with writhe 2: join three pairwise-unlinked Legendrian eyes in $\mathbb{R}^2_{xz}$ into a single component by adding two switched crossings, and then perform arbitrary Legendrian RIII moves.  Note that the RIII moves have no effect on hardness.  One can resolve the question of writhe-2 hardness by either showing that this recipe always results in a diagram which admits a simplifying Reidemeister move or by using the recipe to construct a hard diagram for the standard unknot.  We have done neither.

\begin{figure}
\centering
\includegraphics[scale=0.25]{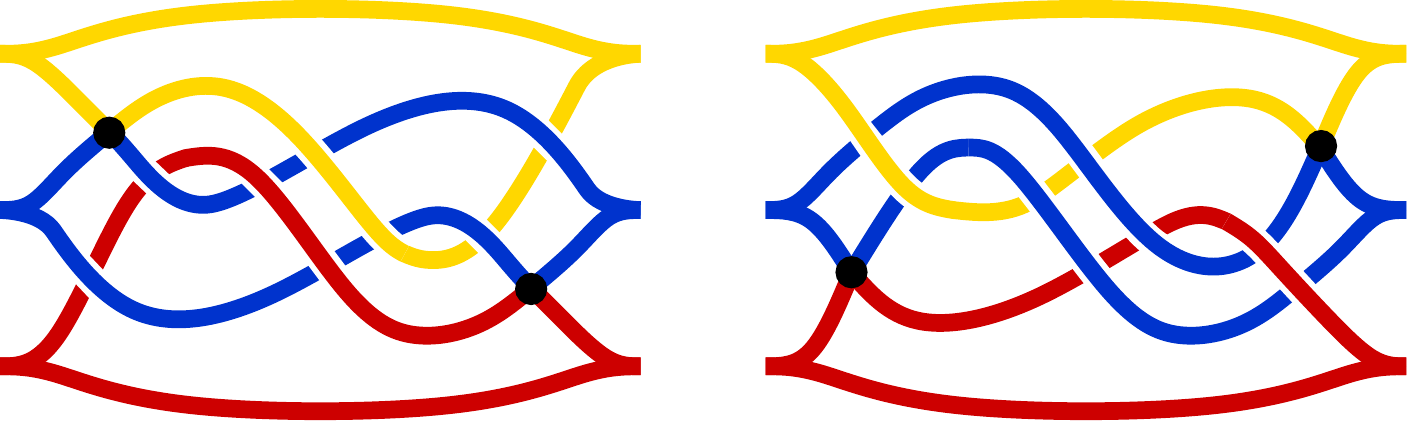}
\caption{A front projection of $m(9_{46})$ which admits no simplifying Legendrian Reidemeister moves.}
\label{fig:m946}
\end{figure}

The essential problem with this strategy is that, while every writhe-2 diagram of the standard Legendrian unknot results from the algorithm described, the algorithm can also produce nontrivial knot types.  Consider \Cref{fig:m946}, which exhibits a front projection $\mathcal{F}$ of a Legendrian knot of smooth type $m(9_{46})$.  As depicted, this front projection admits rulings $R_1$ and $R_2$ so that the 0-resolutions $\Lambda_0(\mathcal{F},R_1)$ and $\Lambda_0(\mathcal{F},R_2)$ are both unlinks of three Legendrian eyes.

While the front projection seen in \Cref{fig:m946} disrupts our initial strategy, recall that the standard Legendrian unknot has ruling invariant 1.  So one may revise the strategy and hope to show that any \emph{uniquely-ruled} front projection resulting from our recipe must admit a simplifying Legendrian RII move.  This hope is quickly dashed by the front projection seen on the left of \Cref{fig:pretzel-family}.  This is a front projection for a Legendrian knot $\Lambda$ of smooth type $11n_{139}$, with classical invariants $\mathrm{tb}(\Lambda)=-1$ and $\mathrm{r}(\Lambda)=0$.  Rutherford has shown in \cite{rutherford2006thurston} that the (ungraded) ruling polynomial of a Legendrian knot is determined by its smooth isotopy class and classical invariants, and Cornwell-Ng-Sivek have exhibited in \cite{cornwell2016concordance} a Legendrian $11n_{139}$ with these classical invariants whose ruling polynomial is 1.  It follows that the Legendrian $11n_{139}$ depicted in \Cref{fig:pretzel-family} has a unique ruling, and is thus an impostor.

\begin{figure}
\centering
\includegraphics[scale=0.25]{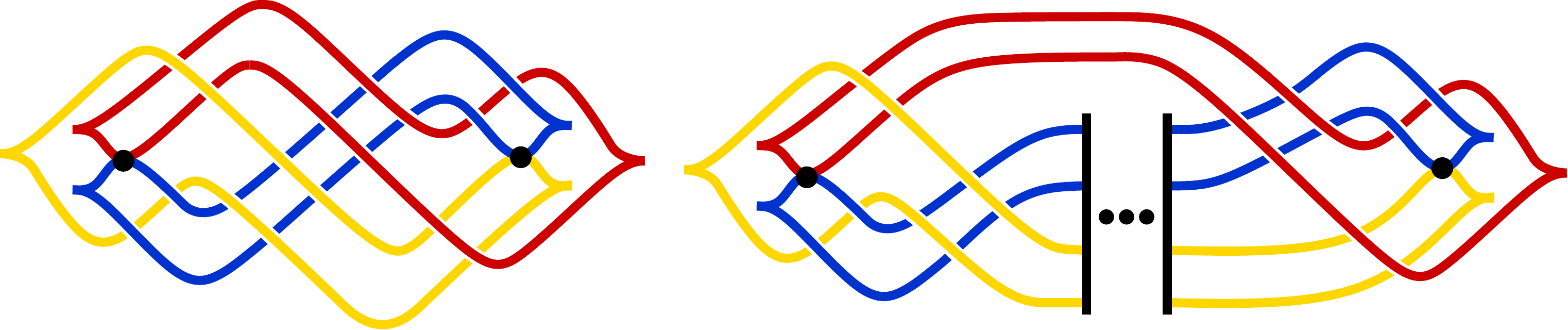}
\caption{On the left, a front projection of $11n_{139}$ which admits no simplifying Legendrian Reidemeister moves.  Twisting the blue and yellow bands around each other $m$ times produces a front projection of $P(-2m-5,-3,3)$.}
\label{fig:pretzel-family}
\end{figure}

In fact, by twisting the blue and yellow bands in our front projection of $11n_{139}$, we obtain infinitely many front projections which have writhe 2, admit a unique ruling, and whose 0-resolution under this ruling is a 3-component unlink of Legendrian eyes.  Indeed, the writhe and 0-resolution of these examples are evident in \Cref{fig:pretzel-family}, and we will compute the ruling polynomials of these fronts by determining their smooth type.

\begin{proposition}\label{prop:pretzels}
The Legendrian knot whose front is given in the right panel of \Cref{fig:pretzel-family}, which includes $m$ full twists of the blue and yellow bands of $11n_{139}$, is smoothly isotopic to the pretzel knot $P(-2m-5,-3,3)$.
\end{proposition}

\begin{proof}
\Cref{fig:pretzel-isotopy-1,fig:pretzel-isotopy-2,fig:pretzel-isotopy-3} exhibit a smooth isotopy from our diagram to the pretzel knot $P(-2m-5,-3,3)$.  Specifically, the diagram obtained by smoothing the front in \Cref{fig:pretzel-family} is seen in the leftmost panel of \Cref{fig:pretzel-isotopy-1}.  By reflecting the blue and red arcs of that diagram, we obtain a diagram with 6 fewer crossings, as seen in the middle panel of \Cref{fig:pretzel-isotopy-1}.  Next, we push the three crossings highlighted in the middle panel along the indicated arcs, obtaining the rightmost panel of \Cref{fig:pretzel-isotopy-1}.  The next stage of our isotopy focuses on the tangle highlighted in that panel.

\begin{figure}
\centering
\includegraphics[scale=0.24]{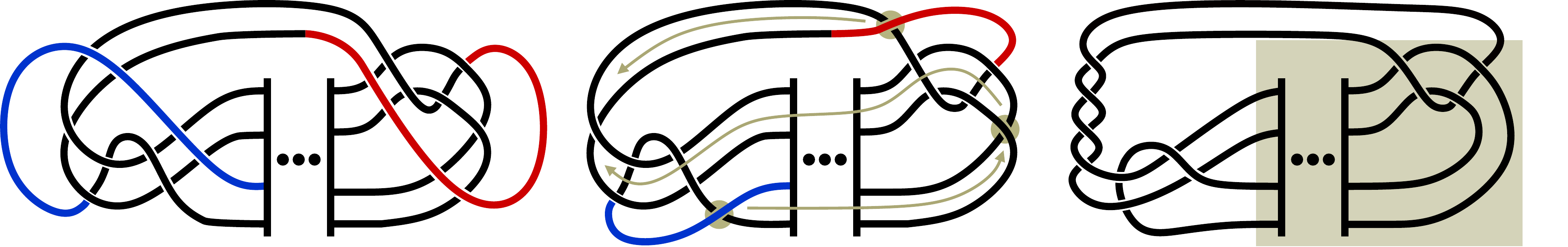}
\caption{The first two steps of our smooth isotopy.}
\label{fig:pretzel-isotopy-1}
\end{figure}

The relevant tangle is seen in panel (A) of \Cref{fig:pretzel-isotopy-2}, where we have highlighted the fact that the vertical ordering of the bands is preserved after $m$ full twists.  This highlighting is reversed in panel (B), where we have moved one half-twist of the bands into view.  Next, we reflect the red arc and then the blue arc to obtain panel (C).  The four crossings revealed in panel (B) have now been reduced to two, and we obtain panel (D) by pushing these crossings along the band.  Moving another half-twist of the bands into view brings us to panel (E), which bears a superficial resemblance to panel (B); notice, however, the crossing data along the vertical arc.  Adding a twist to the vertical arc in panel (E) produces panel (F), and from here we may iterate the process which transforms (B) into (F). The difference in highlighting between panels (B) and (F) does not prevent us from applying these steps. Notice that we have removed four crossings between the bands at the expense of creating three crossings: two in the bottom-left of the tangle, and one in the vertical arc.  By iterating the process, we transform the $8m$ crossings of panel (A) into $4m$ crossings in the bottom-left of the tangle, plus $2m-1$ crossings along the vertical arc. Applying the isotopy of \Cref{fig:pretzel-isotopy-2} to the tangle highlighted in the last panel of \Cref{fig:pretzel-isotopy-1} produces the first panel of \Cref{fig:pretzel-isotopy-3}.

\begin{figure}
\centering
\begin{tikzpicture}
	\node at (0,0) {\includegraphics[scale=0.25]{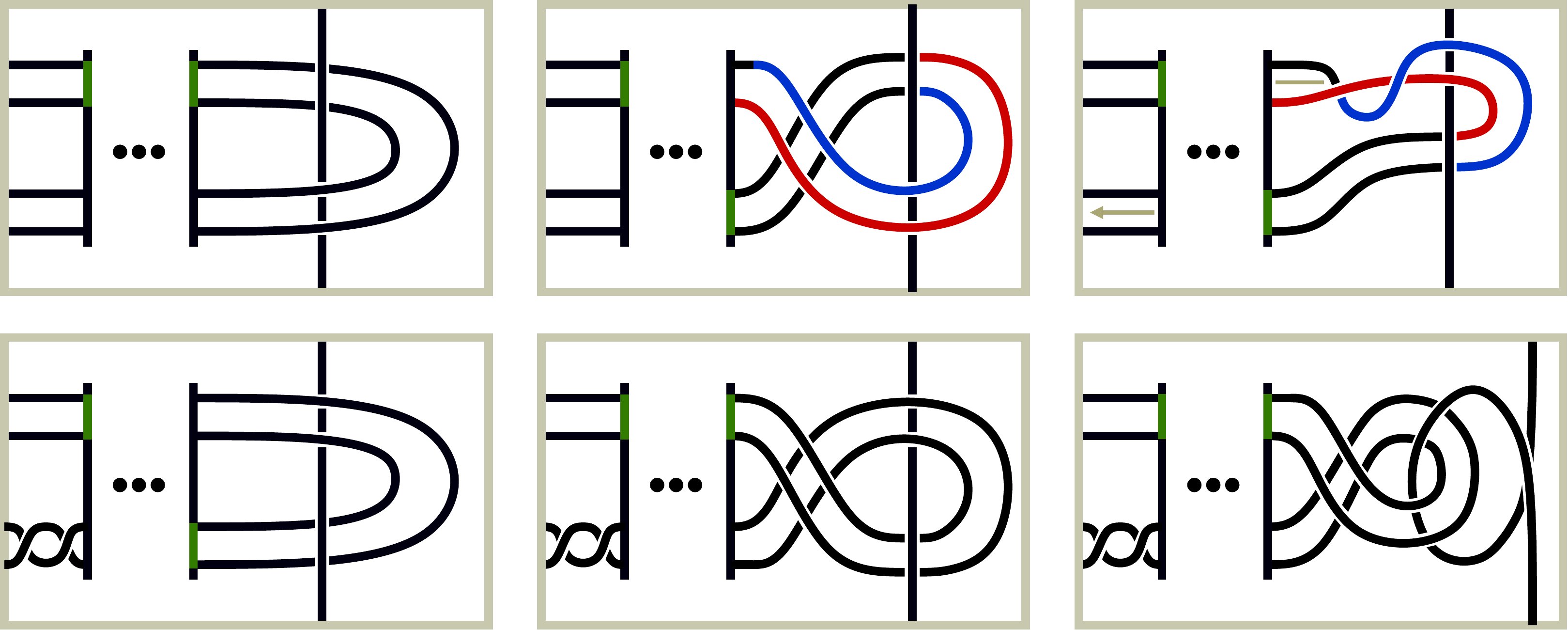}};
	\node at (-5,3.1) {(A)};
	\node at (0,3.1) {(B)};
	\node at (5,3.1) {(C)};
	\node at (-5,-3.1) {(D)};
	\node at (0,-3.1) {(E)};
	\node at (5,-3.1) {(F)};
	\node[font=\footnotesize] at (-5.6, 0.4) {$8m$};
	\node[font=\footnotesize] at (-0.94, 0.4) {$8m-4$};
	\node[font=\footnotesize] at (3.78, 0.4) {$8m-4$};
	\node[font=\footnotesize] at (-5.6, -2.54) {$8m-4$};
	\node[font=\footnotesize] at (-0.94, -2.54) {$8m-8$};
	\node[font=\footnotesize] at (3.78, -2.54) {$8m-8$};
	\end{tikzpicture}
\caption{Unraveling the tangle. The number in each picture represents the number of crossings in the region between the two vertical bars.}
\label{fig:pretzel-isotopy-2}
\end{figure}
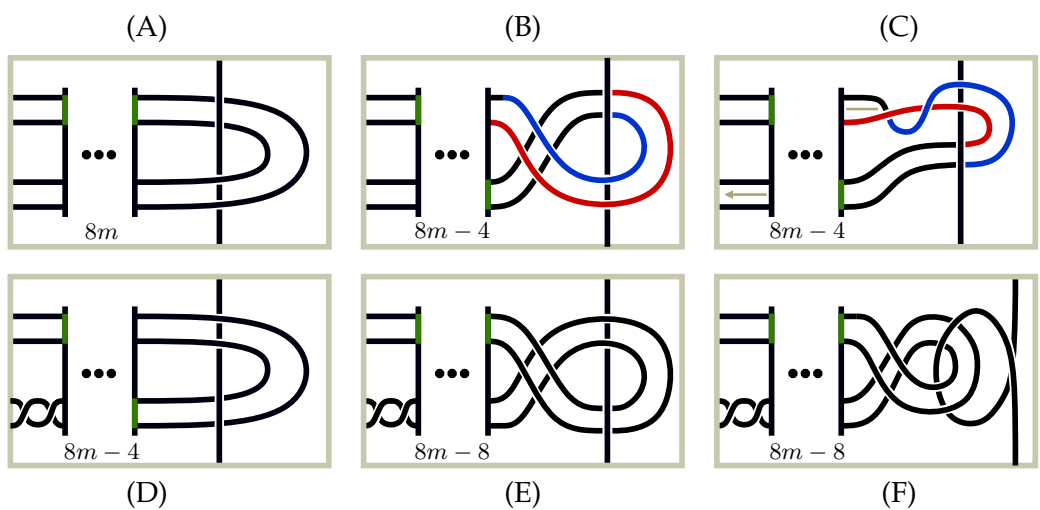

Finally, the $4m$ crossings in the bottom-left of the tangle are pushed along the band, as indicated in the first panel of \Cref{fig:pretzel-isotopy-3}.  The red arc in the second panel of that figure is then reflected, decreasing by one the number of crossings in the diagram.  At last, the highlighted crossing in the third panel of \Cref{fig:pretzel-isotopy-3} is pushed along the indicated arc, producing the diagram for $P(-2m-5,-3,3)$ seen in the final panel.

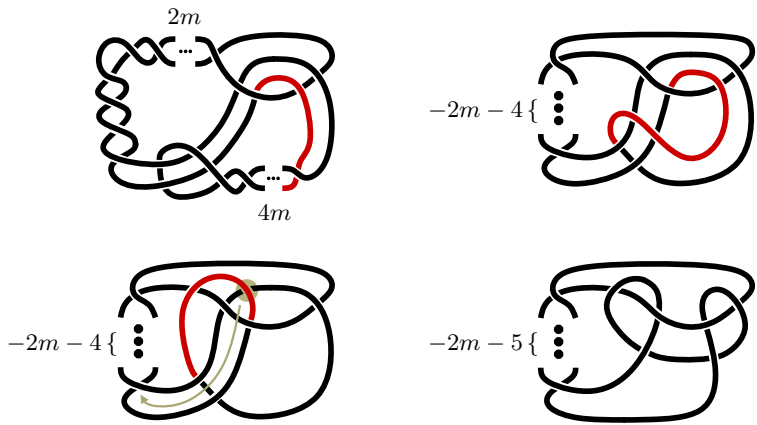
\begin{figure}
\centering
\begin{tikzpicture}
	\node at (0,0) {\includegraphics[scale=0.25]{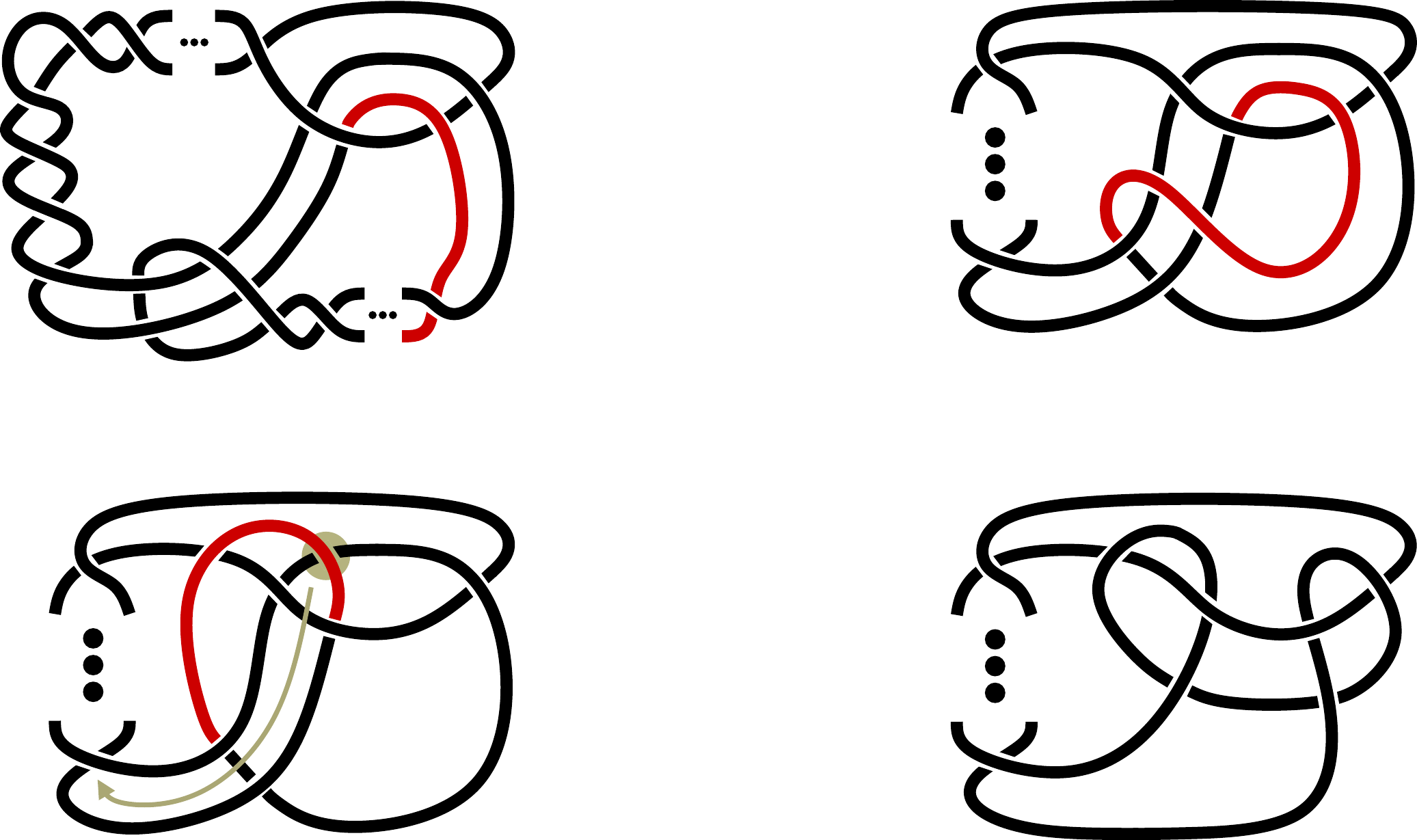}};
	\node[font=\footnotesize] at (-3.2, 2.8) {$2m$};
	\node[font=\footnotesize] at (-2, 0.2) {$4m$};
	\node[font=\footnotesize] at (0.8, 1.55) {$-2m-4\left\{ \right.$};
 	\node[font=\footnotesize] at (-4.8, -1.55) {$-2m-4\left\{ \right.$};
	\node[font=\footnotesize] at (0.8, -1.55) {$-2m-5\left\{ \right.$};
\end{tikzpicture}
\caption{The final steps of our smooth isotopy to $P(-2m-5,-3,3)$.}
\label{fig:pretzel-isotopy-3}
\end{figure}
\end{proof}

From the front projection we see that our Legendrian $P(-2m-5,-3,3)$ has Thurston-Benn\-equin number $-1$ and rotation number $0$.  In \cite{cornwell2016concordance} (c.f. \cite[Appendix A]{etnyre2020legendrian}), Cornwell-Ng-Sivek gave a Legendrian realization of $P(-n,-3,3)$, for each $n\geq 4$, whose Chekanov-Eliashberg DGA is stable tame isomorphic to that of the unknot.  In particular, Cornwell-Ng-Sivek produce a Legendrian $P(-2m-5,-3,3)$ with the same classical invariants as ours, and with ruling polynomial 1.  We conclude, via Rutherford's result cited above, that each of our Legendrian pretzel knots has a unique ruling, and thus that we have infinitely many impostors.

\bibliography{references}

@misc{gowers2011hard,
  author       = {Gowers, Timothy and others},
  title        = {Are there any very hard unknots?},
  howpublished = {MathOverflow},
  year         = {2011},
  url          = {https://mathoverflow.net/questions/53471/are-there-any-very-hard-unknots}
}

@book{stewart2009professor,
  author    = {Stewart, Ian},
  title     = {Professor {S}tewart's {C}abinet of {M}athematical {C}uriosities},
  publisher = {Basic Books},
  year      = {2009},
  pages     = {310},
  isbn      = {9780465013029}
}

@book{adams1994knot,
  title={{The Knot Book: An Elementary Introduction to the Mathematical Theory of Knots}},
  author={Adams, C.C.},
  isbn={9780821886137},
  year={1994},
  publisher={W.H. Freeman and Company}
}

@misc{applebaum2024unknottingnumberhardunknot,
      title={The unknotting number, hard unknot diagrams, and reinforcement learning}, 
      author={Taylor Applebaum and Sam Blackwell and Alex Davies and Thomas Edlich and András Juhász and Marc Lackenby and Nenad Tomašev and Daniel Zheng},
      year={2024},
      note={arXiv:2409.09032} 
}

@article{bennequin1983entrelacements,
  title={Entrelacements et equations de {P}faff},
  author={Bennequin, Daniel},
  journal={Ast\'{e}risque},
  volume={107},
  pages={87--161},
  year={1983}
}

@article{birman1992unlink,
  title={Studying links via closed braids. {V}. {T}he unlink},
  author={Joan S. Birman and William W. Menasco},
  journal={Trans. Amer. Math. Soc. },
  year={1992},
  volume={329},
  pages={585-606},
  url={https://api.semanticscholar.org/CorpusID:15718860}
}

@misc{Breen2024Regularly,
  author = {Breen, Joseph},
  title = {Regularly slice implies once-stably decomposably slice},
  journal = {arXiv preprint},
  note = {arXiv:2410.21031},
  year = {2024}
}

@article{bourgeois2015cobordisms,
author = {Fr{\'e}d{\'e}ric Bourgeois and Joshua M Sabloff and Lisa Traynor},
title = {{Lagrangian cobordisms via generating families: Construction and geography}},
volume = {15},
journal = {Alg. Geom. Topol.},
number = {4},
publisher = {MSP},
pages = {2439--2477},
year = {2015},
doi = {10.2140/agt.2015.15.2439},
URL = {https://doi.org/10.2140/agt.2015.15.2439}
}

@article{burton2024hard,
author = {Benjamin A. Burton and Hsien-Chih Chang and Maarten Löffler and Clément Maria, Arnaud de Mesmay and Saul Schleimer and Eric Sedgwick and Jonathan Spreer},
title = {Hard Diagrams of the Unknot},
journal = {Exp. Math.},
volume = {33},
number = {3},
pages = {482--500},
year = {2024},
publisher = {Taylor \& Francis},
doi = {10.1080/10586458.2022.2161676},


URL = { 
    
        https://doi.org/10.1080/10586458.2022.2161676
    
    

},
eprint = { 
    
        https://doi.org/10.1080/10586458.2022.2161676
    
    

}

}

@article{casals2022infinitely,
author = {Roger Casals and Honghao Gao},
title = {{Infinitely many Lagrangian fillings}},
volume = {195},
journal = {Ann. of Math.},
number = {1},
publisher = {Department of Mathematics of Princeton University},
pages = {207--249},
year = {2022},
doi = {10.4007/annals.2022.195.1.3},
URL = {https://doi.org/10.4007/annals.2022.195.1.3}
}

@article{casals2024steintrace,
author = {Roger Casals and John B. Etnyre and Mark Kegel},
title = {{Stein traces and characterizing slopes}},
volume = {389},
journal = {Math. Ann.},
pages = {1053--1098},
year = {2024},
doi = {10.1007/s00208-023-02662-2},
URL = {https://doi.org/10.1007/s00208-023-02662-2}
}

@article{chantraine2010concordance,
author = {Baptiste Chantraine},
title = {{Lagrangian concordance of Legendrian knots}},
volume = {10},
journal = {Algebr. Geom. Topol.},
number = {1},
publisher = {MSP},
pages = {63--85},
year = {2010},
doi = {10.2140/agt.2010.10.63},
URL = {https://doi.org/10.2140/agt.2010.10.63}
}

@article{chekanov2002dga,
  author = {Yuri Chekanov},
  title = {Differential algebra of {L}egendrian links},
  journal = {Invent. Math.},
  volume = {150},
  number = {3},
  pages = {441--483},
  year = {2002}
}

@article{chekanov2007pushkar,
author = {Chekanov, Yu and Pushkar, P.E.},
year = {2007},
month = {10},
pages = {95--149},
title = {Combinatorics of fronts of {L}egendrian links and the {A}rnol'd 4-conjectures},
volume = {60},
journal = {Russian Math. Surveys},
doi = {10.1070/RM2005v060n01ABEH000808}
}

@article{chongchitmate2013atlas,
author = {Wutichai Chongchitmate and Lenhard Ng},
title = {An Atlas of {L}egendrian Knots},
journal = {Exp. Math.},
volume = {22},
number = {1},
pages = {26--37},
year = {2013},
publisher = {Taylor \& Francis},
doi = {10.1080/10586458.2013.750221},


URL = { 
    
        https://doi.org/10.1080/10586458.2013.750221
    
    

},
eprint = { 
    
        https://doi.org/10.1080/10586458.2013.750221
    
    

}

}

@article{ConwayEtnyreTosun2021Disks,
  author = {Conway, James and Etnyre, John B. and Tosun, B\"{u}lent},
  title = {Symplectic fillings, contact surgeries, and {Lagrangian} disks},
  journal = {Int. Math. Res. Not.},
  volume = {2021},
  number = {8},
  pages = {6020--6050},
  year = {2021},
  doi = {10.1093/imrn/rny291}
}

@article{cornwell2016concordance,
   title={Obstructions to {L}agrangian concordance},
   volume={16},
   ISSN={1472--2747},
   url={http://dx.doi.org/10.2140/agt.2016.16.797},
   DOI={10.2140/agt.2016.16.797},
   number={2},
   journal={Algebr. Geom. Topol.},
   publisher={Mathematical Sciences Publishers},
   author={Cornwell, Christopher and Ng, Lenhard and Sivek, Steven},
   year={2016},
   month=apr, pages={797–824} 
}

@article{ding2001symplectic,
  title={Symplectic fillability of tight contact structures on torus bundles},
  author={Ding, Fan and Geiges, Hansjorg},
  journal={Algebr. {G}eom. {T}opol.},
  volume={1},
  number={1},
  pages={153--172},
  year={2001},
  publisher={Mathematical Sciences Publishers}
}

@article{dynnikov2006arc,
author = {I. A. Dynnikov},
journal = {Fund. Math.},
language = {eng},
number = {1},
pages = {29-76},
title = {Arc-presentations of links: {M}onotonic simplification},
url = {http://eudml.org/doc/283163},
volume = {190},
year = {2006},
}

@article{ekholm2012exactcobordisms,
author = {Ekholm, Tobias and Honda, Ko and Kálmán, Tamás},
year = {2012},
month = {12},
pages = {2627--2689},
title = {Legendrian knots and exact {L}agrangian cobordisms},
volume = {18},
journal = {J. Eur. Math. Soc. (JEMS)},
doi = {10.4171/JEMS/650}
}

@article{eliashberg1987wave,
  title={A theorem on the structure of wave fronts and its application in
  symplectic topology},
  author={Eliashberg, Yakov},
  journal={Funct. Anal. Its. Appl.},
  volume={21},
  pages={227--232},
  year={1987}
}

@article{eliashberg1992contact,
  title={Contact 3-manifolds twenty years since {J}. {M}artinet’s work},
  author={Eliashberg, Yakov},
  journal={Ann. Inst. Fourier (Grenoble)},
  volume={42},
  number={1-2},
  pages={165--192},
  year={1992}
}

@article{eliashberg2018flexiblelagrangians,
    author = {Eliashberg, Yakov and Ganatra, Sheel and Lazarev, Oleg},
    title = "{Flexible Lagrangians}",
    journal = {International Mathematics Research Notices},
    volume = {2020},
    number = {8},
    pages = {2408-2435},
    year = {2018},
    month = {05},
    issn = {1073-7928},
    doi = {10.1093/imrn/rny078},
    url = {https://doi.org/10.1093/imrn/rny078},
    eprint = {https://academic.oup.com/imrn/article-pdf/2020/8/2408/33136983/rny078.pdf},
}

@article{etnyre2001knots,
author = {John B. 
 Etnyre and Ko 
 Honda},
title = {{Knots and Contact Geometry I: Torus Knots and the Figure Eight Knot}},
volume = {1},
journal = {J. Symplectic Geom.},
number = {1},
publisher = {International Press of Boston},
pages = {63 -- 120},
year = {2001},
}

@incollection{etnyre2005surveyknots,
    author = {Etnyre, John B.},
    title = {Legendrian and transversal knots},
    booktitle = {{Handbook of Knot Theory}},
    editor = {Menasco, William and Thistlewaite, Morwen},
    publisher = {Elsevier B.V.},
    year = {2005},
    isbn = {0444514520},
    type = {Survey article},
    pages = {105--186}
}

@article{etnyre2020legendrian,
  author  = {Etnyre, John B. and Ng, Lenhard},
  title   = {{L}egendrian contact homology in {$\mathbb{R}^3$}},
  journal = {Surv. Differ. Geom.},
  volume  = {25},
  pages   = {103--161},
  year    = {2020},
  doi     = {10.48550/arXiv.1811.10966}
}

@article{fuchs2003rulings,
title = {Chekanov–{E}liashberg invariant of {L}egendrian knots: existence of augmentations},
journal = {J. Geom. Phys.},
volume = {47},
number = {1},
pages = {43--65},
year = {2003},
issn = {0393-0440},
doi = {https://doi.org/10.1016/S0393-0440(01)00013-4},
url = {https://www.sciencedirect.com/science/article/pii/S0393044001000134},
author = {Dmitry Fuchs}
}

@article{fuchs2004invariants,
  title={Invariants of {L}egendrian Knots and Decompositions of Front Diagrams},
  author={Dmitry Fuchs and Tigran Ishkhanov},
  journal={Mosc. Math. J.},
  year={2004},
  volume={4},
  number={3},
  pages={707--717},
  url={https://api.semanticscholar.org/CorpusID:17173687}
}

@article{goeritz1934knot,
    author = {Lebrecht Goeritz},
    title = {Bemerkungen zur knotentheorie},
    journal = {Abh. Math. Sem. Univ. Hamburg},
    year = {1934},
    number = {1},
    volume = {10},
    pages = {201--210}
}

@article{gompf1998handlebody,
  title={Handlebody construction of {S}tein surfaces},
  author={Gompf, Robert E},
  journal={Ann. of Math.},
  pages={619--693},
  year={1998},
  publisher={JSTOR}
}

@article{hass2001reidemeister,
  author = {Hass, Joel and Lagarias, Jeffrey C.},
  title = {The Number of {R}eidemeister Moves Needed for Unknotting},
  journal = {J. Amer. Math. Soc.},
  volume = {14},
  number = {2},
  pages = {399--428},
  year = {2001},
  doi = {10.1090/S0894-0347-01-00358-7}
}

@article {hass2010quadratic,
    AUTHOR = {Hass, Joel and Nowik, Tahl},
     TITLE = {Unknot diagrams requiring a quadratic number of {R}eidemeister
              moves to untangle},
   JOURNAL = {Discrete Comput. Geom.},
  FJOURNAL = {Discrete \& Computational Geometry. An International Journal
              of Mathematics and Computer Science},
    VOLUME = {44},
      YEAR = {2010},
    NUMBER = {1},
     PAGES = {91--95},
      ISSN = {0179-5376,1432-0444},
   MRCLASS = {57M25 (68U05)},
  MRNUMBER = {2639820},
MRREVIEWER = {Peiyi\ Zhao},
       DOI = {10.1007/s00454-009-9156-4},
       URL = {https://doi.org/10.1007/s00454-009-9156-4},
}

@article{henrich2014unknotting,
author = {Allison Henrich and Louis H. Kauffman},
title = {Unknotting Unknots},
journal = {Amer. Math. Monthly},
volume = {121},
number = {5},
pages = {379--390},
year = {2014},
publisher = {Taylor \& Francis},
doi = {10.4169/amer.math.monthly.121.05.379},


URL = { 
    
    
        https://www.tandfonline.com/doi/abs/10.4169/amer.math.monthly.121.05.379
    

},
eprint = { 
    
    
        https://www.tandfonline.com/doi/pdf/10.4169/amer.math.monthly.121.05.379
    

}

}

@inbook{kauffman2011collapsing,
author = { Louis H.   Kauffman  and  Sofia   Lambropoulou },
title = {Hard unknots and collapsing tangles},
booktitle = {Introductory Lectures on Knot Theory},
publisher={World Scientific Publishing},
pages = {187--247},
year = {2011},
doi = {10.1142/9789814313001_0009},
URL = {https://www.worldscientific.com/doi/abs/10.1142/9789814313001_0009},
eprint = {https://www.worldscientific.com/doi/pdf/10.1142/9789814313001_0009}
}

@article{lackenby2015polynomial,
  author = {Lackenby, Marc},
  title = {A Polynomial Upper Bound on {R}eidemeister Moves},
  journal = {Ann. of Math.},
  volume = {182},
  number = {2},
  pages = {491--564},
  year = {2015},
  doi = {10.4007/annals.2015.182.2.3}
}

@article{leverson2014augmentations,
  title={Augmentations and Rulings of {L}egendrian Links},
  author={Caitlin Leverson},
  journal={J. Symplectic Geom.},
  year={2014},
  volume={14},
  pages={1089--1143},
  url={https://api.semanticscholar.org/CorpusID:119688269}
}

@misc{lunel2024harddiagramssplitlinks,
      title={Hard diagrams of split links}, 
      author={Corentin Lunel and Arnaud de Mesmay and Jonathan Spreer},
      year={2024},
        note={arXiv:2412.03372},
      eprint={2412.03372},
      archivePrefix={arXiv},
      primaryClass={math.GT},
      url={https://arxiv.org/abs/2412.03372}, 
}

@incollection {mohnke2001legendrian,
    AUTHOR = {Mohnke, Klaus},
     TITLE = {Legendrian links of topological unknots},
 BOOKTITLE = {Topology, geometry, and algebra: interactions and new
              directions ({S}tanford, {CA}, 1999)},
    SERIES = {Contemp. Math.},
    VOLUME = {279},
     PAGES = {209--211},
 PUBLISHER = {Amer. Math. Soc., Providence, RI},
      YEAR = {2001},
      ISBN = {0-8218-2063-X},
   MRCLASS = {57R17 (57M25 57M27 57N37)},
  MRNUMBER = {1850749},
MRREVIEWER = {Ximin\ Liu},
       DOI = {10.1090/conm/279/04562},
       URL = {https://doi.org/10.1090/conm/279/04562},
}

@article {petronio2016unknots,
    AUTHOR = {Petronio, Carlo and Zanellati, Adolfo},
     TITLE = {Algorithmic simplification of knot diagrams: new moves and
              experiments},
   JOURNAL = {J. Knot Theory Ramifications},
  FJOURNAL = {Journal of Knot Theory and its Ramifications},
    VOLUME = {25},
      YEAR = {2016},
    NUMBER = {10},
     PAGES = {1650059, 30},
      ISSN = {0218-2165,1793-6527},
   MRCLASS = {57M25},
  MRNUMBER = {3548477},
       DOI = {10.1142/S0218216516500590},
       URL = {https://doi.org/10.1142/S0218216516500590},
}

@article{rutherford2006thurston,
  title={{The Thurston-Bennequin number, Kauffman polynomial, and ruling invariants of a Legendrian link: The Fuchs conjecture and beyond}},
  author={Rutherford, Dan},
  journal={Int. {M}ath. {R}es. {N}ot. IMRN}, 
  pages ={Art. ID 78591, 15},
  volume={2006},
  number={9},
  year={2006},
  publisher={OUP}
}

@ARTICLE{sabloff2005augmentations,
  author={Sabloff, Joshua M.},
  journal={Int. {M}ath. {R}es. {N}ot. IMRN}, 
  title={Augmentations and rulings of {L}egendrian Knots}, 
  year={2005},
  volume={2005},
  number={19},
  pages={1157--1180},
  doi={10.1155/IMRN.2005.1157}}
\bibliographystyle{amsalpha}

\end{document}